\documentclass[11pt]{article}
	
	\usepackage{algorithm}
	\usepackage{algpseudocode}
    \usepackage{url}
    \usepackage{verbatim}
    \usepackage[titletoc]{appendix}
    \usepackage{graphicx}
	\usepackage{amsmath}
	\usepackage{amsthm}
	\usepackage{amsfonts}
	\usepackage{graphicx}
    \usepackage{nicefrac}
    \usepackage{longtable}
    \usepackage{color}
    \usepackage{graphicx, amssymb,graphics}
 	\usepackage{subfigure}
 	\usepackage{epstopdf}
 	\usepackage{ulem}

\newtheorem{rem}{Remark}[section]
\newtheorem{theorem}{Theorem}[section]

\newtheorem{lem}{Lemma}[section]

\newtheorem{prop}{Proposition}[section]

\newcommand{\n}{\mbox{\boldmath$n$}}

\newcommand{\diff}{\mathrm{d}}
\newcommand{\dS}{\diff S}
\newcommand{\hf}{\nicefrac{1}{2}}
\newcommand{\nrm}[1]{\left\| #1 \right\|}
\newcommand{\ciptwo}[2]{( #1 , #2 )}

\newcommand{\dx}{\mbox{d}\textbf{x}}
\newcommand\dt {{\Delta t}}

\newcommand{\eipx}[2]{\left( #1 , #2 \right)_{\rm x}}
\newcommand{\eipy}[2]{\left( #1 , #2 \right)_{\rm y}}

\newcommand{\eipvec}[2]{\left( #1 , #2 \right)}

\newcommand{\mC}{\mathcal C}

\newcommand{\mL}{\mathcal L}
\newcommand{\mG}{\mathcal G}
\newcommand{\mV}{\mathcal V}
\newcommand{\mP}{\mathcal P}

\newcommand{\mN}{\mathcal N}

\newcommand{\bU}{\boldsymbol{U}}

\allowdisplaybreaks[4]

	\newcommand\be {\begin{equation}}
	\newcommand\ee {\end{equation}}
	\newcommand\bu {\boldsymbol{u}}
    \newcommand\bv {\boldsymbol{v}}
    \newcommand\bw {\boldsymbol{w}}

	\title{Convergence analysis of 
a numerical scheme for the Cahn-Hilliard-Navier-Stokes system with dynamical boundary condition and its application to moving contact line problem}

	\date{\today}

\begin{document}
	
\author{
Yunzhuo Guo \thanks{Department of Applied Mathematics, Hong Kong Polytechnic University, Hung Hom, Hong Kong (yunzguo@polyu.edu.hk)}
\and
Cheng Wang\thanks{Department of Mathematics, The University of Massachusetts, North Dartmouth, MA  02747, USA (cwang1@umassd.edu)}
\and
Xianmin Xu \thanks{LSEC, Academy of Mathematics and Systems Science, Chinese Academy of Sciences, 100190, Beijing, China (xmxu@lsec.cc.ac.cn)}
\and
Zhengru Zhang\thanks{Laboratory of Mathematics and Complex Systems, Beijing Normal University, Beijing 100875, P.R. China (zrzhang@bnu.edu.cn)}
}

 	\maketitle
	\numberwithin{equation}{section}
	
\begin{abstract}	
A finite difference numerical scheme is proposed and analyzed for the Cahn–Hilliard–Navier–
Stokes system, combined with a dynamical boundary condition. Such a physical system has potential applications in the moving contact line problem. The boundary profile is governed by a lower-dimensional energy potential, coupled with a non-homogeneous boundary condition for the phase variable. In the numerical design, a convex-splitting approach is applied to the chemical potential in both bulk and surface levels, which leads to a highly coupled nonlinear system. A semi-implicit discretization is taken in the nonlinear fluid convection, as well as the coupled terms between the fluid motion and phase variable evolution. A careful finite difference approximation and convexity analysis reveals that such a numerical system could be represented as a non-symmetric and monotone mapping associated with the fluid convection. In turn, the unique solvability is valid based on the monotonicity argument. The total energy stability analysis is obtained through a careful summation-by-part calculation. In particular, an optimal rate convergence analysis is theoretically established in this work. The discrete mass conservation of the exact solution is required to preserve the mean-zero property of the error function, so that the associated discrete $H_h^{-1}$ norm is well-defined. A combination of the Fourier projection and an auxiliary function is applied to overcome this difficulty. Furthermore, an approach of rough and refined error estimates concludes the desired convergence result. Some numerical results are presented in this article, which demonstrate the robustness of the proposed numerical scheme. In our knowledge, this work provides a theoretical proof of convergence analysis and error estimate for a numerical scheme to the moving contact line problem, for the first time in the literature.



\bigskip


\end{abstract}
	


\section{Introduction} 
The Cahn-Hilliard-Navier-Stokes (CHNS) system \cite{liu03} is a well-known phase field model coupling the Navier-Stokes (NS) equation for the incompressible fluid motion and the Cahn-Hilliard (CH) equation for the phase field. 
The standard Neumann boundary condition for the phase variable and chemical potential has been widely considered, due to the convenience of analysis and calculation \cite{chen24b, chen22b, diegel17, feng06, GuoY24, han15, kay07}. On the other hand, some scientists have included the surface energy in the description of the short-distance interaction near the boundary, and the dynamical boundary condition has been designed. Let $\Omega$ be a bounded domain with Lipchitz continuous boundary $\Gamma=\partial \Omega$. The bulk and surface free energies are given by
\begin{align} 
& E_{\rm bulk} (\phi) = \int_\Omega \Big( \frac{\varepsilon^2}{2} | \nabla \phi |^2 + F (\phi) \Big) \dx , \quad E_{\rm surf} (\psi) = \int_\Gamma \,\Big( \frac{\kappa\varepsilon}{2} | \nabla_\Gamma \psi |^2 + G (\psi) \Big) \,  \dS, \label{energy-CH-1} \\
& E_{\rm total}=E_{\rm bulk}+E_{\rm surf}+\frac{1}{2\gamma}\|\bu\|_2^2.
\label{nonlinear terms}
\end{align} 
Here $F (\cdot)$ and $G(\cdot)$ are the energy density functions in bulk and on the boundary, respectively. The descent of the total free energy drives the system evolution. On the boundary, system is controlled by a lower dimensional phase field equation and leads to different types of dynamical boundary condition models. 
The Gal model, introduced by Gal et al. \cite{Gal06}, provides an early formulation of dynamical boundary conditions for the Cahn-Hilliard equation. The surface mass fraction dynamics follows an Allen-Cahn equation with an added flux, a feature which leads to non-conservation of both the total system mass and the boundary mass. In contrast, Liu and Wu \cite{LiuC2019} designed a conserved model that assume no mass exchange between boundary and interior, so that the mass conservation is satisfied at both levels. Moreover, Goldstein et al. \cite{Goldstein11} revised the bulk-surface transport mechanism. Their GMS model ensures equal mass exchange between the surface and the bulk, guaranteeing total mass conservation. Besides the above three basic boundary conditions, more general models have been raised. Building on the Liu–Wu model, Knopf and Lam proposed a further set of dynamical boundary conditions for the Cahn–Hilliard equation, known as the KL model \cite{Knopf_2020}. A later article \cite{Knopf2021} established a broader framework, the KLLM model, which subsumes both the Liu-Wu and GMS models in certain limits. As an effective tool for physical modeling, the Onsager principle is applied in thermodynamically consistent boundary conditions \cite{Jing2023} and a non-isothermal model with certain modifications.

In this work, we consider the CHNS system with Allen-Cahn type dynamical boundary condition and $\kappa =0$ in \eqref{energy-CH-1}, which could termed as relaxed boundary condition. The model is closely related to the evolution of an interface in contact with the solid boundary, i.e., the well-known moving contact line problem \cite{PRL766}. The equations are formulated as
\begin{align}
& \bu_t + \bu\cdot\nabla\bu+\nabla p-\nu\Delta\bu=-\gamma\phi\nabla\mu, \label{CHNS-DBC-1} \\
& \phi_t+\nabla\cdot(\phi\bu)=\Delta \mu, \quad \mu = F'(\phi)
-\varepsilon^2\Delta\phi,  \label{CHNS-DBC-2} \\
& \nabla \cdot \bu=0, \label{CHNS-DBC-3} \\
& \partial_n \mu = 0, \quad \phi|_\Gamma = \psi, \quad \text{ on $\Gamma$,}  \label{CHNS-DBC-4} \\
& \psi_t= -\mu_\Gamma, \quad \mu_\Gamma = G'(\psi) + \varepsilon^2 \partial_{n} \phi,  \quad \text{ on $\Gamma$,}  \label{CHNS-DBC-5} \\
& \bu=\boldsymbol{0},  \quad \text{ on $\Gamma$,}\label{CHNS-DBC-6}
\end{align}
where $\gamma > 0$ is related to surface tension, $\nu$ stands for the kinematic viscosity, $\varepsilon$ is an bulk diffuse interface thickness parameter, $p$ is the pressure, $\mu_\Gamma$ stands for the surface chemical potential. 
The normal derivative is the bridge connecting the boundary and the bulk, which leads to the total energy dissipation: $\frac{d}{dt}E_{\rm total}(\phi) = -\|\nabla \mu\|_2^2 - \|\mu_\Gamma\|_2^2-\frac{\nu}{\gamma}\|\nabla \bu\|_2^2 \leq 0$. We remark that the relaxed boundary condition \eqref{CHNS-DBC-5} is an essential component of the general Navier boundary condition (GNBC) for the moving contact line \cite{QianWangSheng03}. Here, instead of the Navier slip-type boundary conditions for velocity, we use the no-slip boundary condition for simplicity, as has been used in other phase-field models of the moving contact line problem (c.f. \cite{Jacqmin2000}). It is also noticed that the no-penetration, no-slip boundary for the velocity profile $\bu$ indicates that the interaction between the motion and the surface was omitted. The tangential velocity barely affects the system evolution near the boundary.

Designing a numerical scheme to preserve the total energy stability turns out to be a very challenging issue, due to the nonlinear and coupled nature of the PDE system. For the square domain, the normal derivative is not well-defined at the vertexes. This fundamental issue may affect the numerical accuracy and the theoretical analysis. 
To avoid this challenging issue, many studies have focused on the structure-preserving temporal semi-discrete schemes, such as first and second order schemes for Liu-Wu model \cite{Bao2021a,Meng23} and KLLM model \cite{Bao2021b}. On the other hand, a fully discrete algorithm requires certain modifications. In \cite{Cherfils_10,CH-DBC-2024}, the physical boundary was applied at $y$-direction with $x$-direction remaining periodic, or mathematically, a quotient domain that 
\begin{equation*} 
\begin{aligned} 
  & 
\Omega=\Pi_{i=1}^{d-1}\left(\mathbb{R} /\left(L_i \mathbb{Z}\right)\right) \times\left(0, L_d\right), \quad L_i>0, i=1, \ldots, d, \quad d=2 \text { or } 3, 
\\
  & \mbox{with a boundary} \quad 
\Gamma=\partial \Omega=\Pi_{i=1}^{d-1}\left(\mathbb{R} /\left(L_i \mathbb{Z}\right)\right) \times\left\{0, L_d\right\}.
\end{aligned} 
\end{equation*} 
A fully discrete, multi-scale convex splitting method was proposed in \cite{CH-DBC-2024} for the pure phase field part, and the convergence analysis was established in a later work!~\cite{Guo_2025_convergence_CHDBC}. 
A scalar auxiliary variable (SAV) finite element simulation was also reported in \cite{Liu2024}. Besides, a Cahn-Hilliard-Hele-Shaw (CHHS) system with Liu-Wu model was considered in \cite{Yao2022} where the tangential velocity is involved in the interaction on the boundary. The difference is that the CHHS model is static and does not take into account the kinetic energy, which means the phase variable could uniquely determine the velocity and no temporal derivative is involved. 

In this work, we follow the idea of partial dynamical boundary in \cite{Cherfils_10,CH-DBC-2024} and carry out a theoretical analysis. The numerical approximation to the chemical potential profiles, at both the interior region and on the boundary section, is based on the convex-concave decomposition of the double-well energy functional, which dates back to Eyre \cite{eyre98}. This approach has been widely used in the phase field model to reach an unconditional energy dissipation \cite{baskaran13a, baskaran13b, chen16, cheng2019d} and the positivity-preserving property for singular energy models \cite{chen19b, dong18a, GuoY24}. In particular, the nonlinear term is implicitly treated and the linear expansive term is explicitly updated. The discrete boundary condition for the phase variable, at the next time step, is coupled with the evolutionary equation of the boundary profile. The normal derivative is discretized by a wide stencil to maintain the second order spatial accuracy and further leads to an unconditional total energy dissipation, combined with a summation-by-part formula. 

The fluid momentum equation is solved using a projection method. First, an intermediate velocity field is constructed via a semi-implicit algorithm: the kinematic diffusion term is treated implicitly, while the pressure gradient is updated explicitly. Both the fluid convection term and the phase field coupled term are approximated in a semi-implicit manner. Given a fixed chemical potential profile, the intermediate velocity is computed using a linear convection-diffusion solver. Afterwards, a Helmholtz projection onto a divergence-free vector field is applied to obtain the velocity and pressure at the subsequent time step \cite{Guermond-20066011,Liu_Shen_15}. 

The unique solvability analysis for the proposed numerical scheme is highly non-standard. The nonlinear convective term in the momentum equation leads to a non-symmetric nature of the whole numerical system, which means the solution could not be represented as the minimization of a convex functional. The Browder-Minty lemma \cite{Browder-Minty-1,Browder-Minty-2} is a helpful tool to overcome this subtle difficulty, based on the monotonic property of the numerical system. The energy stability of the numerical scheme turns out to be a direct consequence of a careful energy estimate, which gives a dissipation law for the discrete total energy functional. The surface inner product for the phase variable is connected to the boundary evolution equation by the normal derivative, with the help of summation-by-parts formulas. The staggered location of the fluid velocity vector and the phase variable plays an important role in the theoretical derivation. 

Meanwhile, it is noticed that, a theoretical justification of the convergence analysis and error estimate for the moving contact line has been an open problem for many years, because of the highly nonlinear nature, as well as the non-trivial coupling between the bulk and the boundary section, although there exist some works of efficient energy stable numerical schemes \cite{GaoWang12,GaoWang14,ShenYangYu15,YangYu18}. In this work, we provide an optimal rate convergence analysis for the proposed numerical scheme, due to its connection with the moving contact line. As commonly observed in the phase field-fluid coupled system, the standard $\ell^\infty(0,T;\ell^2) \cap \ell^2(0,T;H_h^2)$ error estimate for the phase variable could not pass through, because of lack of regularity to control the error inner products associated with the nonlinear coupled terms in the phase evolutionary and the momentum equations. To overcome this difficulty, the energy norms, $\ell^\infty(0,T; H_h^1)$ for the phase variable and $\ell^\infty(0,T;  \ell^2)$ for the velocity vector in the bulk, as well as the $\ell^\infty (0, T; \ell^2)$ for the phase variable on the boundary section, have to be followed in the convergence analysis. Moreover, a combination of rough and a refined error estimates has to be applied to accomplish the analysis. In more details, the rough error estimate gives a maximum norm bound of the numerical solution, at both the bulk and the boundary, and the refined error estimate analyzes the nonlinear term and concludes the convergence analysis. Besides, a combination of $x$-directional Fourier projection and $y$-directional trigonometric auxiliary function replaces the standard Fourier projection since the dynamical boundary. This approach achieves the discrete mass conservation of the exact solution and makes the discrete Poincar\'e inequality  valid; also see the related analysis \cite{Guo_2025_convergence_CHDBC}. 

The rest of this article is organized as follows. In Section \ref{sec:numerical scheme}, we review the finite difference spatial approximation and propose the fully discrete numerical scheme. The unique solvability analysis is derived as well. A modified total energy stability is proved in Section \ref{sec:energy}. The detailed convergence analysis is provided in Section \ref{sec:convergence}, including the rough and refined error estimates. Some numerical simulation results are presented in Section \ref{sec:simulations} to verify the theoretical analysis. Finally, some concluding remarks are made in Section \ref{sec:conclusion}.

\section{Numerical scheme} \label{sec:numerical scheme}
\subsection{Finite difference spatial discretization}
The standard finite difference spatial approximation is applied. We present the numerical approximation over a two-dimensional computational domain $\Omega = (0,1)^2$. The three-dimensional extension is straightforward and omitted for brevity. As mentioned above, a periodic boundary condition is assumed in the $x$-direction and physical boundary condition is imposed at the top and bottom sections, namely, at 
\begin{equation} 
\Gamma_{B} := \left\{ (x,y) \ \middle| \  0\le x\le 1, \  y=0\right\}  ,\quad \Gamma_{T}:= \left\{ (x,y) \ \middle| \ 0\le x\le 1,\  y=1 \right\}. 
	\label{boundary-Gamma-1} 
\end{equation} 
In addition, a uniform spatial mesh size, $\Delta x = \Delta y = h = \frac{1}{N}$ with $N \in\mathbb{N}$, is taken. We define  
	\[
\mC_{\mathrm{p},x}(\Omega) := \left\{ f_{i,j}  \  \middle| \ f_{i,j} = f_{i+\alpha N,j }, \ \forall \, i,\alpha \in \mathbb{Z}, \  j=0,\cdots, N \right\},
	\]
a set of grid functions with discrete periodic boundary conditions imposed in the $x$-direction. In particular, the subscripted symbols ${\mathrm{p},x}$ indicate that periodic boundary conditions are imposed in only the $x$-direction, throughout this paper. Herein the notation $f_{i,j}$ represents  the numerical value of $f\in \mC_{\mathrm{p},x}(\Omega)$ at the real-space point $( p_i, p_j)\in \mathbb{R}^2$, where $p_i:=i\cdot h$. 

The Neumann boundary condition in the $y$-direction requires an additional ghost layer at the top and bottom sections of the domain. The grid function space with ghost layers is defined as
	\[
\mC_{\mathrm{p},x}^{+}(\Omega) := \left\{ f_{i,j}  \  \middle| \ f_{i,j} = f_{i+\alpha N,j }, \ \forall \, i,\alpha \in \mathbb{Z}, \  j=-1,\cdots, N+1 \right\}.
	\]
Analogously, more grid function spaces are introduced: 
\begin{align*}
    & {\mathcal E}_{\mathrm{p},x}(\Omega) := \left\{ f_{i+\hf,j}  \  \middle| \ f_{i+\hf,j} = f_{i+\hf+\alpha N,j }, \ \forall \, i,\alpha \in \mathbb{Z}, \  j=0,\cdots, N \right\}, \\
    & {\mathcal E}_{\mathrm{p},x}^{+}(\Omega) := \left\{ f_{i+\hf,j}  \  \middle| \ f_{i+\hf,j} = f_{i+\hf+\alpha N,j }, \ \forall \, i,\alpha \in \mathbb{Z}, \  j=-1,\cdots, N+1 \right\}, \\
    & {\mathcal N}_{\mathrm{p},x}(\Omega) := \left\{ f_{i,j+\hf}  \  \middle| \ f_{i,j+\hf} = f_{i+\alpha N,j +\hf}, \ \forall \, i,\alpha \in \mathbb{Z}, \  j=0,\cdots, N-1 \right\}, \\
    & {\mathcal N}_{\mathrm{p},x}^{+}(\Omega) := \left\{ f_{i,j+\hf}  \  \middle| \ f_{i,j+\hf} = f_{i+\alpha N,j +\hf}, \ \forall \, i,\alpha \in \mathbb{Z}, \  j=-1,\cdots, N \right\},\\
    & {\mathcal G}_{\mathrm{p},x}(\Omega) := \left\{ f_{i+\hf,j+\hf}  \  \middle| \ f_{i+\hf,j+\hf} = f_{i+\hf+\alpha N,j +\hf}, \ \forall \, i,\alpha \in \mathbb{Z}, \  j=0,\cdots, N-1 \right\}, \\
    & {\mathcal G}_{\mathrm{p},x}^{+}(\Omega) := \left\{ f_{i+\hf,j+\hf}  \  \middle| \ f_{i+\hf,j+\hf} = f_{i+\hf+\alpha N,j +\hf}, \ \forall \, i,\alpha \in \mathbb{Z}, \  j=-1,\cdots, N \right\}. 
\end{align*}
The discrete average and difference operators $A_x,D_x: \mC_{\mathrm{p},x}(\Omega) \to {\mathcal E}_{\mathrm{p},x}(\Omega)$, $A_y,D_y: \mC_{\mathrm{p},x}^{+}(\Omega) \to {\mathcal N}_{\mathrm{p},x}^{+}(\Omega)$, are given by 
	\begin{align} 
	  & 
 A_x f_{i+\hf,j} := \frac{1}{2}\left(f_{i+1,j} + f_{i,j} \right), \quad D_x f_{i+\hf,j} := \frac{1}{h}\left(f_{i+1,j} - f_{i,j} \right), \quad \forall f\in \mC_{\mathrm{p},x}(\Omega) , 
\\
  & 
A_y f_{i,j+\hf} := \frac{1}{2}\left(f_{i,j+1} + f_{i,j} \right), \quad D_y f_{i,j+\hf} := \frac{1}{h}\left(f_{i,j+1} - f_{i,j} \right),  \quad \forall f\in \mC_{\mathrm{p},x}^{+}(\Omega) . 
	\end{align} 
Similarly, the discrete average and difference operators $a_x,d_x: {\mathcal E}_{\mathrm{p},x}(\Omega) \to \mC_{\mathrm{p},x}(\Omega)$, $a_y,d_y: {\mathcal N}_{\mathrm{p},x}^{+}(\Omega) \to \mC_{\mathrm{p},x}(\Omega)$, are defined as 
	\begin{align}
	  & 
a_x f_{i, j} := \frac{1}{2}\left(f_{i+\hf, j} + f_{i-\hf, j} \right),	 \quad d_x f_{i, j} := \frac{1}{h}\left(f_{i+\hf, j} - f_{i-\hf, j} \right), \quad \forall f \in {\mathcal E}_{\mathrm{p},x}(\Omega),
\\
  & 
a_y g_{i,j} := \frac{1}{2}\left(g_{i,j+\hf} + g_{i,j-\hf} \right),	 \quad d_y g_{i,j} := \frac{1}{h}\left(g_{i,j+\hf} - g_{i,j-\hf} \right) , \quad \forall g \in {\mathcal N}_{\mathrm{p},x}^{+}(\Omega).
\end{align}

For a scalar grid function $g\in \mC_{\mathrm{p},x}^{+}(\Omega)$ and a vector function $\boldsymbol{f} = ( f^x , f^y)^T$, with $f^x\in {\mathcal E}_{\mathrm{p},x}(\Omega) $ and $f^y\in {\mathcal N}_{\mathrm{p},x}^+(\Omega)$, the discrete divergence turns out to be   
	\begin{equation} 
\nabla_h\cdot \big( g \boldsymbol{f} \big)_{i,j} = d_x\left( A_x g \, f^x\right)_{i,j}  + d_y\left( A_y g \, f^y\right)_{i,j}  , \quad \nabla_h\cdot \big( g \boldsymbol{f} \big) \in \mC_{\mathrm{p},x}(\Omega) .  
	\label{divergence-1} 
	\end{equation} 
For any $\phi \in \mC_{\mathrm{p},x}^{+}(\Omega)$, the discrete gradient is introduced as $\nabla_h \phi := ( D_x \phi , D_y \phi)^T$, with $D_x \phi\in {\mathcal E}_{\mathrm{p},x}(\Omega)$ and $D_y \phi\in {\mathcal N}_{\mathrm{p},x}^{+}(\Omega)$. For $g\in \mC_{\mathrm{p},x}^{+}(\Omega)$, the following term is defined: 
	\begin{equation} 
\nabla_h\cdot \big( g \nabla_h \phi \big)_{i,j} =  d_x\left( A_x g \, D_x \phi \right)_{i,j}  + d_y\left( A_y g \, D_y \phi \right)_{i,j} , \quad \nabla_h\cdot \big( g \nabla_h \phi \big) \in \mC_{\mathrm{p},x}(\Omega) . 
	\label{divergence-2}  
	\end{equation}
If $g\equiv 1$, the discrete Laplacian operator becomes the standard 5-point stencil. 

For two grid functions $f, g \in \mC_{\mathrm{p},x}(\Omega)$, the discrete $\ell^2$ inner product and the associated norm are defined as 
	\begin{equation}
 (f,g):= h^2 \sum_{i=0}^{N-1} \sum_{j=0}^N \, w_j f_{i,j} g_{i,j} , \quad w_j = \left\{ 
 	\begin{array}{l} 
1 , \, \, \, 1 \le j \le N-1 , 
   \\ 
\frac12 , \, \, \, j=0, N , 
	\end{array} 
\right.  \quad  \| f \|_{2} := \sqrt{ (f,f) } .  
  \label{inner product-1} 
	\end{equation} 
The mean-zero grid function space is introduced as $\mathring\mC_{\mathrm{p},x}(\Omega):=\left\{ f \in \mC_{\mathrm{p},x}(\Omega) \ \middle| \ 0 = \overline{f} :=  \frac{1}{| \Omega|} (f,{\bf 1}) \right\}$, where $|\Omega| = 1$ is the area of $\Omega$. Similarly, for two vector grid functions $\boldsymbol{f} = ( f^x , f^y)^T$ and  $\boldsymbol{g} = ( g^x , g^y )^T$, with $f^x, g^x \in {\mathcal E}_{\mathrm{p},x}(\Omega)$ and $f^y, g^y\in {\mathcal N}_{\mathrm{p},x}(\Omega)$, the corresponding discrete inner product is defined as follows, with the same notation as the scaler case: 
	\begin{equation}
\eipvec{\boldsymbol{f} }{\boldsymbol{g} } : = \eipx{f^x}{g^x} + \eipy{g^y}{g^y} ,  \, \,  \eipx{f^x}{g^x} := (a_x (f^x g^x),1) , \, \,   \eipy{f^y}{g^y} :=  h^2 \sum_{i,j=0}^{N-1} \, f^y_{i,j+\hf} g^y_{i,j+\hf}  . 
     \label{inner product-2} 
	\end{equation}	
In addition to the discrete $\| \, \cdot \, \|_{2}$ norm, the discrete $\ell^p$ and maximum norms turn out to be  
	\[
\| f \|_p := ( | f |^p , 1 )^\frac1p , \, \, \, \nrm{f}_\infty := \max_{i,j}\left| f_{i,j}\right|, 
\quad \forall f\in \mC_{\mathrm{p},x}(\Omega) , \, \, 1 \le p < + \infty .
	\]

For a grid function $f\in \mC_{\mathrm{p},x}^{+}(\Omega)$, the differential derivatives with respect to the second variable, that is, $y$ or $j$, have to be evaluated at the physical  boundaries $\Gamma_{T}$ and $\Gamma_{B}$. To accomplish this, the following operators are used for any $f\in \mC_{\mathrm{p},x}^{+}(\Omega)$: $\tilde{D}_y f_{i,0} := \frac{f_{i,1}-f_{i,-1}}{2h},  \quad 	\tilde{D}_y f_{i,N} :=  \frac{f_{i,N+1}-f_{i,N-1}}{2h}$. As a consequence, the homogeneous Neumann boundary condition is enforced as follows: 
\begin{align} 
  & 
0 = \tilde{D}_y f_{i,0} = \frac{f_{i,1}-f_{i,-1}}{2h} \quad \iff \quad f_{i, - 1}= f_{i,1} ,  \quad 
\forall f\in \mC_{\mathrm{p},x}^{+}(\Omega) , \label{Neumann 1}
\\
  & 
0 =  \tilde{D}_y f_{i,N} = \frac{f_{i,N+1}-f_{i,N-1}}{2h} \quad \iff \quad f_{i,N+1}= f_{i,N-1} , \quad 
\forall f\in \mC_{\mathrm{p},x}^{+}(\Omega) .  \label{Neumann 2}
\end{align}

On the top and bottom boundary sections $\Gamma_{T}$, $\Gamma_{B}$, the following one-dimensional periodic grid function spaces are utilized
	\[
\mC_{\mathrm{p},x}(\Gamma_T) = \mC_{\mathrm{p},x}(\Gamma_B) = \mC_{\mathrm{p},x}(\Gamma) := \left\{ \varphi_{i}  \ \middle| \ \varphi_{i} = \varphi_{i+\alpha N}, \ \forall \, \alpha,i \in \mathbb{Z} \right\}.	
	\]
The corresponding inner product and $\ell^2$ norm on the  boundary become 
	\begin{equation*}
(f,g)_\Gamma := h \sum_{i=0}^{N-1} f_{i} g_{i} , \quad \nrm{ f }_{2, \Gamma} := \sqrt{ (f,g)_\Gamma}  , 
  \quad  \forall  f, g \in \mC_{\mathrm{p},x}(\Gamma) . 
	\end{equation*} 
Similarly, the discrete $\ell^p$ and maximum norms on the boundary section are given by  
	\begin{equation*} 
\| f \|_{p, \Gamma} := ( | f |^p , 1 )_\Gamma^\frac1p , \, \, \, \nrm{f}_{\infty, \Gamma} := \max_{i} \left| f_i \right|, 
\quad \forall f\in \mC_{\mathrm{p},x}(\Gamma) , \, \, 1 \le p < + \infty .
	\end{equation*} 
The operators $A_x,D_x$ act in a natural way on the space $\mC_{\mathrm{p},x}(\Gamma)$ and we suppress the formulas.

The following summation-by-parts formulas could be easily proved using the techniques found in many existing works and elsewhere.

	\begin{lem}  \cite{guo16, wise09a} 
	\label{lemma1} 	
For any $\psi, \phi, g \in \mC_{\mathrm{p},x}^{+}(\Omega)$, and any $\boldsymbol{f} = (f^x, f^y)^T$, with $f^x\in \mathcal{E}_{\mathrm{p},x}(\Omega)$ and $f^y\in\mathcal{N}_{\mathrm{p},x}^+(\Omega)$, the following summation by parts formulas are valid: 
\begin{align} 
\ciptwo{\psi}{\nabla_h\cdot\boldsymbol{f}} & =  - \eipvec{\nabla_h \psi}{ \boldsymbol{f}} + ( a_yf_{\cdot , N}^y , \psi_{\cdot , N} )_\Gamma - ( a_yf_{\cdot , 0}^y , \psi_{\cdot , 0} )_\Gamma , 
	\label{lemma 1-0-1}  
	\\
\ciptwo{\psi}{\nabla_h\cdot \left( g \nabla_h \phi \right)}  & =  -  \eipvec{\nabla_h \psi }{ {\cal A}_h g \nabla_h\phi} +  ( a_y(A_yg D_y\phi)_{\cdot , N},\psi_{\cdot , N} )_\Gamma  
  -  ( a_y(A_yg D_y\phi)_{\cdot , 0},\psi_{\cdot , 0} )_\Gamma ,
	\label{lemma 1-0-2}
	\end{align}
where we have used the notation: 
$\eipvec{\nabla_h \psi }{ {\cal A}_h g \nabla_h\phi} := \eipx{D_x\psi}{A_x g D_x\phi} + \eipy{D_y\psi}{A_y gD_y\phi}$.
In particular, if $g \equiv 1$, the following identity is available: 
\begin{align} 
\ciptwo{\psi}{\Delta_h \phi} = -  \eipvec{\nabla_h \psi }{ \nabla_h\phi} +  ( \tilde{D}_y\phi_{\cdot , N},\psi_{\cdot , N} )_\Gamma  -  ( \tilde{D}_y\phi_{\cdot , 0} , \psi_{\cdot , 0} )_\Gamma .
	\label{lemma 1-0-3} 
\end{align}
\end{lem} 
	
Next, an important positive, linear operator will be needed in the later analysis. First observe that for any $\varphi \in\mathring\mC_{\mathrm{p},x}(\Omega)$, there is a unique solution $\psi \in \mathring\mC_{\mathrm{p},x}^+(\Omega)$ to the problem 
$-\Delta_h \psi = \varphi$,	subject to the homogeneous Neumann boundary condition, \eqref{Neumann 1} and \eqref{Neumann 2}. The solution operator for this problem defines a one-to-one mapping from $\mathring\mC_{\mathrm{p},x}(\Omega)$ to $\mathring\mC_{\mathrm{p},x}(\Omega)$. Let $L_h: \mathring\mC_{\mathrm{p},x}(\Omega) \to \mathring\mC_{\mathrm{p},x}(\Omega)$ be the (one-to-one, onto) forward operator in the problem, by incorporating the boundary conditions into the definition of the operator. In fact, the operator $L_h$ is exactly $-\Delta_h$, except near the top and bottom physical boundaries, where the definition of the operator differs from the negative discrete Laplacian in order to incorporate the discrete homogeneous Neumann boundary conditions. Specifically, for any $\psi$, we see that  
\begin{equation}
L_h \psi_{i,j} = 
	\begin{cases}
\ -(\psi_{i+1,0} + \psi_{i-1,0} + 2\psi_{i,1} - 4\psi_{i,0})/h^2, & j = 0,
	\\
\ -(\psi_{i+1,N} + \psi_{i-1,N} + 2\psi_{i,N-1} - 4\psi_{i,N})/h^2, & j = N,
	\\
\ -\Delta_h \psi_{i,j}, & \mbox{otherwise}.
	\end{cases} \label{L_h}
\end{equation}
Another interpretation for this definition is that we have removed the ghost layer by incorporating the boundary conditions. The following result is available.
\begin{prop} \cite{CH-DBC-2024} 
The above-defined operator $L_h: \mathring\mC_{\mathrm{p},x}(\Omega) \to \mathring\mC_{\mathrm{p},x}(\Omega)$ is positive and symmetric in the sense that
\begin{align} 
  & 
    (\psi,L_h\phi) = (L_h\psi,\phi), \quad \forall \ \psi,\phi\in \mathring\mC_{\mathrm{p},x}(\Omega) , 
\\
  & \mbox{and} \quad 
(\psi,L_h\psi) = \|\nabla_h \psi\|_2^2 > 0 , \quad \forall \ \psi \in \mathring\mC_{\mathrm{p},x}(\Omega), \quad \psi\not\equiv 0.
\end{align}	
\end{prop}
The discrete $H^{-1}$ norm follows the definition of $L_h$: 
	\begin{equation} 
( \varphi_1 , \varphi_2 )_{-1} :=  \left( \varphi_1 , L_h^{-1} (\varphi_2)\right), \quad \| \varphi  \|_{-1} := \sqrt{( \varphi , \varphi )_{-1}} , \quad 
  \forall \varphi_1 ,\varphi_2,\varphi \in\mathring\mC_{\mathrm{p},x}(\Omega) . 
	\end{equation} 

Similarly, the following space is introduced for the vector functions: 
\begin{equation}
	\mV_{\mathrm{p},x}(\Omega) := \left\{ \bu=(u^x,u^y)^T \ \middle| \ u^x \in {\mathcal G}_{\mathrm{p},x}^+(\Omega),\ u^y \in \mC_{\mathrm{p},x}(\Omega)   \right\}.
\end{equation}
The no-penetration, no-slip condition on $\Gamma_B$ and $\Gamma_T$ is enforced as
$$
u^y_{i,N}=u^y_{i,0}=0, \quad u^x_{i+1/2,N+\hf}+ u^x_{i+1/2,N-\frac12} = u^x_{i+1/2,-\hf} + u^x_{i+1/2,\hf}=0.
$$
For $\bu, \bv \in \mV_{\mathrm{p},x}(\Omega)$, the corresponding vector inner product and $\| \cdot \|_2$ norm could be defined in a similar fashion as~\eqref{inner product-2}. In addition, some interpolation strategies are needed in the nonlinear term evaluation. For $\bu=(u^x, u^y)^T$, $\bv=(v^x, v^y)^T$, located at the staggered mesh points $(x_{i+\hf}, y_{j+\hf})$, $(x_{i}, y_j)$, respectively, and variables $\phi$, $\mu$, the following terms are evaluated as
\begin{align}
\bu \cdot \nabla_h \bv & =\binom{\ u_{i+\hf, j+\hf}^x \tilde{D}_x v_{i+\hf, j+\hf}^x+A_{x y} u_{i+\hf,j+\hf}^y \tilde{D}_y v_{i+\hf,j+\hf}^x\ }{\ a_{xy} u_{i,j}^x \tilde{D}_x v_{i,j}^y+ u_{i, j}^y \tilde{D}_y v_{i,j}^y\ }, \\
\nabla_h \cdot\left(\bu \bv^T\right) & =\binom{\ \tilde{D}_x (u^x v^x)_{i+\hf,j+\hf} + D_y (A_x u^y a_y v^x )_{i+\hf,j+\hf}\ }{\ d_x  ( a_y u^x A_x v^y )_{i,j}+ \tilde{D}_y \left(u^y v^y\right)_{i,j}\ }, \\
A_h \phi \nabla_h \mu &= \binom{\ A_{xy} \phi_{i+\hf, j+\hf} (A_y D_x \mu )_{i+\hf, j+\hf}  }{\ \phi_{i,j} \tilde{D}_y \mu_{i,j} }, \\
 \nabla_h \cdot (A_h \phi \bu) & = D_x (A_x \phi a_y u^x )_{i,j} + \tilde{D}_y (\phi u^y)_{i,j} , 
\end{align}
where the averaging operators $A_{xy}$, $a_{xy}$, and the wide stencil difference operators $\tilde{D}_x$ and $\tilde{D}_y$ could be similarly defined. 
The following discrete Poincar{\'e} inequality is an important tool in the later analysis. 

\begin{lem} \cite{chen24b} \label{lem:Poincare}
    The Poincar{\'e} inequality 
 is valid, i.e., there exists positive constant $C_0$ independent with $\dt$ and $h$ such that, 
        for $\phi \in \mC_{\mathrm{p},x}(\Omega)$ with $\overline{\phi}=0$, we have
        \begin{equation}
            \|\phi\|_2 \leq C_0 \|\nabla_h \phi\|_2 , \quad \| \phi \|_{-1, h} \le C_0 \| \phi \|_2 . 
             \label{Poincare 1}
        \end{equation}
\end{lem}

Moreover, an inverse inequality-style estimate is established in the following lemma, which is essential for the maximum norm bound in the later numerical analysis.
\begin{lem} \label{inverse inequality lemma}
    For any grid function $f\in \mC_{\mathrm{p},x}(\Omega)$ with $\overline{f} =0$, we have
    \begin{equation}
        \| f \|_\infty \leq C_\delta  h^{-\delta} \|\nabla_h f  \|_2 ,  \quad \forall \, \delta> 0 ,   \quad 
        \| f \|_6 \le C_1 \| \nabla_h f \|_2 ,  \quad \|f\|_{3}\leq C_1 h^{-\frac{1}{3}}\|f\|_2,  
        \label{inverse ineq}
    \end{equation}
    in which the constants $C_\delta>0$ and $C_1$ are independent on $h$. 
\end{lem}
\begin{proof}
In the $x$-direction, the periodic boundary condition naturally accommodates the standard discrete Fourier transform. In the $y$-direction, an even extension is employed to enable an application of a discrete Fourier cosine transform. Assuming $N = 2K + 1$ is odd, and the case for an even $N$ could be analogously treated. The corresponding discrete transformation for $\phi$ turns out to be 
\begin{align}
    f_{j,k}=\sum_{\ell=-K}^{K} \sum_{m=0}^{N} \alpha_m \hat{F}^N_{\ell,m} {\rm e}^{2\pi i \ell x_j} \cos (m\pi y_k), \quad \alpha_m = \begin{cases}
 1 &  m \neq 0, \\
 2^{-\frac12} & m=0,
\end{cases} \label{inverse ineq proof 1}
\end{align}
where $x_j=j h$, $y_k=kh$ with $h=1/N$. In more details, $\hat{F}^N_{\ell,m}$ is the Fourier coefficient under the combination of the Fourier transformation in the $x$-direction and the cosine transformation in the $y$-direction. In turn, an extension is made to obtain a continuous function
\begin{align}
    f_F(x,y) = \sum_{\ell=-K}^{K} \sum_{m=0}^{N} \alpha_m \hat{F}^N_{\ell,m} 
    {\rm e}^{2\pi i \ell x} \cos (m\pi y).\label{inverse ineq proof 2}
\end{align}
Parseval’s identity (at both the discrete and continuous levels) implies that
\begin{align}
    \| f \|_{2}^2 & = h^2\sum_{j,k}| f_{j,k}|^2 = \frac{1}{2} \sum_{\ell,m}|\hat{F}^N_{\ell,m}|^2, \quad 
    \| f_F \|_{L^2}^2  = \frac{1}{2} \sum_{\ell,m}|\hat{F}^N_{\ell,m}|^2 . \label{inverse ineq proof 3} 
\end{align}

In terms of a comparison between the discrete and continuous gradients, the following Fourier expansions indicate that
\begin{align}
    & (D_x f)_{j+\frac{1}{2},k} = \frac{f_{j+1,k} - f_{j,k}}{h} 
    = \sum_{\ell=-K}^{K} \sum_{m=0}^{N} \alpha_m \mu_{\ell}\hat{F}^N_{\ell,m} 
    {\rm e}^{2\pi i \ell x_{j+\frac{1}{2}}} \cos (m\pi y_k), \label{inverse ineq proof 5} \\
    & \partial_x f_F(x,y) = \sum_{\ell=-K}^{K} \sum_{m=0}^{N} \alpha_m \nu_{\ell} \hat{F}^N_{\ell,m} 
    {\rm e}^{2\pi i \ell x} \cos (m\pi y) ,  \mbox{where} \label{inverse ineq proof 6}
\\ 
  & 
    \mu_{\ell} =  \frac{{\rm e}^{\pi i \ell h}- {\rm e}^{-\pi i \ell h}}{h} = \frac{2i\sin(\pi\ell h)}{h}, \quad 
    \nu_{\ell} = 2\pi i \ell. \label{inverse ineq proof 7}
\end{align}
Again, the Parseval’s identity yields
\begin{align}
    \|D_x f \|_2^2 & = \frac{1}{2} \sum_{\ell,m} |\mu_{\ell}|^2 |\hat{F}^N_{\ell,m}|^2, \quad 
    \|\partial_x f_F\|_{L^2}^2 = \frac{1}{2} \sum_{\ell,m} |\nu_{\ell}|^2 |\hat{F}^N_{\ell,m}|^2. 
    \label{inverse ineq proof 8} 
\end{align}
A comparison of Fourier eigenvalues between $|\mu_{\ell}|$ and $|\nu_{\ell}|$ reveals that
\begin{equation}
   \frac{2}{\pi}|\nu_{\ell}| \leq |\mu_{\ell}| \leq |\nu_{\ell}|.\label{inverse ineq proof 10} 
\end{equation}
This in turn indicates that
\begin{equation}
   \frac{2}{\pi}\|\partial_x f_F\|_{L^2} \leq \|D_x f \|_2 \leq \|\partial_x f_F\|_{L^2}.\label{inverse ineq proof 11} 
\end{equation}
For the cosine part, similar comparison estimates could be derived in the same manner: 
\begin{equation}
   \frac{2}{\pi}\|\nabla f_F\|_{L^2} \leq \|\nabla_h f \|_2 \leq \|\nabla f_F\|_{L^2} . 
   \label{inverse ineq proof 12} 
\end{equation}
In addition, the discrete Fourier transformation, periodic in the $x$ direction and cosine transformation in the $y$ direction, leads to the identity between the discrete average of $f$ and the continuous average of $f_F$:
\begin{equation}
    \overline{f} := \frac{1}{|\Omega|}(f , 1) = \alpha_0 \hat{F}^N_{0,0} = \frac{1}{|\Omega|} \int_{\Omega} f_F(x) dx =:\overline{f_F} = 0 . \label{inverse ineq proof 15} 
\end{equation}
Based on this fact, we see that
\begin{equation}
    \| f \|_{\infty}  \leq \| f_F \|_{L^\infty} \le C_\delta \| f_F \|_{H^{1+\delta}} 
    \le C_\delta h^{-\delta} \| f_F \|_{H^1} \le C C_\delta h^{-\delta} \| \nabla f_F \|_{L^2} 
    \le  C C_\delta h^{-\delta} \| \nabla_h f \|_2  , \label{inverse ineq proof 16} 
\end{equation}
in which the 2-D Sobolev embedding has been applied in the second step, the third step corresponds to an inverse inequality (since $f_F \in {\cal B}^K$), the fourth step makes use of the Poincar\'e inequality (due to the fact that $\overline{f_F} =0$), and~\eqref{inverse ineq proof 12} has been applied in the last step. As a result, the first inequality in \eqref{inverse ineq} has been proved. 

The second and third inequalities in \eqref{inverse ineq} have been proved in an existing work \cite{Chen20-7} and the third inequality is a standard 2-D inverse inequality. This in turn finishes the proof of Lemma~\ref{inverse inequality lemma}.
\end{proof}

\subsection{The fully discrete numerical scheme}
For simplicity of presentation, we choose $F(\phi)=G(\phi)=R (\phi) := \frac{1}{4}\phi^4 - \frac{1}{2} \phi^2$ for the energy densities (in \eqref{energy-CH-1}) in the analysis below. As mentioned in Remark 2.2 below, the analysis could be applied to other choices of the energy densities. 
The following finite difference scheme is proposed for the system \eqref{CHNS-DBC-1}-\eqref{CHNS-DBC-6} with $F'(\phi)=G'(\phi)=\phi^3-\phi $, using the convex-splitting approach: given $\phi^n$, $\bu^n$, find $\phi^{n+1}$, $\mu^{n+1}\in \mC_{\mathrm{p},x}^+(\Omega)$, $\bu^{n+1}\in \mV_{\mathrm{p},x}(\Omega)$ such that
\begin{align}
& \frac{\hat{\bu}^{n+1}-\bu^n}{\dt} + b_h(\bu^n, \hat{\bu}^{n+1}) + \nabla_h p^n-\nu \Delta_h\hat{\bu}^{n+1}=-\gamma A_h\phi^n \nabla_h \mu^{n+1}, \label{scheme-CHDBC-1} \\
& \frac{\phi^{n+1}-\phi^{n}}{\dt} + \nabla_h\cdot\left( A_h\phi^n\hat{\bu}^{n+1} \right)=\Delta_h \mu^{n+1},  \quad \mu^{n+1}=(\phi^{n+1})^3-\phi^n-\varepsilon^2\Delta_h\phi^{n+1},  \label{scheme-CHDBC-2}\\
& \tilde{D}_{y} \mu^{n+1}_{i, 0} = \tilde{D}_{y} \mu^{n+1}_{i, N}  = 0  ,  \quad 
  \phi^{n+1}_{i,0} = \phi^{B,n+1}_{i} , \  \  \phi^{n+1}_{i,N} = \phi^{T,n+1}_{i}, \label{scheme-CHDBC-3} \\
& \frac{\phi^{B,n+1} - \phi^{B,n}}{\dt} = -\mu^{n+1}_B, \quad \mu^{n+1}_B=(\phi^{B,n+1})^3 - \phi^{B,n} - \varepsilon^2 \tilde{D}_{y} \phi^{n+1}_{\cdot, 0}, \label{scheme-CHDBC-4} \\
& \frac{\phi^{T,n+1} - \phi^{T,n}}{\dt} = -\mu^{n+1}_T, \quad \mu^{n+1}_B = (\phi^{T,n+1})^3 - \phi^{T,n} + \varepsilon^2 \tilde{D}_{y} \phi^{n+1}_{\cdot, N}, \label{scheme-CHDBC-5}	\\
& \frac{\bu^{n+1}-\hat{\bu}^{n+1}}{\dt} + \nabla_h(p^{n+1}-p^n)=0, \quad \nabla_h \cdot \bu^{n+1}=0, \label{scheme-CHDBC-6} 
\end{align}
with boundary condition $\hat{\bu}^{n+1} \mid_{\Gamma} =\boldsymbol{0}$, $(\bu^{n+1} \cdot \boldsymbol{n} ) \mid_\Gamma =0$. The bilinear form $b_h(\cdot,\cdot)$ is the discretization of the continuous $b(\cdot,\cdot)$: 
\begin{equation} 
    b(\bu,\bv):= \bu \cdot \nabla \bv, \quad b_h(\bu,\bv):=\frac{1}{2}\left( \bu \cdot \nabla_h \bv + \nabla_h\cdot(\bu \bv^T) \right) . \label{bilinear form}
\end{equation}

\subsection{Unconditionally unique solvability}
As a nonlinear scheme, the unique solvability is the basis of any subsequent analysis. To proceed with the unique solvability analysis, we implicitly define a linear operator $\mL_h^n$ as follows. Assume that the variables $\bu^n$, $\phi^n$ and $p^n$ are fixed. For a given field $\mu$, $\bv=\mL_h^n \mu$ is determined by the following $x$-periodic, discrete convection-diffusion equation: 
\begin{align}
    \frac{\bv-\bu^n}{\dt} + b_h(\bu^n, \bv) + \nabla_h p^n-\nu \Delta_h \bv =-\gamma A_h\phi^n \nabla_h \mu,   \label{solvability 1.1}
\end{align}
with no-penetration, no-slip boundary for $\bv$, 
in the $y$-direction. Given the variables $(\bu^n, \phi^n, p^n)$ at time level $t^n$, it is noticed that $\hat{\bu}^{n+1} = \mL_h^n \mu^{n+1}$, and $\bu^{n+1}$ is a further Helmholtz projection of $\hat{\bu}^{n+1}$ into the divergence-free space. Subsequently, a substitution of $\hat{\bu}^{n+1} = \mL_h^n \mu^{n+1}$ into \eqref{scheme-CHDBC-2} leads to the following equation for $\phi^{n+1}$ and $\mu^{n+1}$:
\begin{align}
    \frac{\phi^{n+1}-\phi^{n}}{\dt} + \nabla_h\cdot\left( A_h\phi^n \mL_h^n \mu^{n+1} \right)=\Delta_h \mu^{n+1},  \quad \mu^{n+1}=(\phi^{n+1})^3-\phi^n-\varepsilon^2\Delta_h\phi^{n+1}. \label{solvability 1.2}
\end{align}
Furthermore, \eqref{solvability 1.2} could be rewritten as
\begin{align}
    \frac{\phi^{n+1}-\phi^{n}}{\dt} + \mG_h^n \mu^{n+1} =0,  \quad \mG_h^n \mu :=\nabla_h\cdot\left( A_h\phi^n \mL_h^n \mu \right) - \Delta_h \mu. \label{solvability 1.3}
\end{align}
The property of the linear operator $\mG_h^n$ is essential for the unique solvability analysis. The following two lemmas demonstrate that it is monotone and therefore invertible; the proofs are very similar to the one presented in~\cite{chen22b}, and the technical details are skipped for the sake of brevity. 
\begin{lem} \cite{chen22b} \label{lem:monotonic}
    The linear operator $\mG_h^n$ is monotone in sense that for any $\mu_1$, $\mu_2$,
    \begin{equation}
        (\mG_h^n \mu_1-\mG_h^n \mu_2, \mu_1-\mu_2) \geq \| \nabla_h(\mu_1-\mu_2) \|_2^2 \geq 0. \label{monotonic}
    \end{equation}
    In addition, if we require $\overline{\mu_1}=\overline{\mu_2}=0$, equality is realized if and only if $\mu_1=\mu_2$. Therefore, the operator $\mG_h^n$ is invertible.
\end{lem}
As a direct consequence of Lemma \ref{lem:monotonic}, the following result for the inverse operator is valid.
\begin{prop} \cite{chen22b} \label{prop:inverse}
    The inverse operator $(\mG_h^n)^{-1}$ is also monotone:
    \begin{align} 
      & 
        ((\mG_h^n)^{-1}(\phi_1-\phi_2), \phi_1-\phi_2) \geq \|\nabla_h (\mG_h^n)^{-1}(\phi_1-\phi_2)\|_2^2 \geq \frac{1}{C_0^2} \| (\mG_h^n)^{-1}(\phi_1-\phi_2)\|_2^2  ,  \label{inverse}
\\
    & \mbox{and} \quad 
         \| (\mG_h^n)^{-1}(\phi_1-\phi_2)\|_2 \leq C_0^2 \|\phi_1-\phi_2\|_2,  \label{inverse 0}
    \end{align}
    for any $\phi_1$, $\phi_2$ with $\overline{\phi_1}=\overline{\phi_2}=0$. The constant $C_0$ is associated with the discrete Poincar{\'e} inequality \eqref{Poincare 1}. In addition, the equality holds if and only if $\phi_1=\phi_2$.
\end{prop}
It is noticed that the operator $\mathcal{G}_h^n$ defined in \eqref{solvability 1.3} only involves the velocity and chemical potential, which are subject to homogeneous no-slip and Neumann boundary conditions, respectively. Consequently, its monotonicity analysis does not incorporate any information regarding dynamic boundary. This aspect must be carefully addressed in the subsequent analysis. To address this issue, we introduce the Browder–Minty lemma, which plays a crucial role in establishing unique solvability.
\begin{lem}[Browder-Minty \cite{Browder-Minty-1,Browder-Minty-2}]
Let $X$ be a real, reflexive Banach space and suppose $X'$ is its dual. Let $T: X \rightarrow X'$ be: (i) bounded; (ii) continuous; (iii) coercive, that is
\begin{align}
    \frac{\langle T(u), u\rangle}{\|u\|_X} \rightarrow+\infty \quad \text { as } \quad\|u\|_X \rightarrow+\infty;
\end{align}
and (iv) monotone. Then for any $g \in X'$ there exists a solution $u \in X$ of the equation $T(u)=g$. Furthermore, if the operator $T$ is strictly monotone, then the solution $u$ is unique.
\end{lem}
Then we proceed into the proof of unique solvability.
\begin{theorem}\label{thm:solvability}
    Given $\phi^n \in $ with $\overline{\phi^n}=\beta_0$, there exists a unique solution $\phi^{n+1}$ to the numerical scheme \eqref{scheme-CHDBC-1}-\eqref{scheme-CHDBC-6}, or equivalently \eqref{inverse 1.2}, with $\overline{\phi^{n+1}}=\beta_0$.
\end{theorem}
\begin{proof}
    By \eqref{solvability 1.3}, numerical scheme \eqref{scheme-CHDBC-2} could be rewritten as a nonlinear system
\begin{equation}
    \mathcal{F}_h^n(\phi):=(\mG_h^n)^{-1}\frac{\phi-\phi^{n}}{\dt} + (\phi)^3 - \phi^{n} - \varepsilon^2 \Delta_h \phi=0, \label{inverse 1.2}
\end{equation}
coupled with the boundary condition \eqref{scheme-CHDBC-4}, \eqref{scheme-CHDBC-5}: for $\phi^{T}=\phi_{\cdot,N}$ and $\phi^{B}=\phi_{\cdot,0}$,
\begin{align}
\frac{\phi^{T} - \phi^{T,n}}{\dt} = - \big( (\phi^{T})^3 - \phi^{T,n} + \varepsilon^2 \tilde{D}_{y} \phi_{\cdot, N} \big), \quad \frac{\phi^{B} - \phi^{B,n}}{\dt} = -\big( (\phi^{B})^3 - \phi^{B,n} - \varepsilon^2 \tilde{D}_{y} \phi_{\cdot, 0} \big). \label{inverse 1.3.1}
\end{align}
It is clear that $\mathcal{F}_h^n$ maps $\mathbb{R}^{N(N+1)}$ into $\mathbb{R}^{N(N+1)}$, with the imposed discrete boundary conditions. In turn, we set $X=X'=\mathbb{R}^{N(N+1)}$ equipped with the discrete $\ell_2$ norm $\|\cdot\|_2$. 

The continuity analysis (in the $\| \cdot \|_2$ norm) is straightforward, and the technical details are skipped for the sake of brevity.

To establish the coercive property of $\mathcal{F}_h^n(\phi_2)$, we calculate the inner product
\begin{align}
    \left( \mathcal{F}_h^n(\phi), \phi \right) = \dt^{-1} \big( (\mG_h^n)^{-1} ( \phi-\phi^{n} ) , \phi \big) + \|\phi\|_4^4 - \left( \phi^{n},\phi \right) - \varepsilon^2 \left( \Delta_h \phi,\phi \right). \label{solvability 1.7}
\end{align}
For the first term, an application of the monotonicity inequality \eqref{inverse} leads to
\begin{align}
    \dt^{-1} \big( (\mG_h^n)^{-1} ( \phi-\phi^{n} ), \phi \big) & = \dt^{-1} \big( (\mG_h^n)^{-1} ( \phi-\phi^{n} ) , \phi-\phi^{n} \big) + \dt^{-1} \big( (\mG_h^n)^{-1} ( \phi-\phi^{n} ), \phi^{n} \big) \nonumber \\
    & \geq \frac{1}{\dt C_0^2}\| (\mG_h^n)^{-1}(\phi-\phi^{n}) \|_2^2 - \frac{1}{2\dt C_0^2}\| (\mG_h^n)^{-1}(\phi-\phi^{n}) \|_2^2 - \frac{C_0^2}{2\dt} \|\phi^{n}\|_2^2 \nonumber \\
    & = \frac{1}{2\dt C_0^2}\| (\mG_h^n)^{-1}(\phi-\phi^{n}) \|_2^2 - \frac{C_0^2}{2\dt} \|\phi^{n}\|_2^2. \label{solvability 1.7.1}
\end{align}
The second and third terms of the right hand of \eqref{solvability 1.7} could be handled by the quadratic inequality
\begin{align}
    \|\phi\|_4^4 - \left( \phi^{n},\phi \right) \geq 2\|\phi\|_2^2 - |\Omega| -\frac{1}{2}(\|\phi\|_2^2 + \|\phi^n\|_2^2) \geq \|\phi\|_2^2 - \frac{1}{2}\|\phi^n\|_2^2 - |\Omega|. \label{solvability 1.7.3}
\end{align}
The Laplacian term should be treated with the help of dynamical boundary condition:
\begin{align}
    - \varepsilon^2 \left( \Delta_h \phi,\phi \right) = \varepsilon^2\|\nabla_h \phi\|_2^2 - \varepsilon^2 (\phi^T, \tilde{D}_{y} \phi_{\cdot, N})_\Gamma + \varepsilon^2(\phi^B, \tilde{D}_{y} \phi_{\cdot, 0})_\Gamma.\label{solvability 1.7.4}
\end{align}
Based on the boundary condition \eqref{inverse 1.3.1} and an quadratic inequality, we see that 
\begin{align}
    - \varepsilon^2 (\phi^T, \tilde{D}_{y} \phi_{\cdot, N})_\Gamma &= \Big(\phi^T, (\phi^{T})^3 - \phi^{T,n} + \frac{\phi^{T} - \phi^{T,n}}{\dt} \Big)_\Gamma \nonumber \\
    & \geq \|\phi^{T}\|_{2,\Gamma}^2 - \frac{1}{2}\|\phi^{T,n}\|_{2,\Gamma}^2 - |\Gamma| + \frac{1}{2\dt}\left( \|\phi^{T}\|_{2,\Gamma}^2 - \|\phi^{T,n}\|_{2,\Gamma}^2 \right) \nonumber \\
    & \geq - \frac{1}{2}\|\phi^{T,n}\|_{2,\Gamma}^2 - |\Gamma| -\frac{1}{2\dt}\|\phi^{T,n}\|_{2,\Gamma}^2. \label{solvability 1.7.5}
\end{align}
Analogously, the following lower bound is available: 
\begin{align}
    \varepsilon^2(\phi^B, \tilde{D}_{y} \phi_{\cdot, 0})_\Gamma \geq - \frac{1}{2}\|\phi^{B,n}\|_{2,\Gamma}^2 - |\Gamma| -\frac{1}{2\dt}\|\phi^{B,n}\|_{2,\Gamma}^2. \label{solvability 1.7.6}
\end{align}
As a result, a substitution of \eqref{solvability 1.7.1}, \eqref{solvability 1.7.3}-\eqref{solvability 1.7.6} into \eqref{solvability 1.7} leads to the following estimate:
\begin{align}
    \frac{\left( \mathcal{F}_h^n(\phi), \phi \right)}{\|\phi\|_2} & \geq \|\phi\|_2 - \frac{1}{\|\phi\|_2} \Big( \frac{C_0^2}{2\dt} \|\phi^{n}\|_2^2 +\frac{1}{2}\|\phi^n\|_2  + |\Omega| +2|\Gamma|  \Big) \nonumber \\
    & \quad - \frac{1}{2\|\phi\|_2} \Big( \|\phi^{T,n}\|_{2,\Gamma}^2 + \|\phi^{B,n}\|_{2,\Gamma}^2 + \frac{1}{\dt}\|\phi^{T,n}\|_{2,\Gamma}^2 + \frac{1}{\dt}\|\phi^{B,n}\|_{2,\Gamma}^2 \Big), \label{solvability 1.7.7}
\end{align}
which in turn demonstrates the coercive property, $\frac{\left( \mathcal{F}_h^n(\phi), \phi \right)}{\|\phi\|_2} \rightarrow +\infty$, as $\|\phi\|_2 \rightarrow +\infty$.

The remaining work is focused on establishing a monotonicity condition. Given $\phi_1$, $\phi_2$, by denoting $\tilde{\phi}=\phi_1-\phi_2$, a direct calculation implies that 
\begin{align}
   \left( \mathcal{F}_h^n(\tilde{\phi}), \tilde{\phi} \right)& = \frac{1}{\dt} \big( (\mG_h^n)^{-1}\tilde{\phi} , \tilde{\phi} \big) + (  \phi_1^3 - \phi_2^3, \tilde{\phi} ) - \varepsilon^2 ( \Delta_h \tilde{\phi}, \tilde{\phi} ) \nonumber \\
   & \geq \frac{1}{\dt C_0^2} \|(\mG_h^n)^{-1}\tilde{\phi}\|_2^2 + \big(  \tilde{\phi}(\phi_1^2+\phi_1\phi_2+\phi_2^2), \tilde{\phi} \big) - \varepsilon^2 ( \Delta_h \tilde{\phi}, \tilde{\phi} ) \nonumber \\
   & \geq \varepsilon^2\|\nabla_h \tilde{\phi}\|_2^2 - \varepsilon^2 (\tilde{\phi}^T, \tilde{D}_{y} \tilde{\phi}_{\cdot, N})_\Gamma + \varepsilon^2(\tilde{\phi}^B, \tilde{D}_{y} \tilde{\phi}_{\cdot, 0})_\Gamma.\label{solvability 1.8.1}
\end{align}
Analogous to \eqref{solvability 1.7.5} and \eqref{solvability 1.7.6}, we have
\begin{align}
   - &\varepsilon^2 (\tilde{\phi}^T, \tilde{D}_{y} \tilde{\phi}_{\cdot, N})_\Gamma =  \big(\tilde{\phi}^T, (\phi^{T}_1)^3 - (\phi^{T}_2)^3 + \dt^{-1} \tilde{\phi}^T \big)_\Gamma \geq 0, \label{solvability 1.8.2}\\
  & \varepsilon^2(\tilde{\phi}^B, \tilde{D}_{y} \tilde{\phi}_{\cdot, 0})_\Gamma = \big(\tilde{\phi}^B, (\phi^{B}_1)^3 - (\phi^{B}_2)^3 + \dt^{-1} \tilde{\phi}^B \big)_\Gamma \geq 0. \label{solvability 1.8.3}
\end{align}
Then we arrive at
\begin{equation}
    \left( \mathcal{F}_h^n(\phi_1)-\mathcal{F}_h^n(\phi_2), \phi_1-\phi_2 \right) \geq 0.  \label{solvability 1.8.4}
\end{equation}
As a result, an application of the Browder-Minty lemma implies a unique solution to \eqref{inverse 1.2}, so that the unique solvability of the numerical scheme \eqref{scheme-CHDBC-1}-\eqref{scheme-CHDBC-6} has been established. 
\end{proof}
\begin{rem}
	In the proof of Theorem \ref{thm:solvability}, the discrete $\ell_2$ norm is taken for the Banach space, without considering the boundary norm. In fact, by denoting an extended $\ell_2$ norm, defined as $\|\phi\|_{2,*}^2:=\|\phi\|_{2}^2+\|\phi^T\|_{2,\Gamma}^2+\|\phi^B\|_{2,\Gamma}^2$, an equivalence inequality is available:  
	\begin{equation}
	    \|\phi\|_{2}^2 \le	\|\phi\|_{2,*}^2 \le (1+ Ch^{-\frac12}  ) \|\phi\|_{2}^2 , 
	\end{equation}
	where the surface parts have been extended to the whole domain, based on the inverse inequality. Therefore, the boundary $\ell_2$ norms $\|\phi^{T}\|_{2,\Gamma}^2$ and $\|\phi^{B}\|_{2,\Gamma}^2$ in \eqref{solvability 1.7.5} and \eqref{solvability 1.7.6} are neglected. The bulk $\ell_2$ norm in \eqref{solvability 1.7.3} is sufficient for the coercive analysis.
\end{rem}
\begin{rem}
	It should be emphasized that the free energy densities $F$ and $G$ appearing in \eqref{energy-CH-1} may not be restricted to polynomial forms. They may, in fact, be any functions with bounded second derivatives -- trigonometric functions, for instance. The analysis of unique solvability, together with the subsequent energy stability and convergence results, holds for all such functions. For the sake of clear presentation, this work specifically considers $G (\psi) = F (\psi) = 1/4 \psi^4 - 1/2 \psi^2$.
\end{rem}

\section{Unconditional total energy stability}  \label{sec:energy}
Now we establish an unconditional total energy stability of the proposed numerical scheme. The discrete energy is defined as (with $R(c)=\frac{1}{4}c^4 - \frac{1}{2} c^2$): 
\begin{align}
    & E_{h}(\phi,\bu)=\frac{1}{2\gamma}\|\bu\|_2^2 + E_{\rm bulk}(\phi) + E_{T}(\phi^T) + E_{B}(\phi^B), \\
    & E_{\rm bulk}(\phi) = \left( R(\phi),1 \right) + \frac{\varepsilon^2}{2}\|\nabla_h \phi\|_2^2, \quad E_{T}(\phi^T) = \left( R(\phi^T),1 \right)_\Gamma, \quad E_{B}(\phi^B) = \left( R(\phi^B),1 \right)_\Gamma . 
    \label{discrete energy}
\end{align}
\begin{theorem} \label{thm: energy dissipation}
    For the numerical solution of \eqref{scheme-CHDBC-1}-\eqref{scheme-CHDBC-6}, a modified energy stability is valid in the sense that
    \begin{align}
    (E_h^{n+1} &+ \frac{\dt^2}{2\gamma} \|\nabla_h p^{n+1}\|_2^2) - (E_h^{n} + \frac{\dt^2}{2\gamma}\|\nabla_h p^{n}\|_2^2) \nonumber \\
    & \le - \dt\| \nabla_h \mu^{n+1} \|_2^2 -\dt \| \mu^{n+1}_T \|_{2,\Gamma}^2 -\dt \| \mu^{n+1}_B \|_{2,\Gamma}^2 - \frac{\dt\nu}{\gamma} \|\nabla_h \hat{\bu}^{n+1}\|_2^2 \leq 0. \label{energy} 
\end{align}
\end{theorem}
\begin{proof}
Taking a discrete inner product with \eqref{scheme-CHDBC-1} by $\hat{\bu}^{n+1}$, \eqref{scheme-CHDBC-2} by $\mu^{n+1}$, gives
\begin{align}
    & \frac{1}{2\dt}( \|\hat{\bu}^{n+1}\|_2^2-\|\bu^{n}\|_2^2+\|\hat{\bu}^{n+1}-\bu^{n}\|_2^2 ) + (\nabla_h p^n, \hat{\bu}^{n+1}) + \nu \|\nabla_h \hat{\bu}^{n+1}\|_2^2 \nonumber \\
    & \quad\quad\quad\quad\quad\quad\quad\quad = -\gamma(A_h\phi^n \nabla_h \mu^{n+1}, \hat{\bu}^{n+1}), \label{energy 1.1} \\
    &  \left(\phi^{n+1} - \phi^n ,(\phi^{n+1})^3-\phi^n-\varepsilon^2\Delta_h\phi^{n+1} \right) - \dt(A_h\phi^n \nabla_h \mu^{n+1}, \hat{\bu}^{n+1}) = -\dt\| \nabla_h \mu^{n+1} \|_2^2 .  \label{energy 1.2}
\end{align}
The convexity analysis for the discrete energy terms reveals that
\begin{align}
    \left(\phi^{n+1} - \phi^n ,(\phi^{n+1})^3-\phi^n \right) \geq (R(\phi^{n+1})-R(\phi^{n}), 1). \label{energy 1.3}
\end{align}
Regarding the Laplacian term, we see that 
\begin{align}
    \left( \phi^{n+1} - \phi^n ,-\varepsilon^2\Delta_h\phi^{n+1} \right) &\geq \frac{\varepsilon^2}{2}(\|\nabla_h \phi^{n+1}\|_2^2-\|\nabla_h \phi^{n}\|_2^2) \nonumber \\
    & -\varepsilon^2(\phi^{n+1} - \phi^n, \tilde{D}_y \phi^{n+1}_{\cdot,N})_\Gamma +  \varepsilon^2(\phi^{n+1} - \phi^n, \tilde{D}_y \phi^{n+1}_{\cdot,0})_\Gamma . \label{energy 1.4}
\end{align}
The boundary parts could be further controlled, with the help of the boundary evolutionary equation. Taking a discrete inner product with \eqref{scheme-CHDBC-4} by $\mu^{n+1}_B$ shows that
\begin{align}
    (R(\phi^{B,n+1})-R(\phi^{B,n}), 1)_\Gamma \leq  -\dt \| \mu^{n+1}_B \|_{2,\Gamma}^2 + \varepsilon^2(\phi^{B,n+1} - \phi^{B,n}, \tilde{D}_y \phi^{n+1}_{\cdot,0})_\Gamma, \label{energy 1.5}
\end{align}
where the convexity of the boundary energy has been considered. The inequality on the top boundary section is also valid:
\begin{align}
    (R(\phi^{T,n+1})-R(\phi^{T,n}), 1)_\Gamma \leq  -\dt \| \mu^{n+1}_T \|_{2,\Gamma}^2 - \varepsilon^2(\phi^{T,n+1} - \phi^{T,n}, \tilde{D}_y \phi^{n+1}_{\cdot,N})_\Gamma. \label{energy 1.6}
\end{align}
Consequently, a substitution of \eqref{energy 1.3}-\eqref{energy 1.6} into \eqref{energy 1.2} leads to
\begin{align}
    (E_{\rm bulk}^{n+1} + E_{T}^{n+1} + E_{B}^{n+1}) &- (E_{\rm bulk}^{n} + E_{T}^{n} + E_{B}^{n}) - \dt(A_h\phi^n \nabla_h \mu^{n+1}, \hat{\bu}^{n+1})  \nonumber \\
    \le & -\dt\| \nabla_h \mu^{n+1} \|_2^2 -\dt \| \mu^{n+1}_T \|_{2,\Gamma}^2 -\dt \| \mu^{n+1}_B \|_{2,\Gamma}^2.\label{energy 1.7}
\end{align}
By \eqref{scheme-CHDBC-6}, the following equality is observed: 
\begin{align}
     (\nabla_h p^n, \hat{\bu}^{n+1}) &=-(p^n, \nabla_h \cdot \hat{\bu}^{n+1})= -(p^n, \dt \Delta_h(p^{n+1}-p^n)) 
     = \dt ( \nabla_h p^n, \nabla_h(p^{n+1}-p^n))  \nonumber \\ 
    & = \frac{\dt}{2}(\|\nabla_h p^{n+1}\|_2^2-\|\nabla_h p^{n}\|_2^2-\|\nabla_h p^{n+1}-\nabla_h p^{n}\|_2^2)   \nonumber \\
    & = \frac{\dt}{2}(\|\nabla_h p^{n+1}\|_2^2-\|\nabla_h p^{n}\|_2^2) - \frac{1}{2\dt} \|\bu^{n+1}-\hat{\bu}^{n+1}\|_2^2 .\label{energy 1.8} 
\end{align}
A combination of \eqref{energy 1.1}, \eqref{energy 1.7} and \eqref{energy 1.8} yields 
\begin{align}
    & (E_{\rm bulk}^{n+1} + E_{T}^{n+1} + E_{B}^{n+1}) - (E_{\rm bulk}^{n} + E_{T}^{n} + E_{B}^{n}) + \frac{1}{2\gamma}(\|\hat{\bu}^{n+1}\|_2^2-\|\bu^{n}\|_2^2-\|\bu^{n+1}-\hat{\bu}^{n+1}\|_2^2) \nonumber \\
    & \quad + \frac{\dt^2}{2\gamma} (\|\nabla_h p^{n+1}\|_2^2-\|\nabla_h p^{n}\|_2^2)    \nonumber \\
    & \leq -\dt\| \nabla_h \mu^{n+1} \|_2^2 -\dt \| \mu^{n+1}_T \|_{2,\Gamma}^2 -\dt \| \mu^{n+1}_B \|_{2,\Gamma}^2 - \frac{\dt\nu}{\gamma} \|\nabla_h \hat{\bu}^{n+1}\|_2^2 . \label{energy 1.9}
\end{align}
On the other hand, taking a discrete inner product with \eqref{scheme-CHDBC-6} by $\bu^{n+1}$ results in
\begin{equation}
    \|\bu^{n+1}\|_2^2 - \|\hat{\bu}^{n+1}\|_2^2 + \|\bu^{n+1}-\hat{\bu}^{n+1}\|_2^2 = 0,\label{energy 1.10}
\end{equation}
 which comes from the fact that $(\nabla_h(p^{n+1}-p^n),\bu^{n+1})=0$. A combination of the last
identity with \eqref{energy 1.9} leads to
\begin{align}
    E_h^{n+1}-E_h^{n} &+ \frac{\dt^2}{2\gamma} (\|\nabla_h p^{n+1}\|_2^2-\|\nabla_h p^{n}\|_2^2)  \nonumber \\
    \le & -\dt\| \nabla_h \mu^{n+1} \|_2^2 -\dt \| \mu^{n+1}_T \|_{2,\Gamma}^2 -\dt \|\mu^{n+1}_B \|_{2,\Gamma}^2 - \frac{\dt\nu}{\gamma} \|\nabla_h \hat{\bu}^{n+1}\|_2^2 \leq 0. \label{energy 1.11} 
\end{align}
This is exactly inequality \eqref{energy}, so that the total energy dissipation has been established. This finishes the proof of Theorem \ref{thm: energy dissipation}.
\end{proof}


\section{Convergence analysis}  \label{sec:convergence} 
Now we proceed into the convergence analysis. Let $(\phi_e, \bu_e,  p_e)$ be the exact solution to the CHNS system \eqref{CHNS-DBC-1}–\eqref{CHNS-DBC-6}, with dynamical boundary condition. With sufficiently regular initial data, it is assumed that the exact solution has regularity of class $\mathcal{R}$: 
\begin{equation}
\phi_e, \, \bu_e, \, p_e \in \mathcal{R} := C^2 (0,T; C_{\rm per}(\Omega) ) \cap C^1 (0,T; C^2_{\rm per}(\Omega) ) \cap L^\infty (0,T; C^8_{\rm per}(\Omega) ).
	\label{assumption:regularity.1}
	\end{equation}
For the phase variable $\phi_e$, there is a two-step adjustment to ensure the mass conservation in discrete level. The first step is a Fourier projection only in $x$-direction, i.e. $\Phi_{N,x}(\cdot,t) = \mathcal{P}_{N,x} \phi_e(\cdot,t)$.
The projection approximation is standard:
\begin{equation}
    \| \Phi_{N,x} - \phi_e\|_{L^\infty(0,T;H^k)} \leq C h^{\ell-k} \| \phi_e \|_{L^\infty(0,T;H^\ell)}, \quad  \forall 0 \le k \le \ell.\label{Fourier-approximation-3}
\end{equation}

In terms of the velocity variable, it is noticed that the exact velocity profile $\bu_e$ does not automatically satisfy the divergence-free condition at the discrete level. To enforce this property, a spatial interpolation operator $\mP_H$ is introduced. For any divergence-free field $\bu_e \in H^1 (\Omega)$ with $\nabla \cdot \bu_e = 0$, there exists an exact stream function $\psi_e$ such that $\bu_e = \nabla^{\perp} \psi_e$. Furthermore, to ensure a higher order consistency for the no-slip boundary condition, namely, $u=v=0$ at $y=0,1$, an asymptotic expansion of the stream function profile is necessary. In more details, the exact stream function satisfies both the Dirichlet and Neumann boundary conditions, namely $\psi_e = \partial_y \psi_e =0$ at $\Gamma_y: y=0$, and a careful Taylor expansion around $y=0$ indicates that 
\begin{equation} 
  ( \psi_e )_{i+\hf,-1} = ( \psi_e )_{i+\hf, 1} - 2 h^3 ( \partial_y^3 \psi_e )_{i+\hf, 0} /6 + O (h^5) .  
  \label{stream function-expansion-1} 
\end{equation} 
Based on this expansion, we construct an auxiliary stream function profile to satisfy 
\begin{equation} 
  \hat{\psi} \mid_{y=0} = \hat{\psi} \mid_{y=1} = 0 , \quad 
  \partial_y \hat{\psi} (x, 0, t) = -  \partial_y^3 \psi_e (x, 0 ,t) /6 , \, \, \, 
  \partial_y \hat{\psi} (x, 1, t) = -  \partial_y^3 \psi_e (x, 1 ,t) /6 . 
  \label{stream function-expansion-2} 
\end{equation} 
In fact, such a construction is straightforward; for example, an auxiliary profile could be taken as $\hat{\psi} (x,y,t) = - \frac16 \partial_y^3 \psi_e (x, 0 ,t) y (1-y)^2 + \frac16 \partial_y^3 \psi_e (x, 1 ,t) y^2 (1-y)$. Subsequently, a Taylor expansion could be applied to this auxiliary profile:  
\begin{equation} 
  \hat{\psi}_{i+\hf,-1} = \hat{\psi}_{i+\hf, 1} - 2 h \partial_y \hat{\psi}_{i+\hf, 0} + O (h^3) 
  = \hat{\psi}_{i+\hf, 1} + h ( \partial_y^3 \psi_e )_{i+\hf, 0}  / 3+ O (h^3)   . 
  \label{stream function-expansion-3} 
\end{equation} 
Therefore, by denoting $\Psi = \psi_e + h^2 \hat{\psi}$, we see that a combination of \eqref{stream function-expansion-1} and \eqref{stream function-expansion-3} leads to an $O (h^5)$ boundary extrapolation: 
\begin{equation} 
  \Psi_{i+\hf,-1} = \Psi_{i+\hf, 1}  + O (h^5)  , \quad  \Psi_{i+\hf, N+1} = \Psi_{i+\hf, N-1}  + O (h^5) . 
  \label{stream function-expansion-4} 
\end{equation} 
Similar asymptotic analysis for the boundary extrapolation could also be found in~\cite{WLJ04}, etc. 

In turn, the spatial interpolation operator $\mP_H$ is then defined as
\begin{equation}
    \bU = \mP_H( \nabla^\perp \Psi)=\nabla_h^{\perp} \Psi = (-D_y \Psi, D_x \Psi )^T,  
     \label{discrete Helmholtz interpolation}
\end{equation}
evaluated at the staggered mesh points, $(i+\hf, j+\hf)$, $(i,j)$, respectively. This operator ensures $\nabla_h \cdot \bU = 0$ at a point-wise level. Of course, the continuous velocity field $\bu$ and the constructed $\bU$ 
differ by an $O(h^2)$ truncation error. 
Moreover, by the $O (h^5)$ boundary extrapolation estimate~\eqref{stream function-expansion-4} for the constructed stream function profile, an $O (h^4)$ boundary estimate becomes available for the constructed horizontal velocity component: 
\begin{equation} 
\begin{aligned} 
  & 
  U_{i+\hf, \hf} = - D_y \Psi_{i+\hf,\hf} =  - \Psi_{i+\hf,1} / h  , \quad 
  U_{i+\hf, - \hf} = - D_y \Psi_{i+\hf, -\hf} =   \Psi_{i+\hf,-1} / h , 
\\
  & \mbox{so that}  \quad 
  U_{i+\hf, \hf}  + U_{i+\hf, - \hf} = O (h^4) . 
\end{aligned} 
  \label{velocity-expansion-1} 
\end{equation} 

A similar process could be applied to construct an approximate chemical potential profile to satisfies a higher order boundary extrapolation estimate. The exact chemical potential is taken as $\mu_e = \phi_e^3 - \phi_e - \varepsilon^2 \Delta \phi_e$, which satisfies a homogeneous Neumann boundary conditions at $\Gamma_B: y=0$. An associated Taylor expansion around $y=0$ gives 
\begin{equation} 
  ( \mu_e )_{i+\hf,-1} = ( \mu_e )_{i+\hf, 1} - 2 h^3 ( \partial_y^3 \mu_e )_{i+\hf, 0} /6 + O (h^5) ,  
  \label{mu-expansion-1} 
\end{equation} 
and an auxiliary chemical potential profile is constructed to satisfy 
\begin{equation} 
  \partial_y \hat{\mu} (x, 0, t) = - \partial_y^3 \mu_e (x, 0 ,t) /6 , \, \, \, 
  \partial_y \hat{\mu} (x, 1, t) = - \partial_y^3 \mu_e (x, 1 ,t) /6 . 
  \label{mu-expansion-2} 
\end{equation} 
Of course, a Taylor expansion to this auxiliary profile (around $y=0$) implies that   
\begin{equation} 
  \hat{\mu}_{i+\hf,-1} = \hat{\mu}_{i+\hf, 1} - 2 h \partial_y \hat{\mu}_{i+\hf, 0} + O (h^3) 
  = \hat{\mu}_{i+\hf, 1} + h ( \partial_y^3 \mu_e )_{i+\hf, 0} /3 + O (h^3)   . 
  \label{mu-expansion-3} 
\end{equation} 
Subsequently, we denote $\mV = \mu_e + h^2 \hat{\mu}$, so that a combination of \eqref{mu-expansion-1} and \eqref{mu-expansion-3} leads to an $O (h^5)$ boundary extrapolation: 
\begin{equation} 
  \mV_{i+\hf,-1} = \mV_{i+\hf, 1}  + O (h^5)  , \quad  \mV_{i+\hf, N+1} = \mV_{i+\hf, N-1}  + O (h^5) . 
  \label{mu-expansion-4} 
\end{equation} 

In terms of the phase variable, a similar process is needed to preserve a higher order accuracy for the boundary extrapolation. We denote ${\bf \varphi}^B = - \varepsilon^{-2} ( \partial_t \Phi_{N,x} + G ( \Phi_{N,x} )) \mid_{y=0}$,  ${\bf \varphi}^T = - \varepsilon^{-2} ( \partial_t \Phi_{N,x} + G ( \Phi_{N,x} )) \mid_{y=1}$, the exact normal derivative profile of the phase variable, determined by~\eqref{CHNS-DBC-5}. The Taylor expansion around $y=0$ implies that 
\begin{equation}  
  ( \Phi_{N,x} )_{i+\hf,-1} = ( \Phi_{N,x} )_{i+\hf, 1} + 2h {\bf \varphi}^B_{i+\hf, 0} - 2 h^3 ( \partial_y^3 \Phi_{N,x} )_{i+\hf, 0} /6 + O (h^5) ,  
  \label{phi-expansion-1} 
\end{equation} 
and an auxiliary phase variable profile satisfies 
\begin{align}  
  & 
  \overline{\hat{\Phi}} = 0 , \quad 
  \partial_y \hat{\Phi} (x, 0, t) = -  \partial_y^3 \Phi_{N,x} (x, 0 ,t) /6 , \, \, \, 
  \partial_y \hat{\Phi} (x, 1, t) = -  \partial_y^3 \Phi_{N,x} (x, 1 ,t) /6 ,  
  \label{phi-expansion-2} 
\\ 
  & 
   \hat{\Phi}_{i+\hf,-1} = \hat{\Phi}_{i+\hf, 1} - 2 h \partial_y \hat{\Phi}_{i+\hf, 0} + O (h^3) 
  = \hat{\Phi}_{i+\hf, 1} + h ( \partial_y^3 \Phi_{N,x} )_{i+\hf, 0}  / 3+ O (h^3)   . 
  \label{phi-expansion-3} 
\end{align} 
In turn, by defining $\breve{\Phi}_N = \Phi_{N,x} + h^2 \hat{\Phi}$, we obtain
\begin{equation} 
\begin{aligned} 
  & 
  ( \breve{\Phi}_N )_{i+\hf,-1} = ( \breve{\Phi} _N)_{i+\hf, 1}  + 2h {\bf \varphi}^B_{i+\hf, 0} 
  + O (h^5)  ,  
\\
  & 
  ( \breve{\Phi}_N )_{i+\hf, N+1} = ( \breve{\Phi}_N )_{i+\hf, N-1}  + 2h {\bf \varphi}^T_{i+\hf, 0} + O (h^5) , 
\end{aligned} 
  \label{phi-expansion-4} 
\end{equation} 

Meanwhile, we have to introduce an auxiliary cosine function to enforce the bulk mass conservation at a discrete level, and a corrected profile is created: 
\begin{equation}
\delta \Phi_N(x,y,t) = \big(\overline{\breve{\Phi}_N} 
- \overline{\breve{\Phi}_N^0} \big) ( 1 - \cos(2 \pi y) ), \quad \Phi_N:= \breve{\Phi}_N - \delta \Phi_N,
  \quad x, y \in [0,1]. \label{cos auxiliary function 1}
\end{equation}
A few helpful properties of the auxiliary function are stated below; see \cite{Guo_2025_convergence_CHDBC} for a detailed proof.
\begin{lem} \cite{Guo_2025_convergence_CHDBC} \label{cos function property}
    For the auxiliary function $\delta \Phi_N$ in \eqref{cos auxiliary function 1}, the following properties are valid:
        \begin{align}
          & 
            \|\delta \Phi_N^n\|_2 \leq C h^2,\quad n \ge 0 ,  \quad 
            \mbox{$C$ is a constant independent on $\dt$ and $h$}, 
             \label{auxiliary function order}
\\
           &      
            \overline{\Phi_N}=\overline{\breve{\Phi}_N-\delta \Phi_N}=\overline{\breve{\Phi}_N^0}=\overline{\phi^0}.  \label{bulk mass conservation}
  \\
      & 
      ( \Phi_N )_{i+\hf,-1} = ( \Phi _N)_{i+\hf, 1}  + 2h {\bf \varphi}^B_{i+\hf, 0} 
  + O (h^5) ,  \label{phi-expansion-5} 
\\
  & 
   ( \Phi_N )_{i+\hf, N+1} = ( \Phi_N )_{i+\hf, N-1}  + 2h {\bf \varphi}^T_{i+\hf, 0} + O (h^5) . \nonumber 
        \end{align}
\end{lem}
Estimate \eqref{auxiliary function order} indicates that the auxiliary function does not alter the overall convergence rate, as it corresponds to a truncation error of the same order. Furthermore, the cosine function inherently satisfies both Dirichlet and Neumann boundary conditions, thereby preserving mass conservation at the boundary and leaving the normal derivative unaffected. 

The corresponding error grid functions 
are introduced as
\begin{equation}
    \tilde{\phi}^n:= \Phi_N^n -\phi^n, \quad 
    \tilde{\bu}^n:=\bU^n-\bu^n, \quad \tilde{p}^n:=p_e^n-p^n, \quad \forall n \ge 0 .\label{error function-1}
\end{equation}
Furthermore, it follows that $\overline{\tilde{\phi}^n}=0$, for any $n \ge 0$, and the discrete norms $\| \cdot \|_{-1}$ and $\| \cdot \|_{-1,\Gamma}$ are well defined for the error grid function $\tilde{\phi}^n$. 

The following theorem is the main result of this section.
\begin{theorem} \label{main result theorem}
    Given initial data $\Phi(\cdot,t=0)$, $\bU(\cdot,t=0)$ with enough regularity, assume the exact solution for CHNS system \eqref{CHNS-DBC-1}-\eqref{CHNS-DBC-6} is of regularity class $\mathcal{R}$. Then, provided that $\dt$ and $h$ are sufficiently small, and under the linear refinement requirement $C_1 h\leq \dt \leq C_2 h$, we have
    \begin{equation}
    \|\nabla_h \tilde{\phi}^{n+1}\|_2 + \| \tilde{\phi}^{T,n+1}\|_{2,\Gamma} + \| \tilde{\phi}^{B,n+1}\|_{2,\Gamma} + \|\tilde{\bu}^{n+1}\|_2 \leq C(\dt + h^2), \label{main result}
\end{equation}
where $C$ is positive constant independent on $\dt$ and $h$.
\end{theorem}

\subsection{The consistency analysis}
The projection method applied to the no-slip velocity boundary condition presents certain theoretical challenges in the boundary analysis. In the consistency analysis, no intermediate velocity vector is needed, in comparison with the numerical solution. A careful truncation error estimate reveals that 
\begin{align}
& \frac{\bU^{n+1}-\bU^n}{\dt} +b_h(\bU^n,\bU^{n+1}) + \nabla_h P^n=\nu \Delta_h \bU^{n+1}-\gamma A_h\Phi_N^n \nabla_h \mV^{n+1} + \tau^{n+1}_{\bu} , \label{semi-CHDBC-1} \\
& \frac{\bU^{n+1}-\bU^{n+1}}{\dt} + \nabla_h(P^{n+1}-P^n) = \tau_p^{n+1}:= \nabla_h(P^{n+1}-P^n), \label{semi-CHDBC-9} \\
& \nabla_h \cdot \bU^{n+1}=0, \quad V^{n+1} |_{y=0,1}=0, \quad 
 ( a_y U^{n+1} ) |_{y=0,1}=0 , \label{semi-CHDBC-2.1} \\
& \frac{\Phi_N^{n+1}-\Phi_N^{n}}{\dt} + \nabla_h \cdot\left( A_h\Phi_N^n\bU^{n+1} \right)=\Delta_h \mV^{n+1} + \tau^{n+1}_{\phi},   \label{semi-CHDBC-2} \\
&\mV^{n+1}=(\Phi_N^{n+1})^3-\Phi_N^n-\varepsilon^2 \Delta_h \Phi_N^{n+1} + \tau_\mu^{n+1} ,
  \quad \partial_n \mV^{n+1} |_{y=0,1} 
 =0, \label{semi-CHDBC-3}\\
& \Phi_N^{n+1}|_{j=0} = \Phi_N^{B,n+1} , \  \  \Phi_N^{n+1}|_{j=N} = \Phi_N^{T,n+1}, \label{semi-CHDBC-4} \\
& \frac{\Phi_N^{B,n+1} - \Phi_N^{B,n}}{\dt} = - \mV_{B}^{n+1} + \tau^{n+1}_{B}, \quad \frac{\Phi_N^{T,n+1} - \Phi_N^{T,n}}{\dt} = - \mV_{T}^{n+1} + \tau^{n+1}_{T},\label{semi-CHDBC-5}  \\
& \mV_{B}^{n+1} = (\Phi_N^{B,n+1})^3 - \Phi_N^{B,n} 
-\varepsilon^2 (\tilde{D}_y \Phi^{n+1})_{\cdot,0}, \label{semi-CHDBC-6} \\
& \mV_{T}^{n+1} = (\Phi_N^{T,n+1})^3 - \Phi_N^{T,n} 
+ \varepsilon^2 (\tilde{D}_y \Phi^{n+1})_{\cdot,N} , \label{semi-CHDBC-8}
\\ 
  & \mbox{with} \, \, \, 
    \|\tau^{n+1}_{\bu}\|_2,\ \| \tau^{n+1}_{\phi} \|_2,\ \| \tau^{n+1}_{p} \|_2,\  \| \tau_\mu^{n+1} \|_{H_h^1} , \ \|\tau^{n+1}_T\|_{2,\Gamma},\ \|\tau^{n+1}_B\|_{2,\Gamma} \leq C(\dt+h^2).
\end{align}
In particular, we notice that, the constructed approximation chemical potential profile, given by $\mV^{n+1} = \mu_e^{n+1} + h^2 \hat{\mu}^{n+1}$, differs with $(\Phi_N^{n+1})^3-\Phi_N^n-\varepsilon^2 \Delta_h \Phi_N^{n+1}$ by an $O(\dt + h^2)$ truncation error, as stated in~\eqref{semi-CHDBC-3}. In comparison, such a construction is not needed to get the boundary chemical potential profile, due to the periodic boundary condition over the boundary sections.  Moreover, because of the higher order ``ghost" point extrapolation estimates~\eqref{velocity-expansion-1}, \eqref{mu-expansion-4} and \eqref{phi-expansion-4}, these boundary extrapolation formulas would not lead to a reduction of accuracy order for the truncation error; also see the related works~\cite{WL02, WLJ04} to deal with similar boundary extrapolation formulas.

\subsection{A rough error estimate}

Subtracting the numerical system \eqref{scheme-CHDBC-1}-\eqref{scheme-CHDBC-6} from the consistency estimate \eqref{semi-CHDBC-1}-\eqref{semi-CHDBC-8} gives
\begin{align}
&\frac{\hat{\tilde{\bu}}^{n+1}-\tilde{\bu}^n}{\dt} + b_h(\tilde{\bu}^n,\bU^{n+1}) + b_h(\bu^n,\hat{\tilde{\bu}}^{n+1}) + \nabla_h \tilde{p}^n - \nu \Delta_h \hat{\tilde{\bu}}^{n+1}   \nonumber  \\  
& \quad\quad\quad\quad\quad\quad\quad\quad\quad\quad\quad\quad\quad\quad\quad 
= - \gamma A_h\tilde{\phi}^n \nabla_h \mV^{n+1} - \gamma A_h\phi^n \nabla_h \tilde{\mu}^{n+1} +\tau^{n+1}_{\bu}, \label{error 1}   \\
& \frac{\tilde{\phi}^{n+1}-\tilde{\phi}^{n}}{\dt} + \nabla_h\cdot\left( A_h\tilde{\phi}^n\bU^{n+1} + A_h\phi^n\hat{\tilde{\bu}}^{n+1} \right)=\Delta_h \tilde{\mu}^{n+1} + \tau^{n+1}_{\phi}, \label{error 2}   \\
&\tilde{\mu}^{n+1}=(\Phi_N^{n+1})^3-(\phi^{n+1})^3-\tilde{\phi}^n-\varepsilon^2\Delta_h \tilde{\phi}^{n+1} 
+ \tau_\mu^{n+1} , \quad \partial_n \tilde{\mu}^{n+1} |_{y=0,1} =0 , \label{error 3} \\
& \tilde{\phi}^{n+1}_{\cdot,0} = \tilde{\phi}^{B,n+1} , \  \  \tilde{\phi}^{n+1}_{\cdot,N} = \tilde{\phi}^{T,n+1},  \label{error 4}\\
& \frac{\tilde{\phi}^{B,n+1} - \tilde{\phi}^{B,n}}{\dt} = - \tilde{\mu}_{B}^{n+1} + \tau^{n+1}_B,  \quad \frac{\tilde{\phi}^{T,n+1} - \tilde{\phi}^{T,n}}{\dt} = - \tilde{\mu}_{T}^{n+1} + \tau^{n+1}_T, \label{error 5} \\
& \tilde{\mu}_{B}^{n+1} = (\Phi_N^{B,n+1})^3-(\phi^{B,n+1})^3 - \tilde{\phi}^{B,n} 
- \varepsilon^2 (\tilde{D}_y \tilde{\phi}^{n+1})_{\cdot,0}, \label{error 6} \\
& \tilde{\mu}_{T}^{n+1} = (\Phi_N^{T,n+1})^3-(\phi^{T,n+1})^3 - \tilde{\phi}^{T,n} 
+ \varepsilon^2 (\tilde{D}_y \tilde{\phi}^{n+1})_{\cdot,N}, 	\label{error 7} \\
& \frac{\tilde{\bu}^{n+1}-\hat{\tilde{\bu}}^{n+1}}{\dt} + \nabla_h(\tilde{p}^{n+1}-\tilde{p}^n)=\tau^{n+1}_{p}, \quad \nabla_h \cdot \tilde{\bu}^{n+1} = 0  . \label{error 8}
\end{align}
The boundary condition holds that $\hat{\tilde{\bu}}^{n+1}  =\boldsymbol{0}$, $\tilde{\bu}^{n+1} \cdot \boldsymbol{n} =0$, on $\Gamma_y$. A discrete $W^{1,\infty}_h$ bound is assumed for the constructed approximate solutions:
\begin{equation} 
    \| \bU^{n+1} \|_{W^{1,\infty}_h},\ \| \mV^{n+1} \|_{W^{1,\infty}_h},\ \| \Phi_N^{n+1} \|_{W^{1,\infty}_h} \leq C^*, \quad \mbox{with} \, \, \, \| f \|_{W_h^{1,\infty}} := \| f \|_\infty + \| \nabla_h f \|_\infty . \label{W1 infty}
\end{equation}
In addition, we make the following a-priori assumption for the previous time step
\begin{equation}
   \|\tilde{\bu}^{n}\|_2,\ \dt \|\nabla_h\tilde{p}^{n}\|_2,\  \| \nabla_h\tilde{\phi}^n \|_2,\  \| \tilde{\phi}^{T,n} \|_{2,\Gamma},\  \| \tilde{\phi}^{B,n} \|_{2,\Gamma} \leq \dt^{\frac{7}{8}}+h^{\frac{15}{8}}. \label{a-priori assumption}
\end{equation}
Such an assumption will be recovered by the convergence analysis at the next time step, which will be illustrated later. In turn, this a-priori assumption leads to an $\ell^\infty$ bound, based on the inverse inequality \eqref{inverse ineq} with $\delta=\frac18$, and the linear refinement requirement $C_1 h \leq \dt \leq C_2 h$:
\begin{equation}
    \|\tilde{\phi}^n\|_\infty \leq C h^{-\frac18} \| \nabla_h \tilde{\phi}^n\|_2 \leq C(\dt^\frac34+h^\frac74) 
    \le \dt^\frac58 \leq 1.
\end{equation}
Before proceeding into the rough error estimate, a bound for $b_h(\bu,\bv)$ is needed.
\begin{lem} \cite{chen24b} \label{binary lemma}
   For the bilinear form $b_h(\bu,\bv)$ in \eqref{bilinear form}, the following inequality is valid:
    \begin{align}
        |\left( b_h(\bu,\bv),\bw \right)| \leq \frac{1}{2} \|\bu\|_2 \left( \|\nabla_h \bv\|_\infty \cdot \| \bw \|_2  + \| \nabla_h \bw \|_2 \cdot \| \bv\|_\infty \right).\label{binary lemma 2}
    \end{align}
\end{lem}
The proof is based on the summation by parts formula and the discrete H{\"o}lder inequality, and the details are skipped for the sake of brevity. The following lemma states a rough error estimate.
\begin{lem}\label{rough error estimate lemma}
   Based on the $W^{1,\infty}_h$ regularity assumption~\eqref{W1 infty} for $\mV^{n+1}$, as well as the a-priori assumption~\eqref{a-priori assumption}, a rough error estimate is valid for the evolutionary system \eqref{error 1}-\eqref{error 8}:
    \begin{equation}
       \|\nabla_h \tilde{\phi}^{n+1}\|_2^2+ \dt^{-1} ( \| \tilde{\phi}^{B,n+1}\|_{2,\Gamma}^2 + \| \tilde{\phi}^{T,n+1}\|_{2,\Gamma}^2 ) +\|\tilde{\bu}^{n+1}\|_2^2 \leq C(\dt^{\frac{3}{4}} + h^{\frac{11}{4}}), \label{rough error estimate conclusion}
    \end{equation}  
provided that the time step size $\dt$ and spatial mesh size $h$ are sufficiently small.
\end{lem}
\begin{proof}
 Taking a discrete inner product with \eqref{error 1} by $\hat{\tilde{\bu}}^{n+1}$ leads to
\begin{align}
    \frac{1}{2\dt}( &\| \hat{\tilde{\bu}}^{n+1} \|_2^2 - \| \tilde{\bu}^n \|_2^2 + \|\hat{\tilde{\bu}}^{n+1}-\tilde{\bu}^n\|_2^2 ) + \left( b_h(\tilde{\bu}^n,\bU^{n+1}) +  b_h(\bu^n,\hat{\tilde{\bu}}^{n+1}), \hat{\tilde{\bu}}^{n+1}\right)+ \nu\|\nabla_h \hat{\tilde{\bu}}^{n+1}\|_2^2 \nonumber \\
    &  = -\left( \nabla_h \tilde{p}^{n}, \hat{\tilde{\bu}}^{n+1}\right)-\gamma \left( \tilde{\phi}^n \nabla_h \mV^{n+1}, \hat{\tilde{\bu}}^{n+1}\right) -\gamma\left(\phi^n \nabla_h \tilde{\mu}^{n+1} , \hat{\tilde{\bu}}^{n+1}\right)  + (\tau_{\bu}^{n+1}, \hat{\tilde{\bu}}^{n+1}) , \label{rough error proof 1.0}
\end{align}
with an repeated application of summation-by-parts formulas. By Lemma \ref{binary lemma}, we get 
\begin{align}
    \left( b_h(\bu^n,\hat{\tilde{\bu}}^{n+1}), \hat{\tilde{\bu}}^{n+1}\right) = & 0, \label{rough error proof 1.1} \\
   \left( b_h(\tilde{\bu}^n,\bU^{n+1}), \hat{\tilde{\bu}}^{n+1} \right) \ge & -\frac{1}{2}\|\tilde{\bu}^n\|_2(\|\nabla_h \bU^{n+1} \|_{\infty} \cdot \|\hat{\tilde{\bu}}^{n+1}\|_2 
    + \|\nabla_h \hat{\tilde{\bu}}^{n+1}\|_2 \cdot \|\bU^{n+1}\|_{\infty} ) \nonumber \\
      \ge & -\frac{\nu}{4} \| \nabla_h\hat{\tilde{\bu}}^{n+1} \|_2^2 
    - \frac{C^*}{4} \| \hat{\tilde{\bu}}^{n+1} \|_2^2 - \tilde{C}_2 \|\tilde{\bu}^n\|_2^2 , \label{rough error proof 1.2}
\end{align}
with $\tilde{C}_2 =\frac{C^*}{4} + (C^*)^2 \nu^{-1}$, where the $W_h^{1,\infty}$ assumption \eqref{W1 infty} and the discrete Poincar{\'e} inequality, $\|\hat{\tilde{\bu}}^{n+1}\|_2 \leq C_0\|\nabla_h \hat{\tilde{\bu}}^{n+1}\|_2$, have been applied. In terms of the pressure gradient, we see that
\begin{align}
    & 
    ( \nabla_h \tilde{p}^n,  \hat{\tilde{\bu}}^{n+1} )= \left( \nabla_h \tilde{p}^n, \tilde{\bu}^{n+1} + \dt \nabla_h(\tilde{p}^{n+1}-\tilde{p}^n)-\dt \tau^{n+1}_p\right) \nonumber \\
    = & \frac{\dt}{2}(\|\nabla_h\tilde{p}^{n+1}\|_2^2-\|\nabla_h\tilde{p}^n\|_2^2-\|\nabla_h(\tilde{p}^{n+1}-\tilde{p}^n)\|_2^2) - \dt (\nabla_h \tilde{p}^n, \tau^{n+1}_p) \nonumber \\
    \ge &  \frac{\dt}{2}(\|\nabla_h\tilde{p}^{n+1}\|_2^2 - \|\nabla_h\tilde{p}^n\|_2^2-\|\nabla_h(\tilde{p}^{n+1}-\tilde{p}^n)\|_2^2) - \frac{\dt^2}{2} \|\nabla_h\tilde{p}^n\|_2^2 - \frac12 \|\tau^{n+1}_p\|_2^2 , \label{rough error proof 1.3}
\end{align}
in which the identity $( \nabla_h \tilde{p}^n, \tilde{\bu}^{n+1} )=0$ has been used in the derivation, which comes from the fact that $\nabla_h \cdot \tilde{\bu}^{n+1} =0$, $( \tilde{\bu}^{n+1} \cdot \n )_\Gamma=0$. Regarding the second term on the right hand side of \eqref{rough error proof 1.0}, an application of discrete H{\"o}lder inequality implies that
 \begin{align}
     -\gamma ( \tilde{\phi}^n \nabla_h \mV^{n+1}, \hat{\tilde{\bu}}^{n+1} ) \le &  \gamma \|\nabla_h \mV^{n+1}\|_\infty 
     \cdot \|\tilde{\phi}^n\|_2 \cdot \|\hat{\tilde{\bu}}^{n+1}\|_2  \nonumber \\ 
     \le & C^* \gamma \|\tilde{\phi}^n\|_2 \cdot \|\hat{\tilde{\bu}}^{n+1}\|_2 
     \le \frac{C^* \gamma}{2} ( \|\tilde{\phi}^n\|_2^2 + \| \hat{\tilde{\bu}}^{n+1}\|_2^2 ) . 
     \label{rough error proof 1.4}
 \end{align}
The local truncation error term could be controlled in a straightforward manner
 \begin{align}
     (\tau_{\bu}^{n+1}, \hat{\tilde{\bu}}^{n+1}) \leq \|\tau_{\bu}^{n+1}\|_2 
     \cdot \| \hat{\tilde{\bu}}^{n+1}\|_2 \leq \frac12 ( \| \hat{\tilde{\bu}}^{n+1}\|_2^2 
     + \|\tau_{\bu}^{n+1}\|_2^2 ) . \label{rough error proof 1.5}
 \end{align} 
Meanwhile, a discrete inner product with \eqref{error 8} by $2 \tilde{\bu}^{n+1}$ reveals that
\begin{align}
    & \|\tilde{\bu}^{n+1}\|_2^2-\|\hat{\tilde{\bu}}^{n+1}\|_2^2+\|\tilde{\bu}^{n+1}-\hat{\tilde{\bu}}^{n+1}\|_2^2=2\dt(\tau_p^{n+1}, \tilde{\bu}^{n+1}) = 0 , \quad \mbox{so that} \nonumber \\
    & \|\hat{\tilde{\bu}}^{n+1}\|_2^2 = \|\tilde{\bu}^{n+1}\|_2^2 + \dt^2 \|\nabla_h(\tilde{p}^{n+1}-\tilde{p}^{n})+\tau_p^{n+1}\|_2^2  \nonumber \\
    & \quad\quad\quad\quad \geq \|\tilde{\bu}^{n+1}\|_2^2 + \dt^2 \|\nabla_h(\tilde{p}^{n+1}-\tilde{p}^{n})\|_2^2 
    - \frac{\dt^3}{2} \|\nabla_h(\tilde{p}^{n+1}-\tilde{p}^{n})\|_2^2 - \frac{\dt}{2} \|\tau_p^{n+1}\|_2^2  . 
    \label{rough error proof 1.6}
\end{align} 
In particular, notice that an identity $(\tau_p^{n+1}, \tilde{\bu}^{n+1}) =0$ has been applied in the first step, based on the fact that $\tau_p^{n+1} = \nabla_h ( P^{n+1} - P^n)$. A combination of \eqref{rough error proof 1.0}-\eqref{rough error proof 1.6} results in
\begin{align}
  &
    \frac{1}{2\dt}( \|\tilde{\bu}^{n+1}\|_2^2-\|\tilde{\bu}^n\|_2^2 + \| \hat{\tilde{\bu}}^{n+1} - \tilde{\bu}^n\|_2^2) 
    +\frac{\dt}{2}(\|\nabla_h \tilde{p}^{n+1}\|_2^2 - \|\nabla_h \tilde{p}^n\|_2^2) 
    + \frac{3 \nu}{4}\|\nabla_h \hat{\tilde{\bu}}^{n+1}\|_2^2  \nonumber \\
   \le & - \gamma (A_h\phi^n \nabla_h \tilde{\mu}^{n+1} , \hat{\tilde{\bu}}^{n+1} ) 
   + \frac12 \|\tau_{\bu}^{n+1}\|_2^2 + \tilde{C}_2 \|\tilde{\bu}^n\|_2^2 
   + \tilde{C}_3 \| \hat{\tilde{\bu}}^{n+1}\|_2^2 + \frac{C^* \gamma}{2} \|\tilde{\phi}^n\|_2^2 \nonumber \\ 
  & 
  + \frac{\dt^2}{2} \|\nabla_h\tilde{p}^{n+1} \|_2^2 
  + \dt^2 \|\nabla_h\tilde{p}^n \|_2^2 + \frac34 \|\tau^{n+1}_p\|_2^2 , \quad 
   \tilde{C}_3 = \frac{C^*}{4} + \frac{C^* \gamma}{2} + \frac12.  \label{rough error proof 1.7}
\end{align}
On the other hand, taking a discrete inner product with \eqref{error 2} by $\tilde{\mu}^{n+1}$ gives 
\begin{align}
    \frac{1}{\dt}(\tilde{\phi}^{n+1},\tilde{\mu}^{n+1})&+\|\nabla_h \tilde{\mu}^{n+1}\|_2^2-\left( A_h \phi^n \nabla_h \tilde{\mu}^{n+1}, \hat{\tilde{\bu}}^{n+1} \right) \nonumber \\ 
    = & ( A_h \tilde{\phi}^{n} \nabla_h \tilde{\mu}^{n+1}, \bU^{n+1} ) 
      + (\tau_{\phi}^{n+1},\tilde{\mu}^{n+1})
      +\frac{1}{\dt}(\tilde{\phi}^n,\tilde{\mu}^{n+1}) , \label{rough error proof 2.0} 
\end{align}
in which the homogeneous Neumann boundary condition for $\tilde{\mu}^{n+1}$ has been used. The bounds of the right hand side terms are straightforward:
\begin{align}
    & ( A_h \tilde{\phi}^{n} \nabla_h \tilde{\mu}^{n+1}, \bU^{n+1} ) \leq C^*\|\tilde{\phi}^{n}\|_2 \cdot \|\nabla_h \tilde{\mu}^{n+1}\|_2 \leq \frac{1}{4}\|\nabla_h \tilde{\mu}^{n+1}\|_2^2 + (C^*)^2\|\tilde{\phi}^{n}\|_2^2, 
    \label{rough error proof 2.1}  \\
    & (\tau_{\phi}^{n+1},\tilde{\mu}^{n+1}) \leq \|\tau_{\phi}^{n+1}\|_{-1, h} \cdot \| \nabla_h \tilde{\mu}^{n+1}\|_2 
     \le  \frac{1}{4}\|\nabla_h \tilde{\mu}^{n+1}\|_2^2 + \|\tau_{\phi}^{n+1}\|_{-1, h}^2, \label{rough error proof 2.2}  \\
    & \frac{1}{\dt}(\tilde{\phi}^n,\tilde{\mu}^{n+1}) \leq  \|\tilde{\phi}^n\|_{-1, h} \cdot \|\nabla_h\tilde{\mu}^{n+1}\|_2 \leq \frac{1}{4}\|\nabla_h \tilde{\mu}^{n+1}\|_2^2 + \dt^{-2} \|\tilde{\phi}^n\|_{-1, h}^2 , \label{rough error proof 2.3}  
\end{align}
in which the mean-zero property for $\tilde{\phi}^{n}$ and $\tau_{\phi}^{n+1}$ is also mean-zero has been used in the derivation. 
The analysis for the $(\tilde{\phi}^{n+1},\tilde{\mu}^{n+1})$ turns out to be more complicated. Based on the convexity estimate and discrete H{\"o}lder inequality, we see that 
\begin{align}
  & 
   \big(\tilde{\phi}^{n+1},(\Phi_N^{n+1})^3-(\phi^{n+1})^3 \big) \ge 0 , \quad 
   \big(\tilde{\phi}^{n+1},  \tau_\mu^{n+1} \big)  
   \ge - \varepsilon^2 \|\nabla_h \tilde{\phi}^{n+1} \|_2^2 /4 
   - C_0^2 \varepsilon^{-2} \| \tau_\mu^{n+1} \|_2^2 ,  \nonumber 
\\
  & 
 (  \tilde{\phi}^{n+1}, -\tilde{\phi}^n ) \ge -\|\tilde{\phi}^n\|_{-1, h} \cdot \| \nabla_h \tilde{\phi}^{n+1}\|_2 
 \ge - \varepsilon^{-2} \|\tilde{\phi}^n\|_{-1, h}^2 + \varepsilon^2 \|\nabla_h \tilde{\phi}^{n+1} \|_2^2 /4 , 
    \label{rough error proof 2.4}  
\\
  & \big(\tilde{\phi}^{n+1}, - \Delta_h \tilde{\phi}^{n+1} + \tau_\mu^{n+1} \big) 
  = \frac12 \|\nabla_h \tilde{\phi}^{n+1} \|_2^2 - ( \tilde{\phi}^{T,n+1}, \tilde{D}_y \tilde{\phi}^{n+1}_{\cdot,N} )_\Gamma + ( \tilde{\phi}^{B,n+1},  \tilde{D}_y \tilde{\phi}^{n+1}_{\cdot,0} )_\Gamma , \nonumber  
\\
    & 
    ( \tilde{\phi}^{n+1} , \tilde{\mu}^{n+1} ) 
     = \big(\tilde{\phi}^{n+1},(\Phi_N^{n+1})^3-(\phi^{n+1})^3 -\tilde{\phi}^n - \Delta_h \tilde{\phi}^{n+1} + \tau_\mu^{n+1} \big) \nonumber 
 \\
    \ge &  - C_0^2 \varepsilon^{-2} ( \| \tilde{\phi}^n\|_2^2 + \| \tau_\mu^{n+1} \|_2^2 ) 
    + \frac{\varepsilon^2}{2} \|\nabla_h \tilde{\phi}^{n+1} \|_2^2 - ( \tilde{\phi}^{T,n+1}, \varepsilon^2 \tilde{D}_y \tilde{\phi}^{n+1}_{\cdot,N} )_\Gamma + ( \tilde{\phi}^{B,n+1}, \varepsilon^2 \tilde{D}_y \tilde{\phi}^{n+1}_{\cdot,0} )_\Gamma .     \nonumber 
\end{align}
In turn, a combination of \eqref{rough error proof 2.0}-\eqref{rough error proof 2.4} leads to
\begin{align}
    & \frac{\varepsilon^2}{2} \|\nabla_h \tilde{\phi}^{n+1} \|_2^2+\frac{\dt}{4}\|\nabla_h \tilde{\mu}^{n+1}\|_2^2-\dt\left( A_h \phi^n \nabla_h \tilde{\mu}^{n+1}, \hat{\tilde{\bu}}^{n+1} \right) 
    - C_0^2 \varepsilon^{-2}  \| \tau_\mu^{n+1} \|_2^2  
    \label{rough error proof 2.41}   \\
    \le &  \Big( \dt(C^*)^2 + \frac{C_0^2}{\dt} + \frac{C_0^2}{\varepsilon^2} \Big)  \|\tilde{\phi}^{n}\|_2^2 
    + \dt C_0 ^2 \|\tau_{\phi}^{n+1}\|_2^2 + ( \tilde{\phi}^{T,n+1}, \varepsilon^2 \tilde{D}_y \tilde{\phi}^{n+1}_{\cdot,N} )_\Gamma - ( \tilde{\phi}^{B,n+1}, \varepsilon^2 \tilde{D}_y \tilde{\phi}^{n+1}_{\cdot,0})_\Gamma , \nonumber
\end{align} 
in which the discrete inequality, $\| f \|_{-1, h} \le C_0 \| f \|_2$ (for $f$ with $\overline{f} =0$), has been applied. The evaluation of boundary terms relies on the evolutionary equations~\eqref{error 5}-\eqref{error 7}. In more details, taking a discrete inner production with \eqref{error 5} by $\tilde{\mu}_{B}^{n+1}$ implies that 
\begin{equation} 
\begin{aligned} 
& ( \tilde{\phi}^{B,n+1}, \varepsilon^2 \tilde{D}_y \tilde{\phi}^{n+1}_{\cdot,0})_\Gamma  
  = \dt^{-1} ( \tilde{\phi}^{B,n+1}, \tilde{\phi}^{B,n+1} - \tilde{\phi}^{B,n} )_\Gamma 
  + ( \tilde{\phi}^{B,n+1},  (\Phi_N^{B,n+1})^3-(\phi^{B,n+1})^3 )_\Gamma \\ 
&   \qquad \qquad \qquad \qquad \qquad 
  - ( \tilde{\phi}^{B,n+1}, \tilde{\phi}^{B,n} )_\Gamma 
    - ( \tilde{\phi}^{B,n+1}, \tau^{n+1}_B )_\Gamma , \\ 
& ( \tilde{\phi}^{B,n+1}, \tilde{\phi}^{B,n+1} - \tilde{\phi}^{B,n} )_\Gamma  
 \ge \frac12 ( \| \tilde{\phi}^{B,n+1} \|_{2, \Gamma}^2 - \| \tilde{\phi}^{B,n} \|_{2, \Gamma}^2 ) ,  \\ 
& ( \tilde{\phi}^{B,n+1},  (\Phi_N^{B,n+1})^3-(\phi^{B,n+1})^3 )_\Gamma \ge 0 , \quad  
  - ( \tilde{\phi}^{B,n+1}, \tilde{\phi}^{B,n} )_\Gamma 
   \ge - \frac12 ( \| \tilde{\phi}^{B,n+1} \|_{2, \Gamma}^2 + \| \tilde{\phi}^{B,n} \|_{2, \Gamma}^2 ) , \\ 
& - ( \tilde{\phi}^{B,n+1}, \tau^{n+1}_B )_\Gamma 
  \ge - \frac12 ( \| \tilde{\phi}^{B,n+1} \|_{2, \Gamma}^2  + \| \tau^{n+1}_B \|_{2, \Gamma}^2 ) ,   
  \quad \mbox{so that}  \\
& ( \tilde{\phi}^{B,n+1}, \varepsilon^2 \tilde{D}_y \tilde{\phi}^{n+1}_{\cdot,0})_\Gamma   
\ge ( \frac{1}{2 \dt} -1 )  \| \tilde{\phi}^{B,n+1} \|_{2, \Gamma}^2 
- ( \frac{1}{2 \dt} + 1) \| \tilde{\phi}^{B,n} \|_{2, \Gamma}^2  
  - \frac12 \| \tau^{n+1}_B \|_{2, \Gamma}^2 .  
\end{aligned} 
  \label{rough error proof 2.5}
\end{equation} 
A similar bound could be obtained on the top boundary section: 
\begin{equation}
   - ( \tilde{\phi}^{T,n+1}, \varepsilon^2 \tilde{D}_y \tilde{\phi}^{n+1}_{\cdot,N})_\Gamma   
\ge ( \frac{1}{2 \dt} -1 )  \| \tilde{\phi}^{T,n+1} \|_{2, \Gamma}^2 
- ( \frac{1}{2 \dt} + 1) \| \tilde{\phi}^{T,n} \|_{2, \Gamma}^2  
  - \frac12 \| \tau^{n+1}_T \|_{2, \Gamma}^2 . \label{rough error proof 2.6} 
\end{equation}
In turn, a substitution of~\eqref{rough error proof 2.5} and \eqref{rough error proof 2.6} into \eqref{rough error proof 2.41} results in 
\begin{align}
  & 
     \frac{\varepsilon^2}{2} \|\nabla_h \tilde{\phi}^{n+1} \|_2^2 + \frac{\dt^{-1}}{4}(\| \tilde{\phi}^{B,n+1}\|_{2,\Gamma}^2 + \| \tilde{\phi}^{T,n+1}\|_{2,\Gamma}^2 ) -\dt\left( A_h \phi^n \nabla_h \tilde{\mu}^{n+1}, \hat{\tilde{\bu}}^{n+1} \right) - C_0^2 \varepsilon^{-2}  \| \tau_\mu^{n+1} \|_2^2   \nonumber \\
    \le &  C \dt^{-1} (\|\tilde{\phi}^{n}\|_2^2 +\|\tilde{\phi}^{B,n}\|_{2,\Gamma}^2+\|\tilde{\phi}^{T,n}\|_{2,\Gamma}^2) 
    + C_0^2 \dt \|\tau_{\phi}^{n+1}\|_2^2+\frac12 ( \|\tau_B^{n+1}\|_{2,\Gamma}^2+\|\tau_T^{n+1}\|_{2,\Gamma}^2),
    \label{rough error proof 2.9} 
\end{align}
where $C$ is independent on $\dt$ and $h$. Furthermore, a combination of \eqref{rough error proof 1.7}  and \eqref{rough error proof 2.9} states that 
\begin{align}
  & 
     \frac{\varepsilon^2}{2} \|\nabla_h \tilde{\phi}^{n+1} \|_2^2 + \frac{\dt^{-1}}{4}(\| \tilde{\phi}^{B,n+1}\|_{2,\Gamma}^2 + \| \tilde{\phi}^{T,n+1}\|_{2,\Gamma}^2 ) + \frac{1}{2 \gamma}\|\tilde{\bu}^{n+1}\|_2^2 + \frac{\dt^2}{4 \gamma}\|\nabla_h \tilde{p}^{n+1}\|_2^2 \nonumber \\
     \le &  C \dt^{-1} (\|\tilde{\phi}^{n}\|_2^2 +\|\tilde{\phi}^{B,n}\|_{2,\Gamma}^2+\|\tilde{\phi}^{T,n}\|_{2,\Gamma}^2) 
     + C_0^2 \dt \|\tau_{\phi}^{n+1}\|_2^2+ \frac12 ( \|\tau_B^{n+1}\|_{2,\Gamma}^2+\|\tau_T^{n+1}\|_{2,\Gamma}^2) \label{rough error proof 2.10}  \\
    & + \frac{1}{2 \gamma}\|\tilde{\bu}^{n}\|_2^2 + \frac{\dt^2}{\gamma}\|\nabla_h \tilde{p}^{n}\|_2^2 + C \dt (\|\tau_{\bu}^{n+1}\|_2^2 + \|\tilde{\bu}^n\|_2^2 + \|\tilde{\phi}^n\|_2^2) + \frac{3\dt}{4 \gamma}\|\tau_{p}^{n+1}\|_2^2 + C_0^2 \varepsilon^{-2}  \| \tau_\mu^{n+1} \|_2^2  \nonumber , 
\end{align}
in which the Cauchy inequality: $\| \hat{\tilde{\bu}}^{n+1}\|_2^2 \le 2 ( \| \tilde{\bu}^n \|_2^2 + \| \hat{\tilde{\bu}}^{n+1} - \tilde{\bu}^n \|_2^2)$, has been applied in the derivation. Subsequently, based on the a-priori assumption \eqref{a-priori assumption}, combined with the linear refinement requirement $C_1 h \leq \dt \leq C_2 h$, the dominant order rough error estimate becomes available: 
\begin{align}
    \|\nabla_h \tilde{\phi}^{n+1}\|_2^2+ \dt^{-1} ( \| \tilde{\phi}^{B,n+1}\|_{2,\Gamma}^2 + \| \tilde{\phi}^{T,n+1}\|_{2,\Gamma}^2 ) +\|\tilde{\bu}^{n+1}\|_2^2 \leq C(\dt^{\frac{3}{4}} + h^{\frac{11}{4}}). \label{rough error proof 2.11} 
\end{align}
This finishes the proof of Lemma~\ref{rough error estimate lemma}. 
\end{proof}

As a consequence of the rough error estimate~\eqref{rough error proof 2.11}, an application of Lemma \ref{inverse inequality lemma} implies that
\begin{align} 
  & 
   \|\tilde{\phi}^{n+1}\|_\infty \le  C h^{-\frac18} \|\nabla_h \tilde{\phi}^{n+1}\|_2 \leq C(\dt^{\frac14} + h^{\frac54}) \le 1 , \label{rough error proof 3.1} \\
   & \| \nabla_h \tilde{\phi}^{n+1}\|_3 \le C h^{-\frac13} \|\nabla_h \tilde{\phi}^{n+1}\|_2 
  \le C(\dt^{\frac{1}{24}} + h^{\frac{25}{24}}) \le 1 . \label{rough error proof 3.2}  
\end{align}
In turn, an $\ell^\infty \cap W_h^{1,3}$ bound for the numerical solution becomes available: 
\begin{align} 
  & 
   \| \phi^{n+1} \|_\infty \le \| \Phi_N^{n+1} \|_\infty + \| \tilde{\phi}^{n+1}\|_\infty  
   \le C^* + 1 := \tilde{C}_1 , \label{rough error proof 3.4} \\
   & \| \nabla_h \phi^{n+1} \|_3 \le \| \nabla_h \Phi_N^{n+1} \|_\infty 
   + \| \nabla_h \tilde{\phi}^{n+1}\|_3 \le C^* + 1 := \tilde{C}_1 . \label{rough error proof 3.5} 
\end{align}

\subsection{A refined error estimate}
Before proceeding into the refined error estimate, the following preliminary result is necessary. 
\begin{lem} \label{nonlinear lemma}
    Based on the functional bounds~\eqref{W1 infty}, \eqref{rough error proof 3.4}-\eqref{rough error proof 3.5}, for the constructed and numerical solutions, the following bound is available to the nonlinear error term: 
    \begin{align}
    & \|\nabla_h( (\Phi_N^{n+1})^3-(\phi^{n+1})^3 )\|_2 \le \tilde{C}_4 \|\nabla_h\tilde{\phi}^{n+1}\|_2,  
    \label{refined nonlinear 1}   \\
     & \| (\Phi_N^{\alpha,n+1})^3-(\phi^{\alpha,n+1})^3 \|_{2,\Gamma} \le \tilde{C}_5 \| \tilde{\phi}^{\alpha,n+1}\|_{2,\Gamma} , \quad  \alpha = B, T , \label{refined nonlinear 2} 
    \end{align}
    in which constants $\tilde{C}_4$ and $\tilde{C}_5$ are independent on $\dt$ and $h$.
\end{lem}
\begin{proof}
Denote $\mN^{n+1} := (\Phi_N^{n+1})^2+\Phi_N^{n+1}\phi^{n+1}+(\phi^{n+1})^2$. A careful calculation indicates the product rule for discrete gradient 
\begin{equation}
    \nabla_h( (\Phi_N^{n+1})^3-(\phi^{n+1})^3 )=\nabla_h( \tilde{\phi}^{n+1}  \mN^{n+1} ) = A_h\tilde{\phi}^{n+1} \nabla_h \mN^{n+1} + A_h \mN^{n+1} \nabla_h \tilde{\phi}^{n+1} . \label{refined nonlinear 2.1} 
\end{equation}
Meanwhile, the following bounds come from the regularity assumption~\eqref{W1 infty} and the derived estimates \eqref{rough error proof 3.4}-\eqref{rough error proof 3.5}: 
\begin{align}
  & 
     \| A_h \mN^{n+1} \|_\infty \le \frac32 ( \| \Phi_N^{n+1} \|_\infty^2+ \| \phi^{n+1} \|_\infty^2 ) 
     \le \frac32 ( (C^*)^2 + \tilde{C}_1^2) \le 3 \tilde{C}_1^2 , \label{refined nonlinear 2.4}
\\
  & 
     \| \nabla_h \mN^{n+1} \|_3 \le 3 \max ( \| \Phi_N^{n+1} \|_\infty , \| \phi^{n+1} \|_\infty )  
     \cdot \max ( \| \nabla_h \Phi_N^{n+1} \|_3 , \| \nabla_h \phi^{n+1} \|_3 ) \le 3 \tilde{C}_1^2 . 
     \label{refined nonlinear 2.5}
\end{align}
In turn, we arrive at 
\begin{equation}
\begin{aligned} 
  & 
   \| \nabla_h( (\Phi_N^{n+1})^3-(\phi^{n+1})^3 )\|_2 
    \le \|  A_h\tilde{\phi}^{n+1} \|_6 \cdot \| \nabla_h \mN^{n+1} \|_3 
    + \| A_h \mN^{n+1} \|_\infty \cdot \| \nabla_h \tilde{\phi}^{n+1} \|_2  
 \\
   \le & 
     (C_1+ 1) ( \| \nabla_h \mN^{n+1} \|_3 + \| \mN^{n+1} \|_\infty ) \| \nabla_h \tilde{\phi}^{n+1} \|_2 
     \le 3 \tilde{C}_1^2 (C_1+ 1)  \| \nabla_h \tilde{\phi}^{n+1} \|_2 , 
\end{aligned} 
  \label{refined nonlinear 2.6} 
\end{equation}
in which the Sobolev inequality \eqref{inverse ineq} (in Lemma~\ref{inverse inequality lemma}) has been applied in the second step. Therefore, \eqref{refined nonlinear 1} has been proved, by taking $\tilde{C}_4 = 3 \tilde{C}_1^2 (C_1+ 1)$. The boundary nonlinear estimate~\eqref{refined nonlinear 2} could be similarly derived.
\end{proof}
Now we carry out the refined error estimate. Taking a discrete inner product with \eqref{error 2} by $-\Delta_h \tilde{\phi}^{n+1}$ leads to
\begin{align}
&\frac{1}{2\dt} ( \|\nabla_h \tilde{\phi}^{n+1}\|_2^2 -\|\nabla_h \tilde{\phi}^n\|_2^2 ) - \Big( \frac{\tilde{\phi}^{T,n+1}-\tilde{\phi}^{T,n}}{\dt}, \tilde{D}_y \tilde{\phi}^{n+1}_{\cdot,N} \Big)_\Gamma + \Big( \frac{\tilde{\phi}^{B,n+1}-\tilde{\phi}^{B,n}}{\dt}, \tilde{D}_y \tilde{\phi}^{n+1}_{\cdot,0} \Big)_\Gamma \nonumber \\
& \leq -\left(  A_h\tilde{\phi}^n\bU^{n+1} + A_h\phi^n\hat{\tilde{\bu}}^{n+1}, \nabla_h\Delta_h \tilde{\phi}^{n+1}\right) + (\nabla_h \tilde{\mu}^{n+1}, \nabla_h \Delta_h \tilde{\phi}^{n+1}) - (\tau^{n+1}_\phi , \Delta_h \tilde{\phi}^{n+1}) , \label{refined nonlinear 3.0}
\end{align}
in which the homogeneous Neumann boundary condition for $\tilde{\mu}^{n+1}$ has been applied. The boundary part could be evaluated as 
\begin{align}
    & - \Big( \frac{\tilde{\phi}^{T,n+1}-\tilde{\phi}^{T,n}}{\dt}, \tilde{D}_y \tilde{\phi}^{n+1}_{\cdot,N} \Big)_\Gamma 
   = \dt^{-2} \| \tilde{\phi}^{T,n+1} - \tilde{\phi}^{T,n} \|_{2, \Gamma}^2  
   - \dt^{-1} ( \tilde{\phi}^{T,n+1} - \tilde{\phi}^{T,n}, \tilde{\phi}^{T,n} )_\Gamma \nonumber \\ 
     &   \qquad \qquad   
  + \dt^{-1} ( \tilde{\phi}^{T,n+1} - \tilde{\phi}^{T,n} ,  (\Phi_N^{T,n+1})^3-(\phi^{T,n+1})^3 )_\Gamma 
    - \dt^{-1} ( \tilde{\phi}^{T,n+1} - \tilde{\phi}^{T,n} , \tau^{n+1}_T )_\Gamma , \nonumber \\
    & - ( \tilde{\phi}^{T,n+1} - \tilde{\phi}^{T,n}, \tilde{\phi}^{T,n} )_\Gamma  
    \ge - \frac{\dt^{-1}}{4} \| \tilde{\phi}^{T,n+1} - \tilde{\phi}^{T,n} \|_{2, \Gamma}^2 
    -  \dt \| \tilde{\phi}^{T,n} \|_{2, \Gamma}^2 , \nonumber\\ 
    & ( \tilde{\phi}^{T,n+1} - \tilde{\phi}^{T,n} ,  (\Phi_N^{T,n+1})^3-(\phi^{T,n+1})^3 )_\Gamma \nonumber \\ 
    & \quad \ge - \frac{\dt^{-1}}{4} \| \tilde{\phi}^{T,n+1} - \tilde{\phi}^{T,n} \|_{2, \Gamma}^2 
    - \dt \| ( \Phi_N^{T,n+1})^3-(\phi^{T,n+1})^3 \|_{2, \Gamma}^2 \label{refined nonlinear 3.1} \\ 
    & \quad \ge  - \frac{\dt^{-1}}{4} \| \tilde{\phi}^{T,n+1} - \tilde{\phi}^{T,n} \|_{2, \Gamma}^2 
    - \tilde{C}_5^2 \dt \|  \tilde{\phi}^{T,n+1} \|_{2, \Gamma}^2 ,  \quad 
    \mbox{(by~\eqref{refined nonlinear 2})} , \nonumber  \\
    & - ( \tilde{\phi}^{T,n+1} - \tilde{\phi}^{T,n}, \tau^{n+1}_T )_\Gamma  
    \ge - \frac{\dt^{-1}}{4} \| \tilde{\phi}^{T,n+1} - \tilde{\phi}^{T,n} \|_{2, \Gamma}^2 
    -  \dt \| \tau^{n+1}_T \|_{2, \Gamma}^2 ,  \quad \mbox{so that} \nonumber \\ 
    & - \Big( \frac{\tilde{\phi}^{T,n+1}-\tilde{\phi}^{T,n}}{\dt}, \tilde{D}_y \tilde{\phi}^{n+1}_{\cdot,N} \Big)_\Gamma 
   \ge \frac{\dt^{-2}}{2} \| \tilde{\phi}^{T,n+1} - \tilde{\phi}^{T,n} \|_{2, \Gamma}^2  
   -  \| \tilde{\phi}^{T,n} \|_{2, \Gamma}^2 -  \tilde{C}_5^2 \| \tilde{\phi}^{T,n+1} \|_{2, \Gamma}^2 - \| \tau^{n+1}_T \|_{2, \Gamma}^2 .\nonumber 
\end{align}

The following boundary error quantities are introduced to facilitate the convergence analysis:  
\begin{equation} 
  G^{\alpha, \ell} = \sum_{k=0}^{\ell-1} \dt^{-1} \| \tilde{\phi}^{\alpha,k+1} - \tilde{\phi}^{\alpha,k} \|_{2, \Gamma}^2  
  = \dt \sum_{k=0}^{\ell-1} \| \frac{\tilde{\phi}^{\alpha,k+1} - \tilde{\phi}^{\alpha,k}}{\dt} \|_{2, \Gamma}^2 , 
  \quad \alpha = B, T. 
  \label{quantity-G-1} 
\end{equation} 
Meanwhile, based on the fact that $\tilde{\phi}^{\alpha,0} \equiv 0$, for $\alpha = B, T$, the following estimate is observed: 
\begin{equation} 
  \tilde{\phi}^{\alpha,\ell} =  \sum_{k=0}^{\ell-1} ( \tilde{\phi}^{\alpha,k+1} - \tilde{\phi}^{\alpha,k} ) , \quad 
  \mbox{so that} \, \, \, 
  \| \tilde{\phi}^{\alpha,\ell} \|_{2, \Gamma}^2 \le \ell \sum_{k=0}^{\ell-1} 
  \| \tilde{\phi}^{\alpha,k+1} - \tilde{\phi}^{\alpha,k} \|_{2, \Gamma}^2 
  \le T G^{\alpha, \ell} . 
  \label{refined nonlinear 3.2} 
\end{equation}  
In turn, a substitution of~\eqref{quantity-G-1} and \eqref{refined nonlinear 3.2} into \eqref{refined nonlinear 3.1} indicates that 
\begin{equation} 
    - \Big( \frac{\tilde{\phi}^{T,n+1}-\tilde{\phi}^{T,n}}{\dt}, \tilde{D}_y \tilde{\phi}^{n+1}_{\cdot,N} \Big)_\Gamma 
   \ge \frac{1}{2 \dt} ( G^{T, n+1} - G^{T, n} )  - T ( G^{T, n} + \tilde{C}_5^2 G^{T, n+1})
    - \| \tau^{n+1}_T \|_{2, \Gamma}^2 . 
   \label{refined nonlinear 3.3} 
\end{equation} 
The bottom boundary estimate could be similarly derived: 
\begin{equation}
    \Big( \frac{\tilde{\phi}^{B,n+1}-\tilde{\phi}^{B,n}}{\dt}, \tilde{D}_y \tilde{\phi}^{n+1}_{\cdot,0} \Big)_\Gamma 
    \ge \frac{1}{2 \dt} ( G^{B, n+1} - G^{B, n} )  - T ( G^{B, n} + \tilde{C}_5^2 G^{B, n+1})
    - \| \tau^{n+1}_B \|_{2, \Gamma}^2 . \label{refined nonlinear 3.6} 
\end{equation}
For the right hand side terms of \eqref{refined nonlinear 3.0}, the following bounds are valid:
\begin{align}
    & - (\tau^{n+1}_\phi , \Delta_h \tilde{\phi}^{n+1}) \leq \|\tau^{n+1}_\phi\|_{-1,h} \cdot \|\nabla_h\Delta_h \tilde{\phi}^{n+1} \|_2 \leq \frac{\varepsilon^2}{4} \|\nabla_h \Delta_h \tilde{\phi}^{n+1}\|_2^2 
    + C_0^2 \varepsilon^{-2} \|\tau^{n+1}_\phi\|_2^2,    \label{refined nonlinear 3.7} \\
    & (\nabla_h \tilde{\mu}^{n+1}, \nabla_h \Delta_h \tilde{\phi}^{n+1}) = -\varepsilon^2 \|\nabla_h \Delta_h \tilde{\phi}^{n+1}\|_2^2 + \left( \nabla_h ( (\Phi_N^{n+1})^3 - (\phi^{n+1})^3 ) -\nabla_h \tilde{\phi}^{n} 
    + \nabla_h \tau_\mu^{n+1} ,  \nabla_h \Delta_h \tilde{\phi}^{n+1}\right)   \nonumber \\
    & \quad\quad\quad \leq -\varepsilon^2 \|\nabla_h \Delta_h \tilde{\phi}^{n+1}\|_2^2 + \|\nabla_h \Delta_h \tilde{\phi}^{n+1}\|_2 (\tilde{C}_4 \|\nabla_h \tilde{\phi}^{n+1}\|_2+\|\nabla_h \tilde{\phi}^{n}\|_2 
    + \| \nabla_h \tau_\mu^{n+1} \|_2 ) \nonumber \\
    & \quad\quad\quad \leq -\frac{\varepsilon^2}{2} \|\nabla_h \Delta_h \tilde{\phi}^{n+1}\|_2^2 
    + \varepsilon^{-2} ( \tilde{C}_4^2 \|\nabla_h \tilde{\phi}^{n+1}\|_2^2+ 2 \|\nabla_h \tilde{\phi}^{n}\|_2^2 
    + 2 \| \nabla_h \tau_\mu^{n+1} \|_2^2 ), \label{refined nonlinear 3.8} \\
    & - (  A_h\tilde{\phi}^n\bU^{n+1}, \nabla_h\Delta_h \tilde{\phi}^{n+1} ) \leq \|\bU^{n+1}\|_\infty 
    \cdot \| \tilde{\phi}^n \|_2 \cdot \| \nabla_h\Delta_h \tilde{\phi}^{n+1} \|_2 \leq C^*C_0 \|\nabla_h \tilde{\phi}^n \|_2 
    \cdot \| \nabla_h\Delta_h \tilde{\phi}^{n+1} \|_2  \nonumber \\
    & \quad\quad\quad \leq \frac{\varepsilon^2}{8}\| \nabla_h\Delta_h \tilde{\phi}^{n+1} \|_2^2 
    + 2 C_0^2 ( C^*)^2 \varepsilon^{-2} \|\nabla_h \tilde{\phi}^n \|_2^2.  \label{refined nonlinear 3.9}
\end{align}
The last velocity term could be linked to \eqref{rough error proof 1.7} from the rough error estimate, based on the $\ell^\infty$ bound of the numerical solution:
\begin{align}
    -& (A_h\phi^n \nabla_h \tilde{\mu}^{n+1} , \hat{\tilde{\bu}}^{n+1} )  = ( A_h\phi^n\hat{\tilde{\bu}}^{n+1}, \nabla_h\Delta_h \tilde{\phi}^{n+1} ) - ( A_h\phi^n\hat{\tilde{\bu}}^{n+1}, \nabla_h ( (\Phi_N^{n+1})^3- (\phi^{n+1})^3 ) -\nabla_h \tilde{\phi}^{n} )  \nonumber \\
    & \quad\quad\quad \leq ( A_h\phi^n\hat{\tilde{\bu}}^{n+1}, \nabla_h\Delta_h \tilde{\phi}^{n+1} ) + \tilde{C}_1 
    \|\hat{\tilde{\bu}}^{n+1}\|_2^2 + \frac{\tilde{C}_1 ( \tilde{C}_4^2 +1)}{2} ( \|\nabla_h\tilde{\phi}^{n+1}\|_2^2 + \|\nabla_h\tilde{\phi}^{n}\|_2^2 ) .\label{refined nonlinear 3.10}
\end{align}
A substitution of \eqref{refined nonlinear 3.10} into \eqref{rough error proof 1.7} yields
\begin{align}
  & 
    \frac{1}{2\dt}( \|\tilde{\bu}^{n+1}\|_2^2-\|\tilde{\bu}^n\|_2^2 + \|\hat{\tilde{\bu}}^{n+1} - \tilde{\bu}^n \|_2^2  )+\frac{\dt}{2}(\|\nabla_h \tilde{p}^{n+1}\|_2^2 - \|\nabla_h \tilde{p}^n\|_2^2) \nonumber \\
   \le &  \gamma ( A_h\phi^n\hat{\tilde{\bu}}^{n+1}, \nabla_h\Delta_h \tilde{\phi}^{n+1} )   
   + \frac12 \|\tau_{\bu}^{n+1}\|_2^2 + \frac34 \|\tau^{n+1}_p\|_2^2 
    + \tilde{C}_2 \|\tilde{\bu}^n\|_2^2 
   + ( \tilde{C}_3 + \tilde{C}_1 \gamma) \| \hat{\tilde{\bu}}^{n+1}\|_2^2  \nonumber 
 \\
   & 
   + \frac{( \tilde{C}_1 ( \tilde{C}_4^2 +1)+ C^* C_0^2 ) \gamma}{2} 
    ( \|\nabla_h\tilde{\phi}^{n+1}\|_2^2 + \|\nabla_h\tilde{\phi}^{n}\|_2^2 ) 
  + \frac{\dt^2}{2} \|\nabla_h\tilde{p}^{n+1} \|_2^2 
  + \dt^2 \|\nabla_h\tilde{p}^n \|_2^2  . \label{refined nonlinear 3.11}
\end{align}
In turn,  a combination of \eqref{refined nonlinear 3.0}, \eqref{refined nonlinear 3.3}-\eqref{refined nonlinear 3.9} and \eqref{refined nonlinear 3.11} reveals that
\begin{align}
&\frac{1}{2\dt} ( \|\nabla_h \tilde{\phi}^{n+1}\|_2^2 -\|\nabla_h \tilde{\phi}^n\|_2^2 )+\frac{1}{2\dt\gamma}( \|\tilde{\bu}^{n+1}\|_2^2-\|\tilde{\bu}^n\|_2^2)+\frac{\dt}{2\gamma}(\|\nabla_h \tilde{p}^{n+1}\|_2^2 - \|\nabla_h \tilde{p}^n\|_2^2)  \nonumber \\
& +\frac{1}{2 \dt} ( G^{\Gamma, n+1} - G^{\Gamma,n} ) + \frac{3 \nu}{4} \| \nabla_h \hat{\tilde{\bu}}^{n+1}\|_2^2 
+ \frac{\varepsilon^2}{8}\| \nabla_h\Delta_h \tilde{\phi}^{n+1} \|_2^2 
\nonumber \\
\le &   \tilde{C}_7 (\|\nabla_h \tilde{\phi}^{n+1}\|_2^2+\|\nabla_h \tilde{\phi}^{n}\|_2^2 
+ G^{\Gamma, n+1} + G^{\Gamma,n} + \|\tilde{\bu}^n\|_2^2 ) 
  + \frac{\dt^2}{2 \gamma} \|\nabla_h\tilde{p}^{n+1} \|_2^2  
  + \frac{\dt^2}{\gamma} \|\nabla_h\tilde{p}^n \|_2^2 
  \nonumber \\
&    
 + \tilde{C}_8 ( \|\tau^{n+1}_\phi\|_2^2 + \|\tau^{n+1}_T\|_{2,\Gamma}^2 + \|\tau^{n+1}_B\|_{2,\Gamma}^2 
  + \|\tau_{\bu}^{n+1}\|_2^2 + \|\tau_{p}^{n+1}\|_2^2 + \| \nabla_h \tau_\mu^{n+1} \|_2^2 ),
\end{align}
where $\tilde{C}_7 = \max ( ( \tilde{C}_4^2 + 2 + 2 C_0^2 (C^*)^2 ) \varepsilon^{-2} + \frac{\tilde{C}_1 ( \tilde{C}_4^2 +1)+ C^* C_0^2}{2}, T, \tilde{C}_5^2 T, ( \tilde{C}_2 + 2 \tilde{C}_3 ) \gamma^{-1} + \tilde{C}_1 )$, $\tilde{C}_8 = \max ( \frac{3}{4 \gamma} , C_0^2 \varepsilon^{-2} , 2 \varepsilon^{-2})$. Notice that a quantity, $G^{\Gamma,n} := G^{T, n}+ G^{B, n}$, has been introduced in the derivation. Therefore, with sufficiently small $\dt$ and $h$, an application of discrete Gronwall inequality leads to the desired convergence estimate
\begin{equation}
    \|\nabla_h \tilde{\phi}^{n+1}\|_2 + (G^{\Gamma, n+1} )^\frac12 
    + \|\tilde{\bu}^{n+1}\|_2 + \dt \|\nabla_h \tilde{p}^{n+1}\|_2 \leq C(\dt + h^2),\label{refined nonlinear 3.12}
\end{equation}
based on the higher order truncation error accuracy. Moreover, by the preliminary estimate~\eqref{refined nonlinear 3.2}, we see that 
\begin{equation}
    \|\nabla_h \tilde{\phi}^{n+1}\|_2 + \| \tilde{\phi}^{T,n+1}\|_{2,\Gamma} + \| \tilde{\phi}^{B,n+1}\|_{2,\Gamma} 
    + \|\tilde{\bu}^{n+1}\|_2 + \dt \|\nabla_h \tilde{p}^{n+1}\|_2 \leq C(\dt + h^2),\label{refined nonlinear 3.12-2}
\end{equation}
This completes the refined error estimate. 

With the established convergence estimate \eqref{refined nonlinear 3.12-2} at hand, the a-priori assumption
in \eqref{a-priori assumption} could be recovered at the next time step $t^{n+1}$:
\begin{equation}
    \|\tilde{\bu}^{n+1}\|_2,\ \dt \|\nabla_h\tilde{p}^{n+1}\|_2,\  \| \nabla_h\tilde{\phi}^{n+1} \|_2,\ \| \tilde{\phi}^{T,n+1} \|_{2,\Gamma} \, \| \tilde{\phi}^{B,n+1} \|_{2,\Gamma}  \leq C(\dt+h^2) \leq \dt^{\frac{7}{8}}+h^{\frac{15}{8}}.
\end{equation}
provided that $\dt$ and $h$ are sufficiently small, in which a discrete Poincar{\'e} inequality has been used again. Therefore, an induction analysis could be effectively applied. This finishes the optimal rate convergence analysis and also completes the proof of Theorem \ref{main result theorem}. 

\begin{rem} 
In terms of the theoretical analysis of the CHNS system~\eqref{CHNS-DBC-1}-\eqref{CHNS-DBC-6} with a dynamical boundary condition, one key difficulty is associated with the boundary trace of the bulk profile, as well as its equivalence to the boundary profile. If a finite difference spatial approximation is taken, the corresponding argument becomes more straightforward, facilitated by the point-wise evaluations on the boundary grid points. This is the key reason why we choose such a spatial discretization in this work. If another spatial approach is taken, such as the mixed finite element method, the associated boundary trace arguments are expected to be more complicated. The corresponding convergence analysis with other spatial discretization will be considered in the future works. 
\end{rem} 

\begin{rem} 
The theoretical analysis reported in this article could be extended to the three-dimensional system, with the help of higher order symptomatic expansion in the consistency analysis, so that an application of inverse inequality would be able to control the discrete maximum norm of the numerical solution. The technical details will be reported in the future works. 
\end{rem} 

\begin{rem} 
A second order (in time) numerical scheme, based on a modified Crank-Nicolson temporal discretization, could be similarly designed for the system~\eqref{CHNS-DBC-1}-\eqref{CHNS-DBC-6}; also see the related works~\cite{chen24b, diegel17} for the CHNS system with the homogeneous boundary conditions. The convergence analysis and error estimate for the second order scheme is expected to follow a similar idea, and the technical details will be left in the future works.  
\end{rem} 

\begin{rem} 
For simplicity, we only consider the no-penetration, no-slip boundary condition for the velocity vector, so that the convection term does not appear on the boundary evolutionary equation. Meanwhile, the PDE system~\eqref{CHNS-DBC-1}-\eqref{CHNS-DBC-6} with either a free slip or Navier boundary condition is also of great physical importance~\cite{QianWangSheng03}, and the numerical design and theoretical analysis turns out to be more challenging. The associated total energy stability and convergence analysis will be considered in the future works. 
\end{rem} 

\begin{rem} 
For simplicity of presentation, the density variation in the fluid model has not been considered in this article.  Meanwhile, the variable density Cahn-Hilliard-Navier-Stokes system has also been studied in recent years, and the theoretical analysis with the homogeneous Neumann boundary condition has been reported as well. If the density variation is taken into consideration in the PDE system~\eqref{CHNS-DBC-1}-\eqref{CHNS-DBC-6}, the convergence analysis is also expected to go through, as long as the total energy dissipation law becomes valid. This challenging and scientifically important issue will be reported in the future works. 
\end{rem}

\section{Numerical results}  \label{sec:simulations} 
In this section, some numerical examples will be presented to validate the proposed scheme. The phase field equation is treated by a nonlinear full approximation scheme (FAS) multigrid method \cite{trottenberg01}, coupled with the standard Newton iteration on the 1-D boundary. The NS equation is solved using sparse matrix decomposition technique, with the chemical potential serving as a bridge to perform Picard iteration. The fast Fourier transform (FFT) is applied to the Poisson equation.
\subsection{Convergence test}
In this part we verify the accuracy of the numerical scheme \eqref{scheme-CHDBC-1}-\eqref{scheme-CHDBC-6} at the final time $T=0.1$. For simplicity, the computational domain is chosen as $ \Omega=(0,1)^2$. The surface diffusion parameter is given by $\varepsilon=0.1$; the kinematic viscosity $\nu$ and the surface tension parameter $\gamma$ are both set to $1$. The initial data of the phase variable and velocity field are given by
\begin{align}
\phi^0=\frac{1}{\pi}\sin(2\pi x)\cos(2\pi y), \quad  \mathbf{u}^0=\binom{-\cos (2 \pi x) \sin (2 \pi y)}{\sin (2 \pi x) \cos (2 \pi y)}. \label{initial phi velocity}
\end{align}
In terms of spatial accuracy, the time step size is fixed as $\dt=10^{-6}$, combined with decreasing mesh sizes. Since the exact solution is not explicit, we compute the Cauchy difference, defined as
\begin{equation}
	\delta_{\phi} := \phi_{h_c} - I_c^f (\phi_{h_f}),
\end{equation}
where $h_f$ denotes the numerical solution on the fine mesh and $h_c$ is the one on the coarse mesh.
Here we set $h_c=2h_f$. In fact, $I_c^f$ is a bilinear projection, mapping the solution from fine grid to
coarse grid. The $\ell^2$ and $\ell^\infty$ Cauchy differences and convergence rates of $\phi$ are presented in Table \ref{spatial accuracy}. It is evident that the accuracy orders gradually converge towards the theoretical value of 2.

In terms of temporal accuracy, we fix the spatial mesh size as $h=1/128$ and calculate the benchmark solution with $\dt=0.1/1024$. The numerical temporal errors and rates are displayed in Table \ref{temporal accuracy}, which verifies the theoretical convergence order of 1. 
\begin{table}[h!]
	\centering
	\begin{tabular}{llllll}
		\hline
	     Grid size	& 32-16      & 64-32       & 128-64     & 256-128     & 512-256     \\ \hline
		$\ell^2$ error     & 1.543E-03  & 3.880E-04  & 9.780E-05 & 2.494E-05  & 6.583E-06  \\
		$\ell^2$ rate      &            & 1.992      & 1.988     & 1.971      & 1.922  \\
		$\ell^\infty$ error & 2.967E-03 & 7.649E-04  & 1.927E-04 & 5.651E-05  & 1.577E-05  \\
		$\ell^\infty$ rate &            & 1.956      & 1.989     & 1.770      & 1.841 \\ \hline
	\end{tabular}\caption{Numerical spatial errors and orders.}
\label{spatial accuracy}
\end{table}
\begin{table}[h!]
	\centering
	\begin{tabular}{llllll}
		\hline
      Time step      & 0.1/32   & 0.1/64     & 0.1/128    & 0.1/256    & 0.1/512  \\ \hline
 $\ell^2$ error      & 7.502E-02& 4.070E-02  & 2.009E-02  & 8.847E-03  & 2.988E-03 \\
 $\ell^2$ rate       &          & 0.882      & 1.018      & 1.183      & 1.566  \\
 $\ell^\infty$ error & 1.448E-01& 7.924E-02  & 3.925E-02  & 1.732E-02  & 5.857E-03 \\
 $\ell^\infty$ rate  &          & 0.870      & 1.014      & 1.180      & 1.565 \\ \hline
	\end{tabular}\caption{Numerical temporal errors and orders.}
\label{temporal accuracy}
\end{table}

\subsection{Spinodal decomposition}
The spinodal decomposition begins with a constant initial data plus a random perturbation, combined with a zero velocity:
\begin{equation}
	\phi^0_{i,j}=0.2+0.02*r_{i,j}, \quad \bu^0=\boldsymbol{0} , \label{rand initial}
\end{equation}
where $r_{i,j}$ are uniformly distributed random numbers in $[-1,1]$. Figure \ref{fig:Spinodal mass and energy} illustrates the evolution of the mass difference, which is defined as $\overline{\phi^n}-\overline{\phi^0}$, and the total energy. Both the mass conservation and the energy dissipation are confirmed.
Figure \ref{fig:Spinodal} displays several snapshots of the scalar field $\phi$ at representative time instants. The computational domain is bounded by physical walls at its top and bottom, whereas periodic boundary conditions are applied along the $x$-direction. The snapshots demonstrate that the interface does not intersect with the wall perpendicularly, deviating from the behavior observed under homogeneous Neumann boundary conditions. A transition layer of characteristic thickness develops adjacent to the boundary, suggesting the presence of a short-range, wetting-type interaction. 
\begin{figure}[h!]
	\centering
	\includegraphics[width=0.35\linewidth]{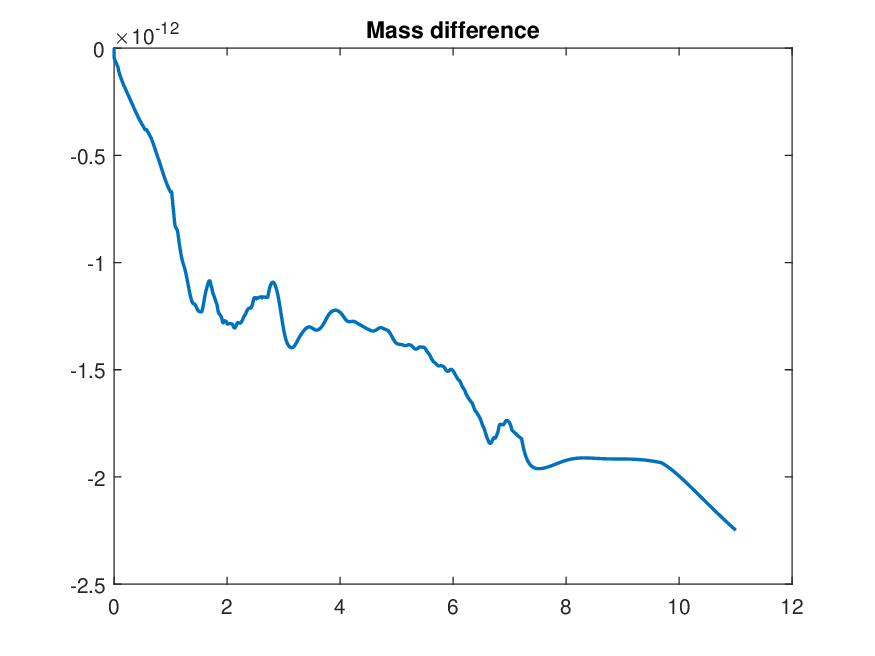}\hspace{-5mm}
	\includegraphics[width=0.35\linewidth]{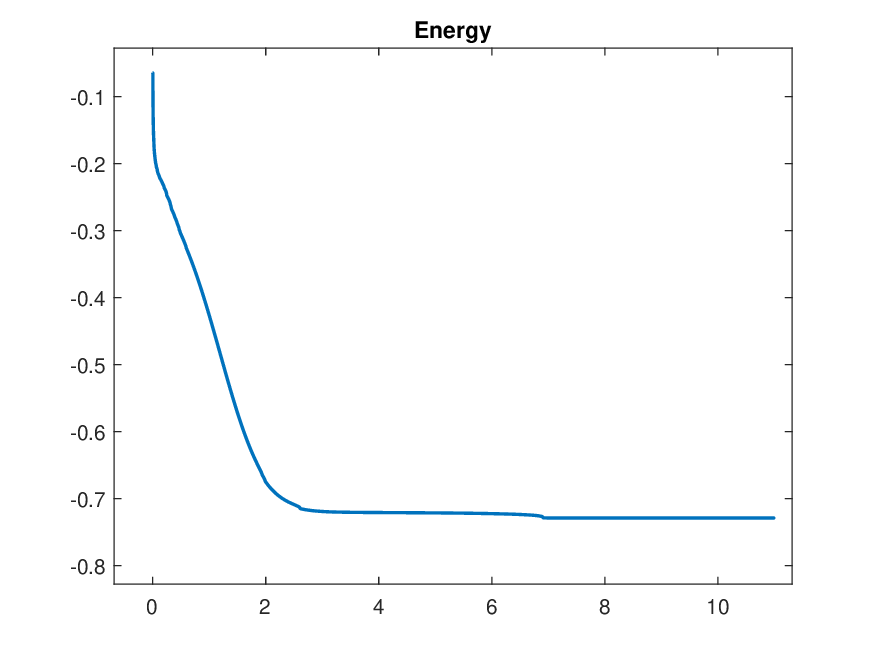}
	\caption{The mass difference and energy evolution with initial data \eqref{rand initial}.}
	\label{fig:Spinodal mass and energy}
\end{figure}
\begin{figure}[h!]
    \centering
    \includegraphics[width=0.33\linewidth]{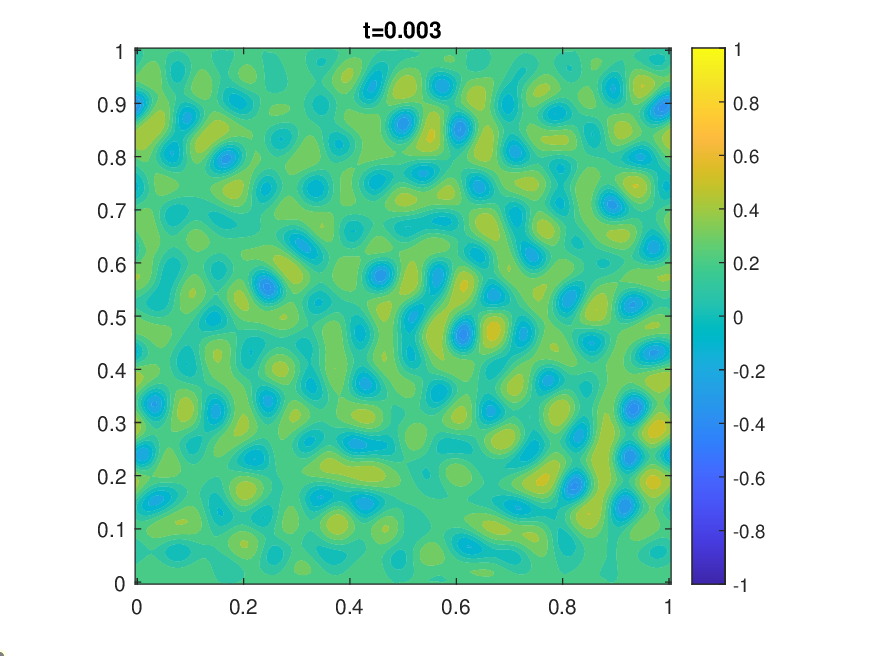}\hspace{-8mm}
    \includegraphics[width=0.33\linewidth]{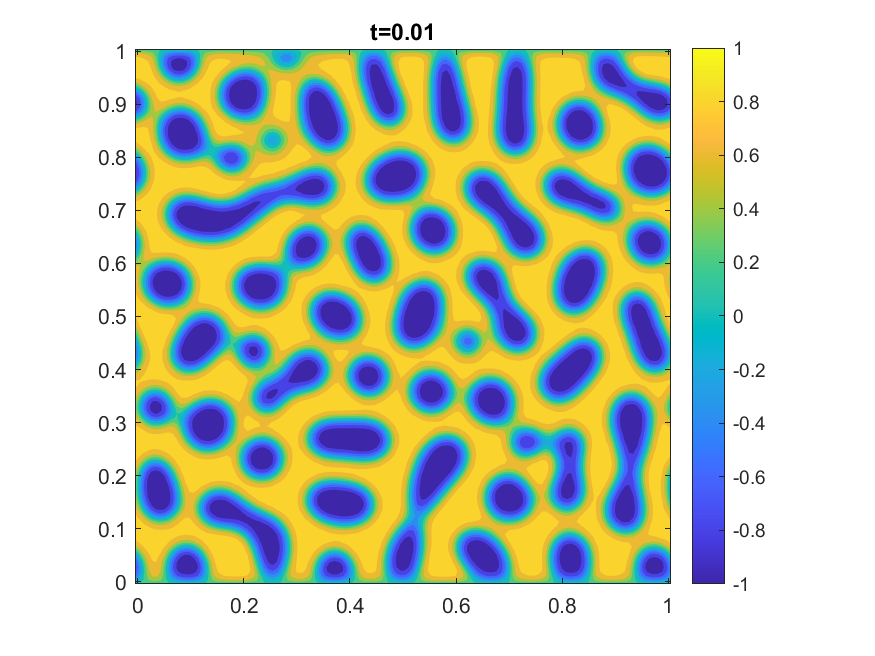}\hspace{-8mm}
    \includegraphics[width=0.33\linewidth]{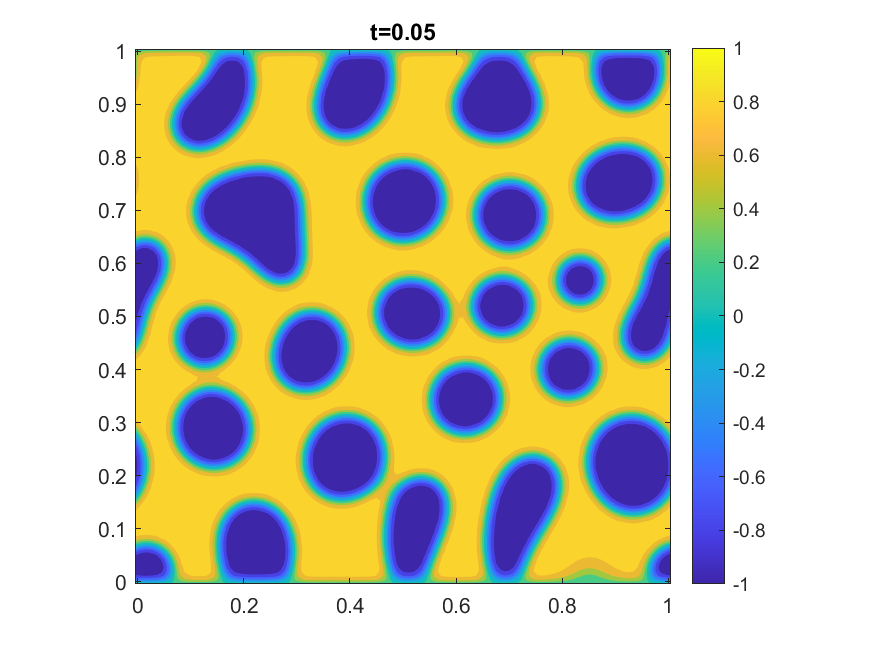}

    \includegraphics[width=0.33\linewidth]{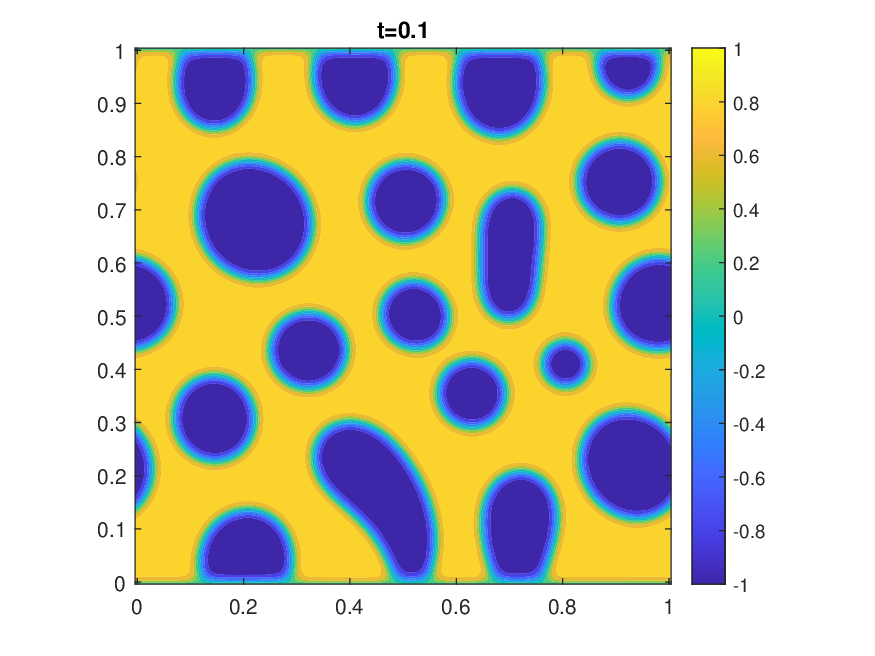}\hspace{-8mm}
    \includegraphics[width=0.33\linewidth]{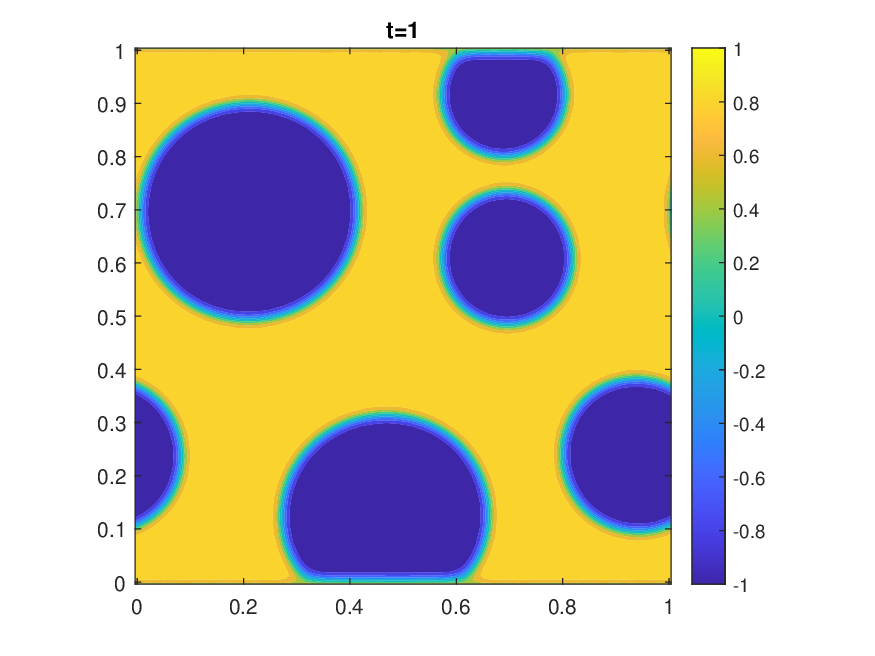}\hspace{-8mm}
    \includegraphics[width=0.33\linewidth]{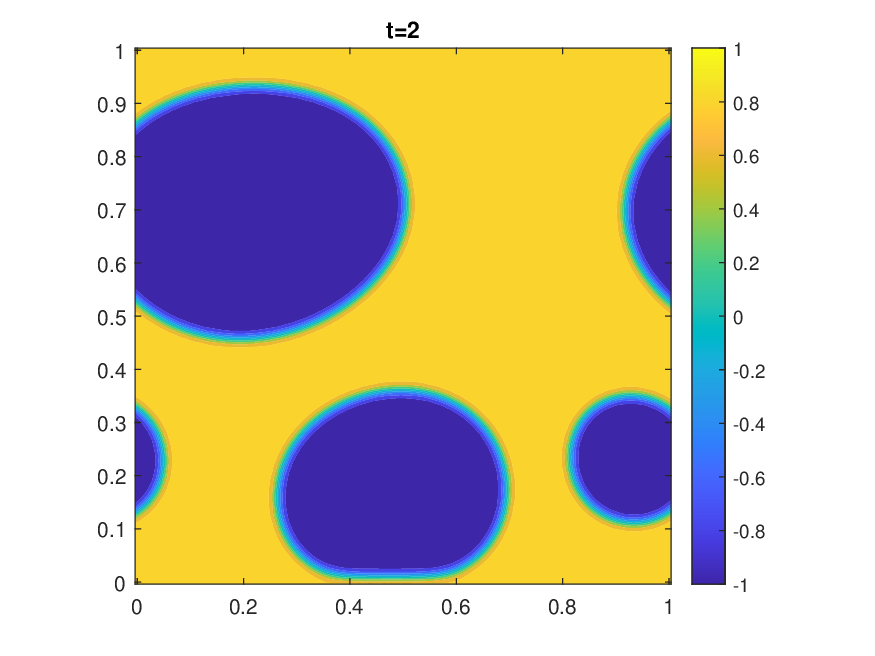}

    \includegraphics[width=0.33\linewidth]{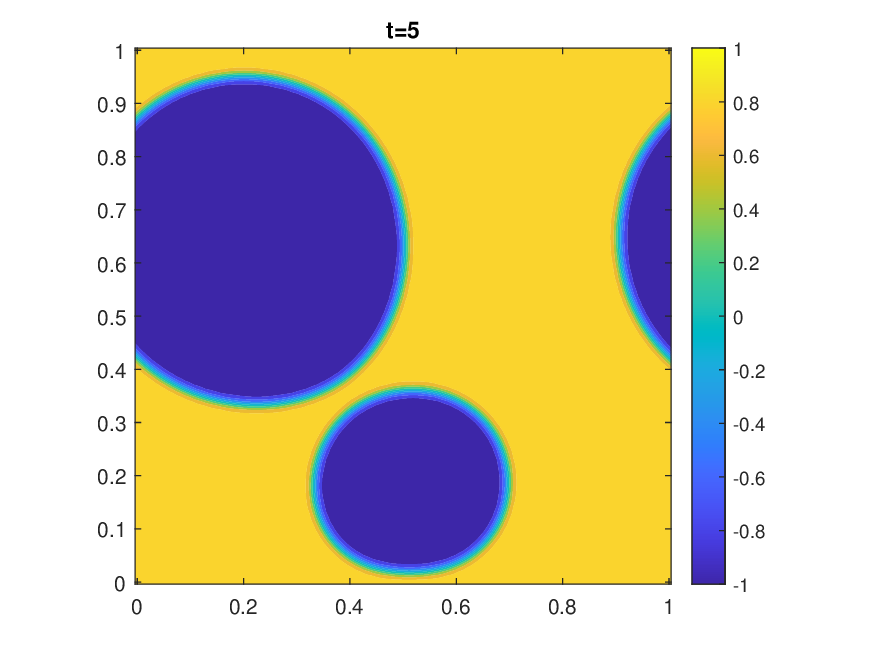}\hspace{-8mm}
    \includegraphics[width=0.33\linewidth]{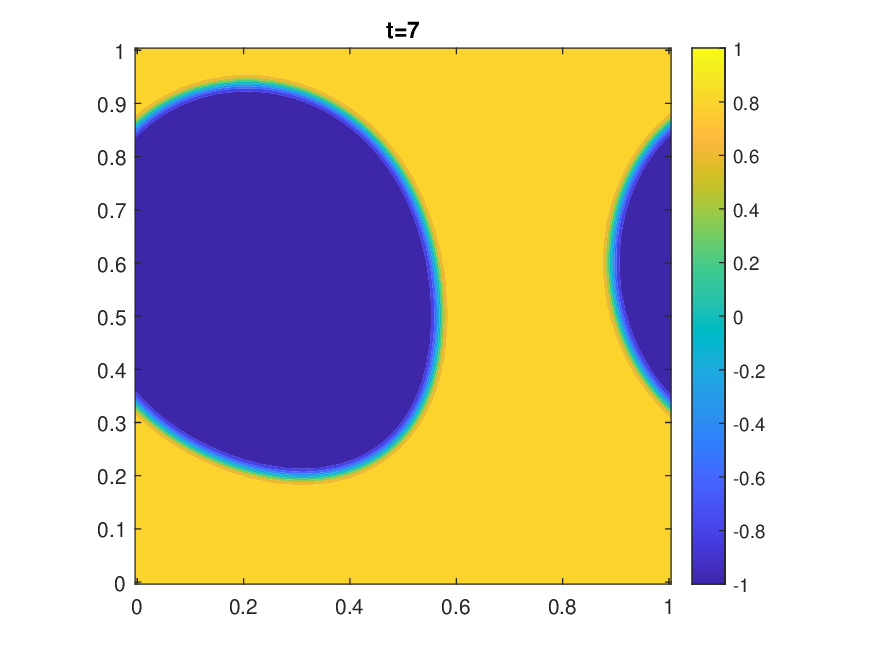}\hspace{-8mm}
    \includegraphics[width=0.33\linewidth]{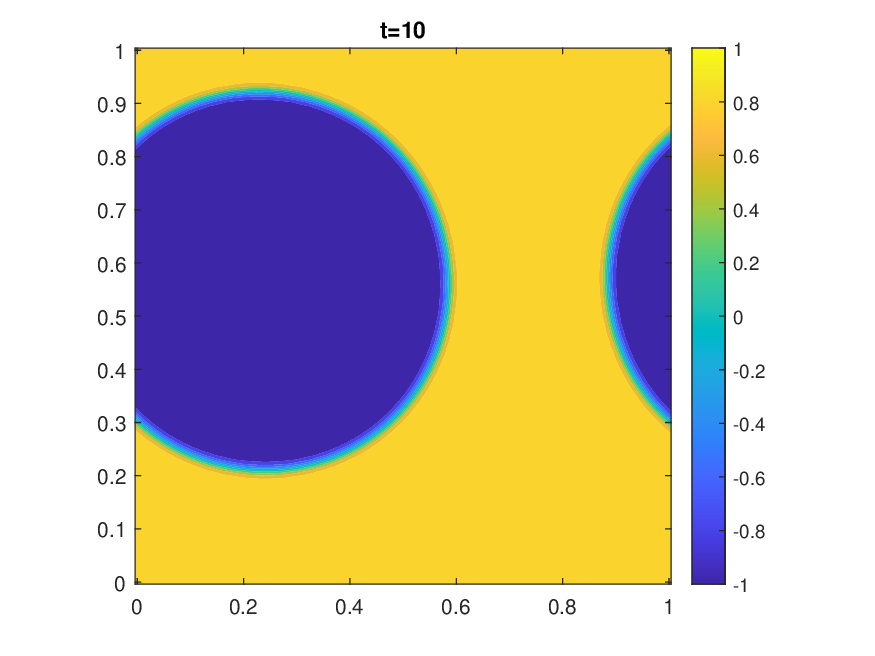}
    \caption{The phase evolution of $\phi$ at several time instants with initial data \eqref{rand initial}.}
    \label{fig:Spinodal}
\end{figure}

\subsection{Moving contact line}
In this part, we simulate the moving contact line evolution with difference static contact angles. Considering different time scale, the surface energy is defined as follows \cite{wang2008moving, Xu_Di_Yu_2018, Xu_Wang_analysis}
\begin{align}
	E_{\rm surf} (\psi) = \int_\Gamma \, \frac{\sqrt{2}\varepsilon}{3} \cos(\theta_s) \sin(\frac{\pi}{2}\psi)  \,  \dS ,  \label{Moving contact line}
\end{align}
where $\theta_s$ is the static contact angle for the  $\phi=-1$ phase, and $\varepsilon$ is the bulk interface thickness, related to time scale. 
The initial data is set as a semicircular droplet with its center located at the dynamic boundary $(0.5, 0)$, which indicates that the interface is perpendicular to the physical boundary. The droplet takes a value of $1$ in the interior and $-1$ in the exterior. The surface diffusion parameter is set as $\varepsilon = 0.01$. The static contact angle is taken to be $\cos(\theta_s)=\pm\frac{1}{2}$ respectively, to simulate both wetting and non-wetting phenomena.

Figure \ref{fig:120} illustrates the temporal evolution with $\cos(\theta_s)=-\frac{1}{2}$. It can be observed that the droplet slowly extends outward, and the contact angle gradually changes from $\frac{\pi}{2}$ to an obtuse angle, which is consistent with the theoretical model.
\begin{figure}[h!]
	\centering
	\includegraphics[width=0.33\linewidth]{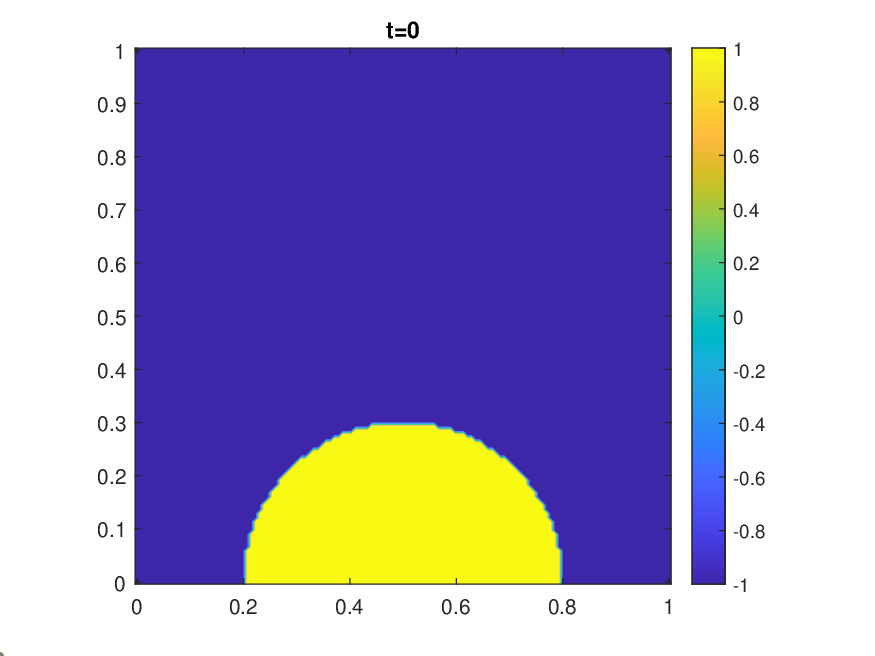}\hspace{-8mm}
	\includegraphics[width=0.33\linewidth]{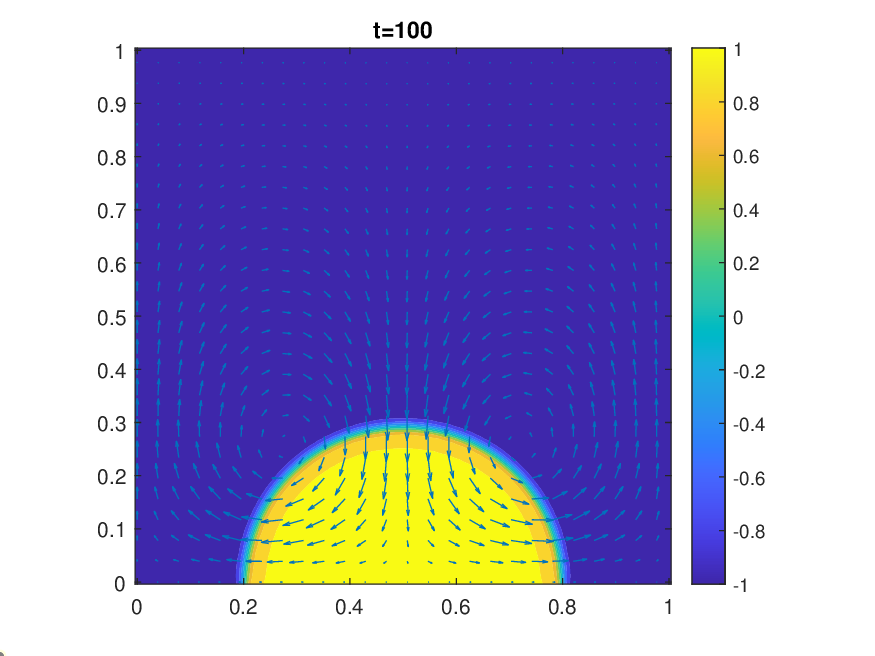}\hspace{-8mm}
	\includegraphics[width=0.33\linewidth]{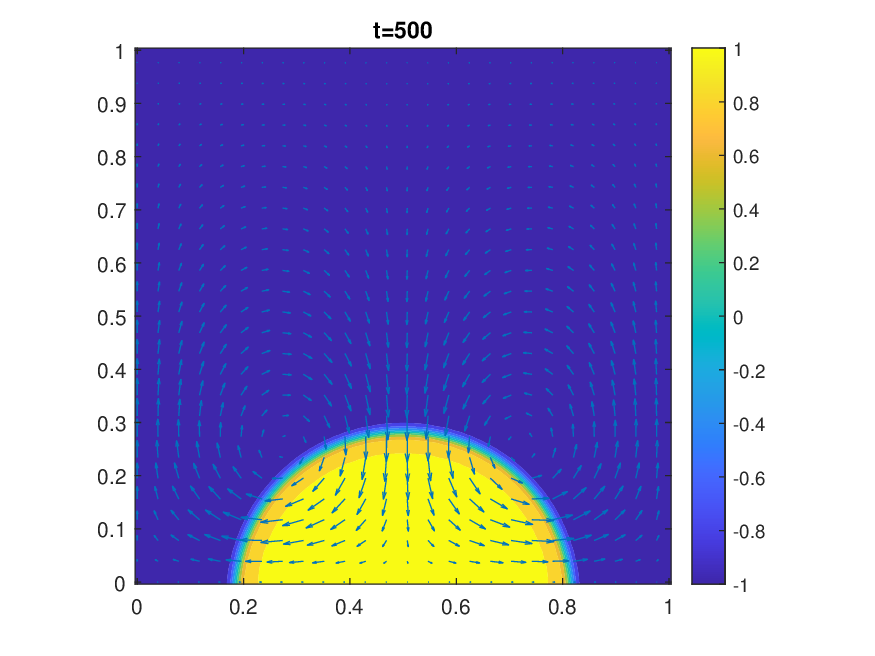}
	
	\includegraphics[width=0.33\linewidth]{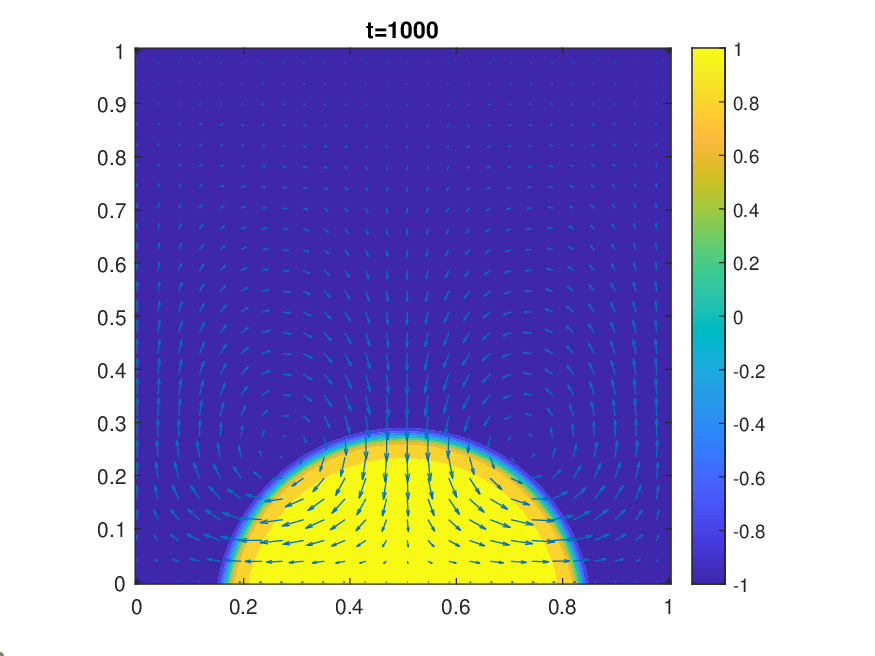}\hspace{-8mm}
	\includegraphics[width=0.33\linewidth]{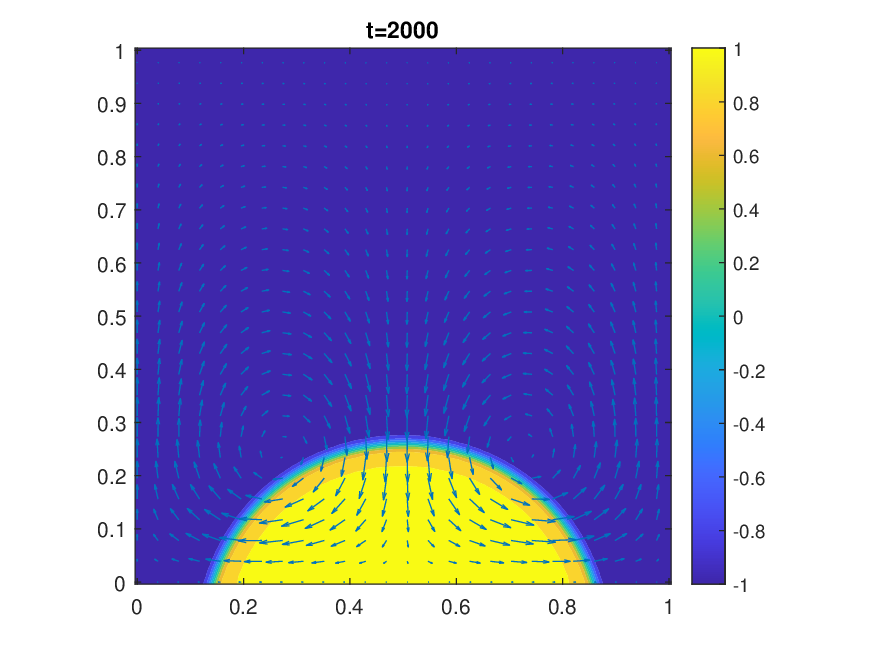}\hspace{-8mm}
	\includegraphics[width=0.33\linewidth]{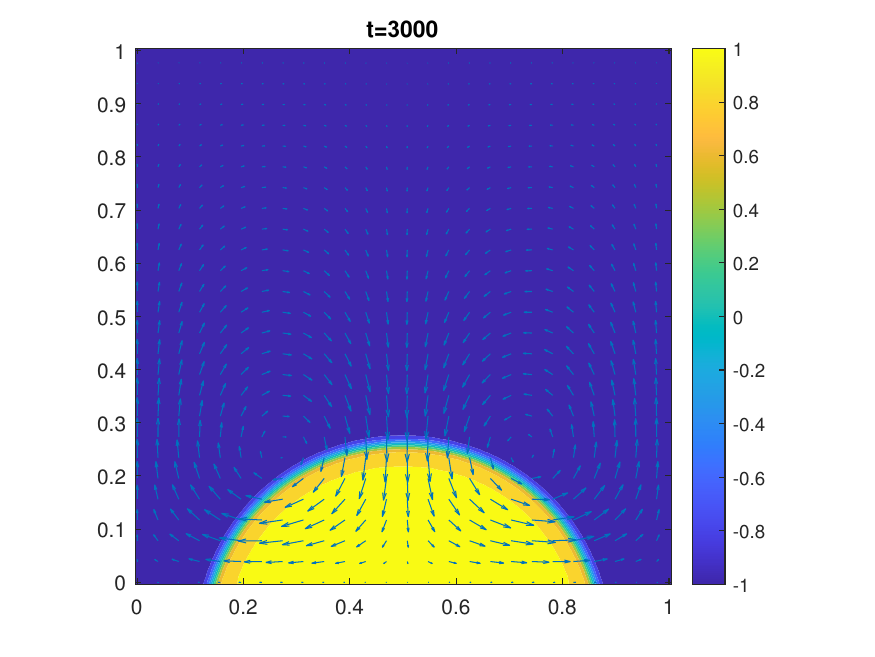}
	\caption{The phase evolution of $\phi$ with $\cos(\theta_s)=-\frac{1}{2}$.}
	\label{fig:120}
\end{figure}

Figure \ref{fig:60} displays the evolution of phase field with $\cos(\theta_s)=\frac{1}{2}$. Both scenarios require a very long time evolution to reach equilibrium. In combination with the previous example, the evolution can be divided into two stages. Initially, the bulk interface begins to diffuse and makes the primary contribution to energy dissipation, during which the total energy decreases rapidly while the evolution of the physical boundary is less pronounced. Subsequently, as the interior approaches a steady state, the contact angle begins to change. Owing to the absence of a diffusion term in the boundary energy, this process is slower and eventually reaches an equilibrium state. In addition, Figure \ref{fig:energy mass} plots the energy and mass curves for both two examples, confirming the properties of energy dissipation and mass conservation.
\begin{figure}[h!]
	\centering
	\includegraphics[width=0.33\linewidth]{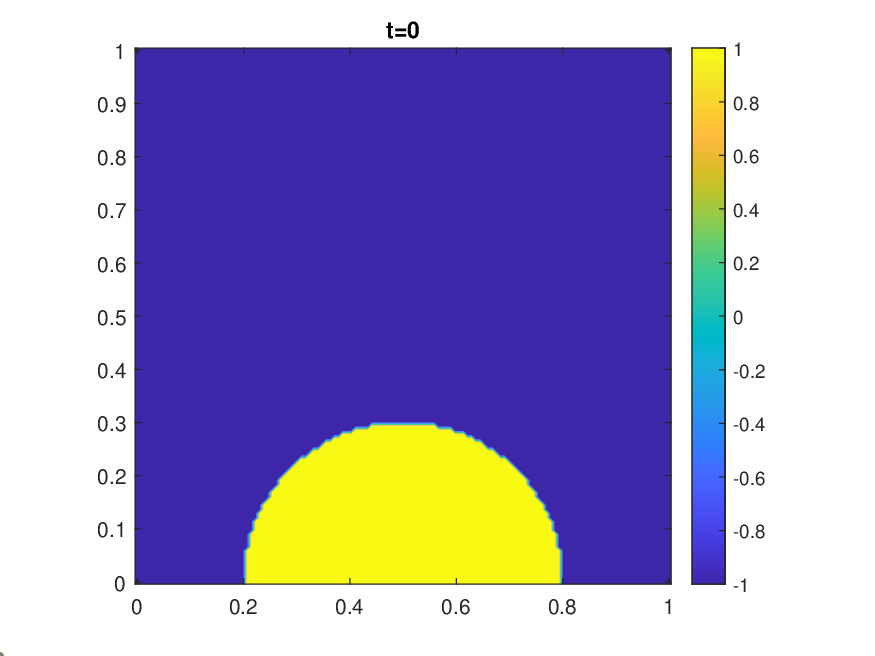}\hspace{-8mm}
	\includegraphics[width=0.33\linewidth]{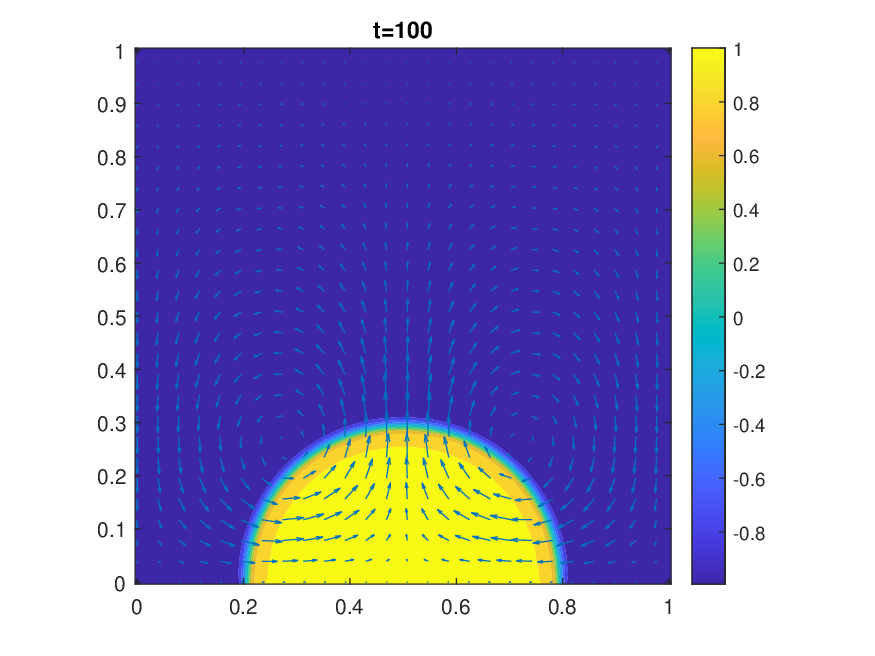}\hspace{-8mm}
	\includegraphics[width=0.33\linewidth]{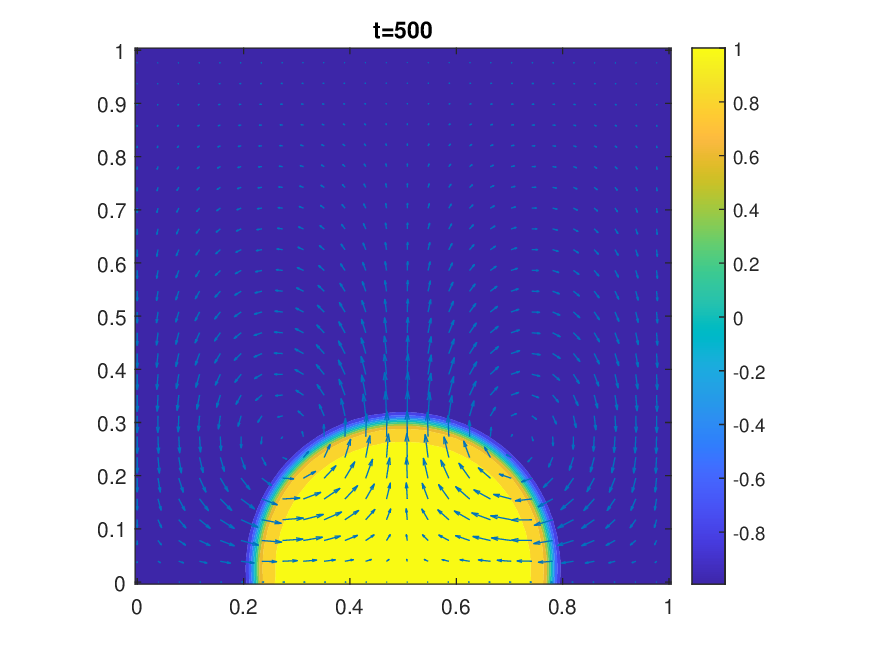}
	
	\includegraphics[width=0.33\linewidth]{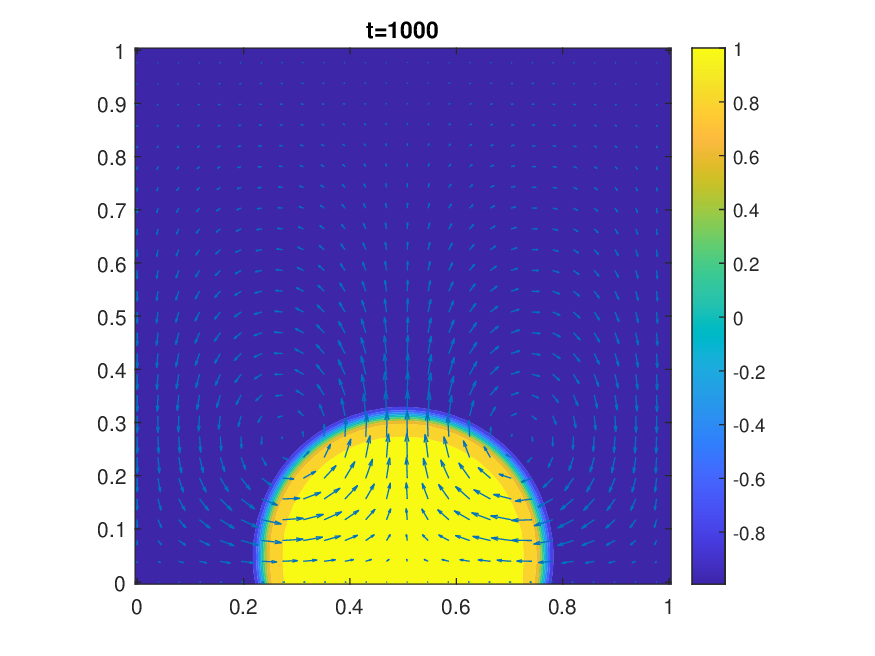}\hspace{-8mm}
	\includegraphics[width=0.33\linewidth]{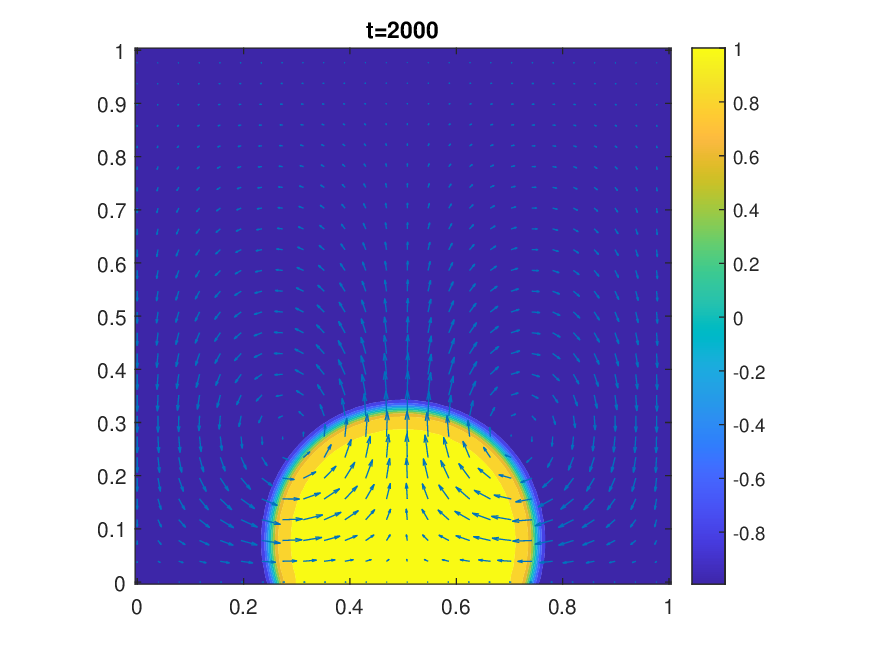}\hspace{-8mm}
	\includegraphics[width=0.33\linewidth]{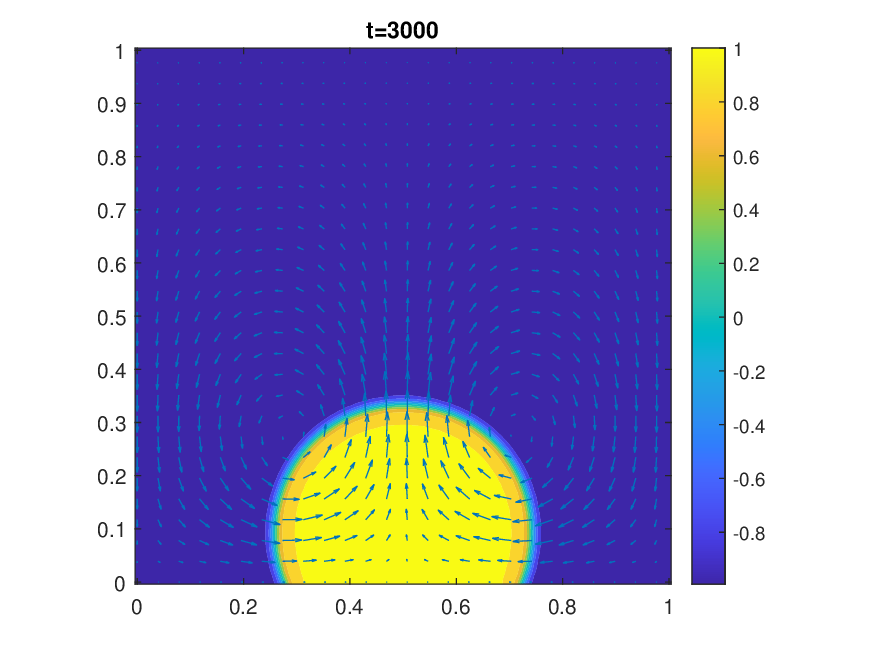}
	\caption{The phase evolution of $\phi$ with $\cos(\theta_s)=\frac{1}{2}$.}
	\label{fig:60}
\end{figure}
\begin{figure}[h!]
	\centering
	\includegraphics[width=0.3\linewidth]{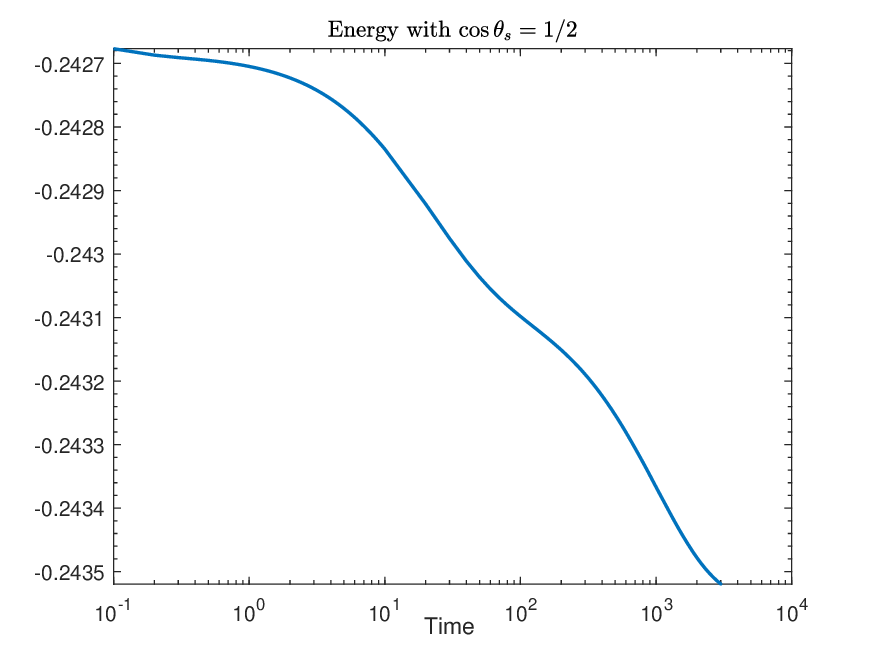}
	\includegraphics[width=0.3\linewidth]{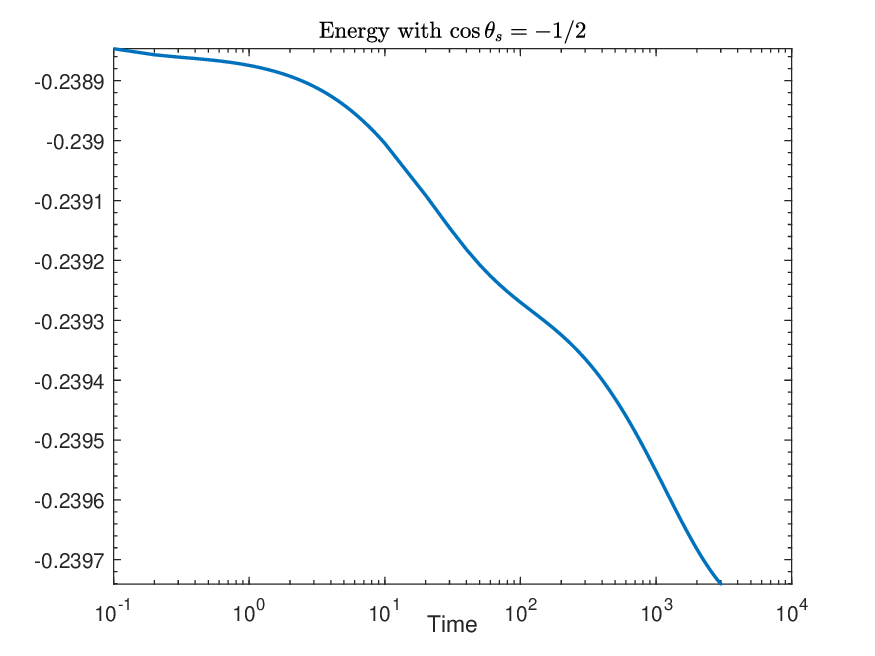}
	\includegraphics[width=0.3\linewidth]{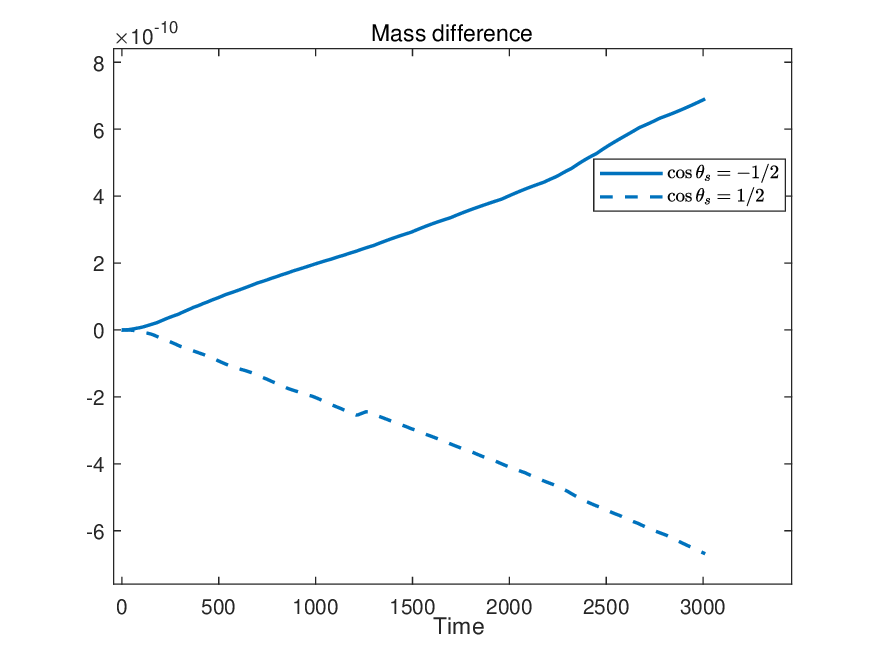}
	\caption{The evolution of energies and mass changes of both two cases.}
	\label{fig:energy mass}
\end{figure}

\section{Conclusions}  \label{sec:conclusion} 
In this paper, we have presented a fully discrete finite difference numerical scheme of the Cahn-Hilliard-Navier-Stokes equation with a relaxed dynamical boundary condition, with potential applications to the moving contact line problem. 
The classic polynomial potential is included in both bulk and boundary energies, and a dynamical evolution equation for the boundary profile corresponds to a lower-dimensional phase field coupled by the normal derivative. In the numerical scheme, a convex splitting technique is applied to treat the chemical potential,  both at the interior area and on the boundary section. The resulting numerical system could be represented as a monotone mapping, and the Browder–Minty lemma ensures its unique solvability, i.e., the existence and uniqueness of the numerical solution. The total energy stability analysis could be theoretically justified based on the convex-concave decomposition. An optimal rate convergence analysis and error estimate is established at a theoretical level, with the help of higher order consistency analysis for the boundary extrapolation, combined with rough and refined error estimates. Due to the absence of a diffusion term in the boundary equation, the convergence analysis relies on a rearrangement of summations, along with the introduction of an auxiliary mass function. Some numerical experiments have also been presented, which demonstrate the theoretical properties of the proposed scheme.

\section*{Acknowledgments}
C. Wang is partially supported by the National Science Foundation, DMS-2309548 and Z.R. Zhang is partially supported by the NSFC, PR China No. 11871105 and 12231003. Y.Z. Guo thanks the Hong Kong Polytechnic University for the generous support.
\bibliographystyle{plain}
\bibliography{draft_ref}



\end{document}